\documentclass[12pt]{article}
\pdfoutput=1
\usepackage{hyperref}
\usepackage{times}

\usepackage{amsmath}
\usepackage{amsfonts}
\usepackage{amssymb}
\usepackage{amsthm}
\usepackage{layout}
\usepackage[spanish]{babel}
\usepackage[dvips]{color}
\usepackage[latin1]{inputenc}

\advance\textwidth by 3cm

\advance\textheight by 3.5cm

\theoremstyle{plain}
\newtheorem{teorema}{Teorema}[section]
\newtheorem{coro}[teorema]{Corolario}
\newtheorem{lema}[teorema]{Lema}
\newtheorem{prop}[teorema]{Proposición}
\theoremstyle{definition}
\newtheorem{obse}[teorema]{Observación}
\newtheorem{ejem}[teorema]{Ejemplo}
\newtheorem{defi}[teorema]{Definición}

\numberwithin{equation}{subsection}

\newcommand{\E}{\mathcal{E}}
\newcommand{\V}{\mathcal{V}}
\newcommand{\R}{\mathrm{I\!R}}
\newcommand{\ba}{\mathbf{B}}
\newcommand{\F}{\mathcal{F}}
\newcommand{\A}{\mathcal{A}}
\newcommand{\C}{\mathcal{C}}
\newcommand{\M}{\mathcal{M}}

\newcommand{\B}{\mathcal{B}}
\newcommand{\X}{X_s^{t,x}}

\newcommand{\Z}{Z_s^{t,x}}

\begin{document}
\begin{center}

\textbf{\huge{Cálculo Estocástico, EDEs y EDPs}}
\end{center}

\begin{center}
\large{\textsc{Notas de Clase}}

\vspace{2cm}

\large\textbf{Rafael Serrano}
\end{center}

\begin{center}
\large{\textsc{Universidad del Rosario}\\Calle 12C No. 4-69\\
Bogot\'a, Colombia}

\end{center}

\tableofcontents

\section*{Notación}
\addcontentsline{toc}{section}{Notación}

\noindent $a\wedge b:$ el mínimo entre $a$ y $b$

\noindent $a\vee b:$ el máximo entre $a$ y $b$

\noindent $A^c:$ complemento de A

\noindent $\partial A:$ frontera de A

\noindent $\bar{A}:=A\cup\partial A:$ clausura de A

\medskip\noindent $\mathbf{1}_A(x)=\begin{cases}
1 \ &\text{si} \ x\in A\\
0 \ &\text{si} \ x\in A^c.
\end{cases}$

\medskip\noindent Si $x\in\R^d: \ |x|^2:=\sum_{i=1}^d|x_i|^2$

\smallskip\noindent $B(x,r):=\{y\in\R^d:|y-x|<r\}$

\smallskip\noindent $B[x,r]:=\{y\in\R^d:|y-x|\le r\}$

\smallskip\noindent $\langle\cdot,\cdot\rangle:$ producto escalar

\noindent $e_i:=(\delta_{ij})_{1\le j\le d}, \ 1\le i\le d,$ donde
$\delta_{ij}$ es el símbolo de Kronecker.

\medskip\noindent Si $z\in\R^{d\times m}: \ z^*=$ transpuesta de
$z$

\noindent $Tr(z)=$ traza de $z$

\noindent $||z||^2:=Tr(zz^*)=\sum_{i=1}^d\sum_{j=1}^m|z_{ij}|^2$

\medskip\noindent Para $E\subseteq\R^d$ abierto
notaremos:

\smallskip\noindent $\C(E)=\{f:E\to\R\mid f$ es continua\}

\noindent $\C(\bar{E})=\{f\in\C(E):f$ es uniformemente continua
sobre subconjuntos acotados de $E\}$

\noindent Así, si $f\in\C(\bar{E})$ entonces $f$ se extiende
continuamente sobre $\bar{E}.$

\smallskip\noindent $\C_b(E)=\{f\in\C(E):f$ es acotada\}

\noindent $\C^k(E)=\{f:E\to\R\mid f$ es $k-$veces continuamente
diferenciable\}

\smallskip\noindent $\C^k(\bar{E})=\{f\in\C^k(E):$ las derivadas
parciales de $f$ hasta el orden $k$ tienen extensiones continuas
sobre $\bar{E}\}$

\smallskip\noindent $\C_b^k(E)=\{f\in\C^k(E):f$ y sus derivadas
parciales hasta orden $k$ son acotadas\}

\noindent $\C^\infty(E)=\bigcap_{k=1}^\infty\C^k(E)$

\medskip\noindent $\C\bigl([0,T]\times E\bigr)=\{f:[0,T]\times E\to\R\mid f$
es continua\}

\noindent $\C^{1,2}((0,T)\times E)=\{f\in\C([0,T]\times
E):\frac{\partial f}{\partial t}, \ \frac{\partial f}{\partial
x_i}, \ \frac{\partial^2 f}{\partial x_i\partial x_j}$ son
continuas sobre $(0,T)\times E\}$

\smallskip\noindent $\C^{1,2}\bigl([0,T)\times E\bigr)=\{f\in\C^{1,2}((0,T)\times
E):\frac{\partial f}{\partial t}, \ \frac{\partial f}{\partial
x_i}, \ \frac{\partial^2 f}{\partial x_i\partial x_j}$ tienen
extensiones continuas sobre $[0,T)\times E\}$

\medskip Si $f\in\C^1(E),$ el gradiente de $f$ con respecto
a $x$ se define por \[\nabla_x f:=\Bigl(\frac{\partial f}{\partial
x_1 },\cdots,\frac{\partial f}{\partial x_d}\Bigr).\]Para las
derivadas de orden mayor usaremos la notación
\[D_x^\alpha f=\frac{\partial^{|\alpha|}f}{\partial x_1^{\alpha_1}\cdots\partial
x_d^{\alpha_d}}\] donde $\alpha=(\alpha_1,\ldots,\alpha_d), \
\alpha_i=0,1,2,\ldots$ y $|\alpha|=\alpha_1+\cdots+\alpha_d.$

\medskip Dados dos espacios medibles $(E^1,\E^1), (E^2,\E^2),$ y una función
$f:E^1\to E^2,$ diremos que $f$ es $\E^1/\E^2-$medible si para
todo $A\in\E^2$ se tiene $f^{-1}(A)\in\E^1.$

\medskip Para $x\in\R^d$ fijo, se define la medida de Dirac
concentrada en $x,$ denotada por $\epsilon_x,$ como
$\epsilon_x(A):=\mathbf{1}_A(x).$

\bigskip\noindent $f_n\xrightarrow[n\to\infty]{c.s}f$ \
convergencia \textit{casi siempre}

\medskip\noindent $f_n\xrightarrow[n\to\infty]{\mathbf{P}}f$ \
convergencia en probabilidad

\newpage\section*{Introducción}
\addcontentsline{toc}{section}{Introducción}

Este documento tiene como propósito recopilar y presentar de la manera más auto-contenida posible resultados esenciales del cálculo estocástico y de la teoría de ecuaciones diferenciales estocásticas y su conexión con problemas de valor final de ecuaciones diferenciales parciales lineales de segundo orden.

El material aquí incluido ha sido tomado en gran parte de los libros \cite{friedman2}, \cite{karatzas}, \cite{tudor} y \cite{blanco}. Para la lectura de estas notas se requieren buenos conocimientos de teoría de la probabilidad y ecuaciones diferenciales. Conocimientos básicos de teoría de la medida y procesos estocásticos son recomendables pero no obligatorios.

\section{Cálculo estocástico}
\subsection{Procesos estocásticos}
A lo largo de estas notas $(\Omega,\F,\mathbf{P})$ será un espacio
de probabilidad fijo. Diremos que un evento $A\in\F$ ocurre
\textit{casi siempre} (abreviado \textit{c.s.}) si
$\mathbf{P}(A)=1.$
\begin{defi}
Un \textit{proceso estocástico con parámetro de tiempo continuo}
es una familia $X=(X_t)_{t\geq 0}$ de variables aleatorias
definidas sobre $(\Omega,\F,\mathbf{P})$ y con valores en un
espacio medible $(E,\mathcal{E})$ llamado \textit{espacio de
estados.}
\end{defi}
Para cada $\omega\in\Omega$ fijo, la función $t\mapsto
X_t(\omega)$ es llamada la \textit{trayectoria} o {\it
realización} del proceso $X$ asociada a $\omega.$ En ocasiones
restringiremos el parámetro de tiempo $t$ a un intervalo contenido
en $[0,\infty).$

\begin{defi}
Sean $X=(X_t)_{t\geq 0}, \ Y=(Y_t)_{t\geq 0}$ procesos
estocásticos definidos sobre el mismo espacio de probabilidad
$(\Omega,\F,\mathbf{P})$ y con valores en $(E,\E).$ Diremos que
\begin{description}
  \item[(a)] $X$ es una \textit{modificación} o {\it versión } de $Y,$ o que $X$ y $Y$ son \textit{equivalentes,}
  si para todo $t\geq 0$ se tiene que $\mathbf{P}(X_t=Y_t)=1,$
  \item[(b)] $X$ y $Y$ son {\it indistinguibles} si $\mathbf{P}(X_t=Y_t,\ \forall t\geq
  0)=1.$
\end{description}
\end{defi}
Claramente la segunda propiedad implica la primera. Sin embargo,
dos procesos pueden ser equivalentes pero tener trayectorias
completamente distintas. El siguiente es un típico ejemplo:
\begin{ejem}
Considere una variable aleatoria $\tau$ real positiva con
distribución continua, y sean $X_t\equiv 0$ y
\[Y_t(\omega):=
\begin{cases}
  0, & \ \text{si} \ t\neq\tau(\omega),\\
  1, & \ \text{si} \ t=\tau(\omega),
\end{cases}\]
para $t\geq 0.$ Entonces $Y$ es una modificación de $X,$ pues
\[\mathbf{P}(Y_t=X_t)=\mathbf{P}(\tau\neq t)=1, \ \ \ \forall t\geq 0.\]
Sin embargo, $\mathbf{P}(X_t=Y_t, \ \forall t\geq 0)=0.$
\end{ejem}

Para nuestro propósito, la mayoría de las veces el espacio de
estados será el espacio Euclídeo $d-$dimensional equipado con su
$\sigma-$álgebra de \textit{Borel} e.d. $E=\R^d, \ \E=\B(\R^d),$
donde $\B(U)$ denota la $\sigma-$álgebra generada por los
conjuntos abiertos de un espacio topológico $U.$

\begin{defi}
Diremos que un proceso $X=(X_t)_{t\geq 0}$ con valores en $\R^d$
es \textit{continuo c.s.} (resp. \textit{continuo por la derecha
c.s.}, resp. \textit{continuo por la izquierda c.s.}) si para casi
todo $\omega\in\Omega$ la aplicación $t\mapsto X_t(\omega)$ es
continua (resp. continua por la derecha, resp. continua por la
izquierda).
\end{defi}

\begin{teorema}
Sean $X=(X_t)_{t\geq 0}$ y $Y=(Y_t)_{t\geq 0}$ dos procesos con
valores en $\R^d,$ equivalentes y continuos por la derecha c.s.
(resp. continuos por la izquierda c.s.), entonces $X$ y $Y$ son
indistinguibles.
\end{teorema}
\begin{proof}
Por ser equivalentes, $\mathbf{P}(X_r\neq Y_r)=0$ para todo
$r\in\mathbb{Q}_+=\mathbb{Q}\cap[0,\infty).$ Sea
\[G=\bigcup_{r\in\mathbb{Q}_+}\{X_r\neq Y_r\},\]
entonces $\mathbf{P}(G)=0.$ Si $\omega\notin G$ entonces
$X_t(\omega)=Y_t(\omega), \ \forall t\in\mathbb{Q}_+,$ y por la
continuidad a derecha (resp. continuidad a izquierda)
$X_t(\omega)=Y_t(\omega), \ \forall t\geq 0,$ es decir,
\[\{X_t\neq Y_t\}\subseteq G, \ \ \ \forall t\geq 0,\]
luego $\bigcup_{t\geq 0}\{X_t\neq Y_t\}\subseteq G,$ lo que
implica $\mathbf{P}(X_t=Y_t, \ \forall t\geq 0)=1.$
\end{proof}

Cuando decimos que un proceso estocástico $X=(X_t)_{t\geq 0}$ es
una familia de variables aleatorias con espacio de estados común
$(E,\E),$ implícitamente estamos diciendo que cada $X_t$ es
$\F/\E-$medible. Sin embargo, $X$ es realmente una función de dos
variables $(t,\omega),$ y en ocasiones es conveniente tener
medibilidad con respecto a ambas variables:

\begin{defi}
Sea $X=(X_t)_{t\geq 0}$ un proceso estocástico definido sobre
$(\Omega,\F,\mathbf{P})$ y con espacio de estados $(E,\E).$
Diremos que $X$ es un proceso \textit{medible} si la aplicación
\begin{align*}
  \bigl([0,\infty)\times\Omega,\B\bigl([0,\infty)\bigr)\otimes\F\bigr)&\longrightarrow(E,\E))\\
  (t,\omega)&\longmapsto X_t(\omega)
\end{align*}
es $\B\bigl([0,\infty)\bigr)\otimes\F/\E-$medible.
\end{defi}

\begin{defi} Una filtración $\{\F_t\}_{t\geq 0}$ en $(\Omega,\F)$
es una familia de sub-$\sigma$-álgebras, contenidas en $\F,$ tales
que $\F_t\subseteq\F_s$ si $t<s.$ Al sistema
$(\Omega,\F,\{\F_t\}_{t\geq 0},\mathbf{P})$ se le llama
\textit{espacio de probabilidad filtrado.}
\end{defi}

\begin{ejem}
Para un proceso estocástico $X=(X_t)_{t\geq 0}$ fijo, la familia
de sub-$\sigma-$álgebras $\{\F_t^X\}_{t\geq 0}$ definida por
\[\F_t^X:=\sigma(X_s:0\le s\le t), \ \ \ t\geq 0\]
es una filtración en $(\Omega,\F)$ y se denomina la
\textit{filtración canónica} asociada al proceso $X.$
\end{ejem}

\noindent\textbf{Notación.} Para una filtración $\{\F_t\}_{t\geq
0}$ fija notaremos
\[\F_\infty:=\sigma\left(\textstyle\bigcup\nolimits_{\, t\geq 0}\F_t\right), \ \ \ \
\F_{t+}:=\textstyle\bigcap\nolimits_{s>t}\F_s.\]

\begin{defi}
(a) Una filtración $\{\F_t\}_{t\geq 0}$ se dice \textit{continua a
derecha} si $\F_t=\F_{t+}$ para cada $t\geq 0.$

{}{} (b) Diremos que una filtración $\{\F_t\}_{t\geq 0}$
\textit{satisface las condiciones usuales} si es continua a
derecha y $\F_0$ contiene todos los conjuntos $\mathbf{P}-$nulos
de $\Omega.$
\end{defi}

\begin{defi}\label{defiadaptado}
Sea $(\Omega,\F,\{\F_t\}_{t\geq 0},\mathbf{P})$ un espacio de
probabilidad filtrado y sea $X=(X_t)_{t\geq 0}$ un proceso
estocástico definido sobre $(\Omega,\F,\mathbf{P})$ y con espacio
de estados $(E,\E).$ Diremos que
\begin{description}
  \item[(i)] $X$ es \textit{adaptado} a $\{\F_t\}_{t\geq 0}$ si para cada $t\geq 0,$ $X_t$ es $\F_t/\E-$medible,
  es decir, $\F_t^X\subseteq\F_t$ para cada $t\geq 0.$
  \item[(ii)] $X$ es \textit{progresivamente medible} con respecto
  a $\{\F_t\}_{t\geq 0}$ si para cada $t\geq 0$ la aplicación
\begin{align*}
  X\bigr|_{[0,t]\times\Omega}:\bigl([0,t]\times\Omega,\B\bigl([0,t]\bigr)\otimes\F_t\bigr)&\longrightarrow(E,\E)\\
  (s,\omega)&\longmapsto X_s(\omega)
\end{align*}
es $\B\bigl([0,t]\bigr)\otimes\F_t/\E-$medible.
\end{description}
\end{defi}
{} En ocasiones, escribiremos solo $\F_t-$medible en vez de
$\F_t/\E-$medible, si no hay lugar a confusión sobre la
$\sigma-$álgebra $\E.$

{} Evidentemente todo proceso prog. medible es adaptado y medible.
El siguiente teorema, debido a Chung \& Doob \cite{chung}, nos
dice que el recíproco, en cierto sentido, también es válido:
\begin{prop}
Sea $X=(X_t)_{t\geq 0}$ un proceso real medible y adaptado a una
filtración $\{\F_t\}_{t\geq 0}.$ Entonces $X$ posee una versión
progresivamente medible.
\end{prop}
La extensa demostración de este resultado se puede encontrar en
\cite{meyer}. Sin embargo, dado que casi todos los procesos de
nuestro interés son continuos a derecha (o a izquierda), se puede
establecer, bajo estas condiciones, un resultado similar y cuya
prueba es mucho más fácil:

\begin{teorema}
Si $X=(X_t)_{t\geq 0}$ es un proceso con valores en $\R^d,$
continuo por la derecha (o por la izquierda) y adaptado a una
filtración $\{\F_t\}_{t\geq 0},$ entonces es \textit{prog.
medible.}
\end{teorema}
\begin{proof}
Asumamos que $X$ es continuo por la derecha c.s. (el caso continuo
por la izquierda es similar): sea $t\geq 0$ fijo, y para cada
$n\geq 1$ defina $f_n:[0,t]\times\Omega\to\R^d$ por
\[f_n(s,\omega):=
\begin{cases}
 X_0(\omega)& \ \text{si} \ \ \ s=0,\\
 X_{(j-1)t/2^n}(\omega)& \ \text{si} \ \ \ (j-1)t/2^n<s\le jt/2^n,
 \ \ \ j=1,\ldots,2^n.
\end{cases}\]
Por la continuidad a derecha,
$f_n(s,\omega)\xrightarrow[n\to\infty]\,X|_{[0,t]\times\Omega}(s,\omega)$
para casi todo $(s,\omega)\in[0,t]\times\Omega,$ y cada $f_n$ es
$\B([0,t])\otimes\F_t/\B(\R^d)-$medible, pues para todo
$A\in\B(\R^d)$ se tiene que
\begin{align*}
f_n^{-1}(A)&=\{(s,\omega)\in[0,t]\times\Omega:f_n(s,\omega)\in A\}\\
&=\left(\{0\}\times\{X_0\in
A\}\right)\cup\left(\textstyle\bigcup\limits_{\,
j=1}^{2^n}\left(\tfrac{(j-1)t}{2^n},\tfrac{jt}{2^n}\right]
\times\left\{X_{(j-1)t/2^n}\in
A\right\}\right)\in\B([0,t])\otimes\F_t.
\end{align*}
Entonces el límite $X\bigr|_{[0,t]\times\Omega}$ es también
$\B([0,t])\otimes\F_t/\B(\R^d)-$medible.
\end{proof}

\begin{defi}
Sea $(\Omega,\F,\{\F_t\}_{t\geq 0},\mathbf{P})$ un espacio de
probabilidad filtrado. Una variable aleatoria
$\tau:(\Omega,\F)\to[0,\infty]$ se llama un \textit{tiempo de
parada} con respecto a $\{\F_t\}_{t\geq 0}$ (o $\F_t-$tiempo de
parada) si el evento $\{\tau\le t\}\in\F_t$ para todo $t\geq 0.$
\end{defi}

\begin{defi}\label{tiempoentrada}
Sea $X=(X_t)_{t\geq 0}$ con espacio de estados $(\R^d,\B(\R^d))$ y
sea $A\in\B(\R^d).$ La variable aleatoria
$\tau_A:\Omega\to[0,\infty]$ definida por
\begin{equation}\label{tiempoentrada1}
\tau_A(\omega):=\inf\{t\geq 0:X_t(\omega)\in A\}
\end{equation}
se denomina la \textit{primera entrada} (o \textit{primera
visita}) del proceso $X$ al conjunto $A.$
\end{defi}

\begin{prop}
Sean $X, \ A$ y $\tau_A$ como en la definición
\ref{tiempoentrada},
\begin{description}
  \item[(a)] Si $X$ es continuo a derecha y $A$ es un conjunto
  abierto entonces $\tau_A$ es un $\F_{t+}-$tiempo de parada.
  \item[(b)] Si $X$ es continuo y $A$ es un conjunto
  cerrado entonces $\tau_A$ es un $\F_t-$tiempo de parada.
\end{description}
\end{prop}
\begin{proof}
\textbf{(a)} Usando la continuidad por la derecha se deduce
fácilmente que
\begin{equation}\label{tiempoentrada2}
\{\tau_A<t+\varepsilon\}
=\bigcup_{\substack{s\in\mathbb{Q}_+\\
s<t+\varepsilon}}\{X_s\in A\}\in\F_{t+\varepsilon}, \ \ \
\forall\varepsilon>0,
\end{equation}
para cada $t\geq 0,$ luego
\[\{\tau_A\le t\}
=\bigcap_{\varepsilon>0}\{\tau_A<t+\varepsilon\}\in\bigcap_{\varepsilon>0}\F_{t+\epsilon}=\F_{t+}.\]

{}{}\textbf{(b)} Para cada $x\in\R^d$ sea
$d(x,A):=\inf\{|x-y|:y\in A\},$ y considere la sucesión de
vecindades abiertas de $A$ dada por $A_n:=\{x\in E:
d(x,A)<\frac{1}{n}\}, \ n\geq 1.$ Entonces $A_n\downarrow A,$ y
por (\ref{tiempoentrada2}) cada $\tau_{A_n}$ satisface
\[\{\tau_{A_n}<t\}\in\F_t, \ \ \forall t\geq 0.\]
La sucesión $\{\tau_{A_n}\}_{n\geq 1}$ es una sucesión
no-decreciente y dominada por $\tau_A,$ luego existe el límite
$\eta:=\lim_{n\to\infty}\tau_{A_n}\le\tau_A.$ Note que

{}{} si $\tau_A(\omega)=0,$ entonces $\tau_{A_n}(\omega)=0, \
\forall n\geq 1;$

{}{} si $\tau_A(\omega)>0,$ existe un entero $k=k(\omega)\geq 1$
tal que
\[\tau_{A_n}=0 \ \text{para} \ 1\le n\le k, \ \ \text{y} \ \ 0<\tau_{A_n}<\tau_{A_{n+1}}<\tau_A, \ \
\text{para todo} \ n\geq k.\] Vamos a probar que $\eta=\tau_A.$
Para esto es suficiente probar que sobre el conjunto
\[\{\tau_A>0, \ \eta<+\infty\}\]
se cumple la desigualdad $\eta\geq\tau_A.$ Si $\tau_A(\omega)>0$ y
$\eta(\omega)<+\infty,$ por continuidad de $X,$
$X_{\eta}(\omega)=\lim_{n\to\infty}X_{\tau_{A_n}}(\omega)$ y
$X_{\tau_{A_m}}(\omega)\in\partial A_m\subseteq A_n, \ \forall
m>n\ge k(\omega).$ Haciendo $m\to\infty$ se obtiene que
$X_{\eta}(\omega)\in A_n, \ \forall n\ge k(\omega),$ y así
$X_\eta(\omega)\in\bigcap_{n=1}^\infty A_n=A.$ Por lo tanto
$\tau_A\le\eta$ y $\tau_A=\lim_{n\to\infty}\tau_{A_n}.$ Se sigue
finalmente que
\begin{align*}
\{\tau_A=0\}&=\{X_0\in A\}\in\F_0,\\
\{\tau_A\le t\}&=\bigcap_{n=1}^\infty\{\tau_{A_n}<t\}\in\F_t, \ \
\ \ \forall t>0.
\end{align*}
\end{proof}
La parte \textbf{(b)} del anterior teorema, debida a Wentzell
\cite{wentzell}, será de gran utilidad en los capítulos 2 y 5,
pues casi todos los tiempos de parada allí considerados serán de
la forma (\ref{tiempoentrada1}) con $A$ abierto y $X$ continuo.

\begin{defi}
Sea $\tau$ un tiempo de parada con respecto a una filtración
$\{\F_t\}_{t\geq 0}.$ Se define la $\sigma-$álgebra $\F_\tau$ como
\[\F_\tau:=\{A\in\F:A\cap\{\tau\le t\}\in\F_t, \ \forall t\geq 0\}.\]
\end{defi}

\begin{prop}
(a) Si $\tau$ es un $\F_t-$tiempo de parada, entonces $\tau$ es
$\F_\tau-$medible.

{}{}(b) Si $\tau,\,\eta$ son $\F_t-$tiempos de parada, entonces
$\tau\wedge\eta, \ \tau\vee\eta$ también lo son. Si además
$\tau\le\eta,$ entonces $\F_\tau\subseteq\F_\eta.$

{}{}(c) Si $\{\tau_n\}_{n\geq 1}$ son $\F_t-$tiempos de parada,
$\sup_{n\geq 1}\tau_n$ también lo es.

{}{}(d) Si $\tau$ es un $\F_t-$tiempo de parada y
$\eta:\Omega\to[0,+\infty]$ es $\F_\tau-$medible tal que
$\eta\geq\tau,$ entonces $\eta$ es también un $\F_t-$tiempo de
parada. En particular, la suma de dos $\F_t-$tiempos de parada es
también un $\F_t-$tiempo de parada.

{}{}(e) Si $\tau,\,\eta$ son $\F_t-$tiempos de parada entonces
$\F_{\tau\wedge\eta}=\F_\tau\cap\F_\eta.$
\end{prop}
\begin{proof}
(a) Debido a que $\B\bigl([0,+\infty)\bigr)$ es generada por los
intervalos de la forma $[0,s], \ s\geq 0,$ basta con verificar que
$\tau^{-1}\bigl([0,s]\bigr)=\{\tau\le s\}\in\F_\tau, \ \forall
s\geq 0.$ Pero esto es trivial, en vista de que para cada $s\geq
0,$
\[\{\tau\le s\}\cap\{\tau\le t\}=\{\tau\le s\wedge t\}\in\F_{s\wedge t}\subseteq\F_t, \ \ \forall t\geq 0.\]

{}{}(b) Para todo $t\geq 0,$
\begin{align*}
\{\tau\wedge\eta\le t\}&=\{\tau\le t\}\cap\{\eta\le t\}\in\F_t\\
\{\tau\vee\eta\le t\}&=\{\tau\le t\}\cup\{\eta\le t\}\in\F_t.
\end{align*}
Si $\tau\le\eta,$ para todo $A\in\F_\tau$ se tiene que
\[A\cap\{\eta\le t\}=A\cap\{\tau\le t\}\cap\{\eta\le t\}\in\F_t, \ \ \forall t\geq 0\]
luego $A\in\F_\eta.$

{}{}(c) $\{\sup_{n\geq 1}\tau_n\le
t\}=\bigcap_{n=1}^\infty\{\tau_n\le t\}\in\F_t.$

{}{}(d) Dado que $\eta$ es $\F_\tau-$medible, $\{\eta\le
t\}\in\F_\tau$ para todo $t\geq 0,$ es decir \[\{\eta\le
t\}=\{\eta\le t\}\cap\{\tau\le t\}\in\F_t, \ \ \ \forall t\geq 0\]
lo que implica que $\eta$ es un $\F_t-$tiempo de parada.

{} Si $\tau$ y $\eta$ son $\F_t-$tiempos de parada, entonces
$\tau\vee\eta$ es también un $\F_t-$tiempo de parada y
$\tau+\eta\geq\tau\vee\eta.$ Como $\tau$ es $\F_\tau-$medible,
$\eta$ es $\F_\eta-$medible y
$\F_\tau,\F_\eta\subseteq\F_{\tau\vee\eta}$ entonces $\tau+\eta$
es $\F_{\tau\vee\eta}-$medible. Con esto y la primera parte de (d)
se concluye que $\tau+\eta$ es también un $\F_t-$tiempo de parada.

{}{}(e) Dado que $\tau\wedge\eta\le\tau$ y
$\tau\wedge\eta\le\eta,$ la inclusión
$\F_{\tau\wedge\eta}\subseteq\F_\tau\cap\F_\eta$ es consecuencia
inmediata de la última afirmación de (b). Sea ahora
$A\in\F_\tau\cap\F_\eta.$ Entonces $A\cap\{\tau\wedge\eta\le
t\}=\bigl[A\cap\{\tau\le t\}\bigr]\cup\bigl[A\cap\{\eta\le
t\}\bigr]\in\F_t, \ \forall t\geq 0.$
\end{proof}

\begin{prop}
Sean $X=(X_t)_{t\geq 0}$ un proceso estocástico con valores en
$(E,\E)$ y sea $\tau$ un tiempo de parada con respecto a una
filtración $\{\F_t\}_{t\geq 0}.$ Si $X$ es progresivamente medible
con respecto a $\{\F_t\}_{t\geq 0}$ entonces la v.a.
$X_\tau\mathbf{1}_{\{\tau<\infty\}}$ es $\F_\tau-$medible y el
proceso $(X_{t\wedge\tau})_{t\geq 0}$ es también progresivamente
medible con respecto a $\{\F_t\}_{t\geq 0}.$
\end{prop}
\begin{proof}
Sea $A\in\E.$ Dado que
\[\{X_\tau\mathbf{1}_{\{\tau<\infty\}}\in A\}\cap\{\tau\le t\}
=\{X_{t\wedge\tau}\in A\}\cap\{\tau\le t\},\] basta con demostrar
la segunda afirmación de la proposición: para $t\geq 0$ fijo, la
aplicación $(s,\omega)\mapsto s\wedge\tau(\omega)$ es
$\B\bigl([0,t]\bigr)\otimes\F_t/\B\bigl([0,t]\bigr)-$medible, pues
para cada $s\le t$
\begin{align*}
\{&(r,\omega)\in[0,t]\times\Omega:r\wedge\tau(\omega)\le s\}\\
&=\bigl([0,s]\times\Omega\bigr)\cup\bigl([0,t]\times\{\tau\le
s\}\bigr)\in\B\bigl([0,t]\bigr)\otimes\F_s\subseteq
\B([0,t])\otimes\F_t.
\end{align*}
En consecuencia la función $\psi(s,\omega):=(s\wedge\tau(\omega),
\omega)$ de $[0,t]\times\Omega$ en sí mismo es
$\B\bigl([0,t]\bigr)\otimes\F_t/\B\bigl([0,t]\bigr)\otimes\F_t-$medible.
Por hipótesis
\begin{align*}
  X\bigr|_{[0,t]\times\Omega}:\bigl([0,t]\times\Omega,\B\bigl([0,t]\bigr)\otimes\F_t\bigr)&\longrightarrow(E,\E)\\
  (s,\omega)&\longmapsto X_s(\omega)
\end{align*}
es $\B\bigl([0,t]\bigr)\otimes\F_t/\E-$medible. Por lo tanto, la
compuesta \[(s,\omega)\mapsto
\bigl(X\bigr|_{[0,t]\times\Omega}\circ\psi\bigr)(s,\omega)=X_{s\wedge\tau(\omega)}(\omega),\]
es $\B\bigl([0,t]\bigr)\otimes\F_t/\E-$medible
\end{proof}

\begin{defi} Un proceso real $X=(X_t)_{t\geq 0}$ con
$E(|X_t|)<+\infty, \ \forall t\geq 0,$ y adaptado a una filtración
$\{\F_t\}_{t\geq 0}$ es llamado una \textit{martingala} (resp.
\textit{super-martingala}, resp. \textit{sub-martingala}) con
respecto a $\{\F_t\}_{t\geq 0}$ si
\[E[X_s\mid\F_t]=X_t, \ \ \ \text{(resp. \ $\le,$ \ resp. \ $\ge)$ \ \ \ para todo $s\geq t\geq 0.$}\]
\end{defi}
\begin{obse}
Note que si $X$ es una martingala entonces $E[X_t]=E[X_0], \
\forall t\geq 0.$
\end{obse}

Los siguientes tres teoremas son fundamentales en la teoría de
martingalas. Omitiremos sus demostraciones debido a que se
requieren algunos resultados de martingalas con parámetro de
tiempo discreto y el concepto de integrabilidad uniforme (ver por
ejemplo \cite{doob,blanco}).

\begin{teorema}[Primera desigualdad de Doob]
Sea $X=(X_t)_{t\geq 0}$ una $\F_t-$martingala continua por la
derecha tal que $E(|X_t|^p)<\infty, \ \forall t\in[0,T],$ con
$p\geq 1.$ Entonces
\[\mathbf{P}\biggl(\sup_{\, t\in [0,T]}|X_t|\geq\varepsilon\biggr)
\le\frac{1}{\varepsilon^p}E\bigl(|X_T|^p\bigr).\]
\end{teorema}
\begin{teorema}[Segunda desigualdad de Doob]
Sea $X=(X_t)_{t\geq 0}$ una $\F_t-$martingala continua por la
derecha tal que $E(|X_t|^p)<\infty, \ \forall t\in[0,T],$ con
$p>1.$ Entonces
\[E\biggl[\sup_{\, t\in [0,T]}|X_t|^p\biggr]
\le\left(\frac{p}{p-1}\right)^p E\bigl(|X_T|^p\bigr).\]
\end{teorema}
\begin{teorema}[Teorema de parada de Doob]
Sea $(X_t)_{t\geq 0}$ una martingala real con respecto a una
filtración $\{\F_t\}_{t\ge 0}$ tal que existe una v.a.
$X_\infty\in L^1(\Omega,\F,\mathbf{P})$ para la cual
\[X_t=E[X_\infty\mid\F_t], \ \ \ \forall t\geq 0,\]
y sean $\tau, \ \eta$ tiempos de parada con respecto a
$\{\F_t\}_{t\ge 0}$ tales que $\tau\le\eta.$ Si $X$ es continuo a
derecha c.s. entonces
\[X_\tau=E[X_\eta\mid\F_\tau]=E[X_\infty\mid\F_\tau]
, \ \ \ \text{c.s.}\]
\end{teorema}

{} Este último teorema es un caso particular de un teorema más
general para sub-martingalas uniformemente integrables dado por
Doob y denominado ``Optional sampling", ver \cite{karatzas}.

\begin{teorema}\label{lemarepmart}
Sea $X=(X_t)_{t\geq 0}$ un proceso real adaptado a una filtración
$\{\F_t\}_{t\ge 0}$ tal que para todo $\F_t-$tiempo de parada
acotado $\tau$ se tiene
\[E[X_\tau]=E[X_0].\]
Entonces $X$ es una martingala con respecto a $\{\F_t\}_{t\ge 0}.$
\end{teorema}
\begin{proof}
Sean $t\ge s\ge 0$ y $A\in\F_s.$ Entonces
$\tau:=s\mathbf{1}_{A^c}+t\mathbf{1}_{A}$ es un tiempo de parada
con $\tau=\tau\wedge t.$ Esto implica
\[E[X_0]=E[X_\tau]=E[X_s\mathbf{1}_{A^c}+X_t\mathbf{1}_{A}]=E[X_s\mathbf{1}_{A^c}]+E[X_t\mathbf{1}_{A}].\]
Por otro lado, el tiempo de parada $\tau\equiv s$ satisface
\[E[X_0]=E[X_s]=E[X_s\mathbf{1}_{A^c}]+E[X_s\mathbf{1}_{A}].\]
Por lo tanto, para todo $A\in\F_s$ tenemos
$E[X_t\mathbf{1}_A]=E[X_s\mathbf{1}_A].$ Usando la definición de
esperanza condicional, obtenemos $E[X_t\mid\F_s]=X_s.$
\end{proof}
{} Usando el teorema de parada de Doob se puede fácilmente ver que
el recíproco también es cierto.

{}\begin{defi} Un \textit{Movimiento Browniano} o \textit{Proceso
de Wiener} unidimensional estándar es un
proceso estocástico $(W_t)_{t\geq 0}$ con valores en $\R$ que
satisface
\begin{description}
  \item[(i)] $W_0=0$ c.s.
  \item[(ii)] para todo $0\le t\le s,$ el incremento $W_s-W_t$ tiene distribución normal con media cero y
  varianza $t-s.$
  \item[(iii)] para todo $0\le u\le r\le s\le t,$ el incremento $W_t-W_s$ es
  independiente de $W_r-W_u.$
\end{description}
\end{defi}

\begin{obse}
La función de probabilidad conjunta de las variables aleatorias
{}$W_{t_1},W_{t_2},\ldots,W_{t_n},$ con
$0<t_1<t_2<\cdots<t_n<\infty,$ está dada por
\begin{equation}\label{fpcmovbrown}
f_{t_1\cdots
t_n}(x_1,\ldots,x_n)=f_{t_1}(x_1)\,f_{t_2-t_1}(x_2-x_1)\cdots
f_{t_n-t_{n-1}}(x_n-x_{n-1}),
\end{equation}
donde \[f_t(x)=\frac{1}{\sqrt{2\pi
t}}\exp\left(-\frac{x^2}{2t}\right).\] El sistema de
distribuciones dado en (\ref{fpcmovbrown}) satisface las
condiciones del teorema de consistencia de Kolmogorov, ver
\cite{friedman2,karatzas,tudor,blanco}. Esto garantiza la
existencia del movimiento Browniano estándar.

{} Existen sin embargo otros métodos, aún más técnicos, para
garantizar la existencia del movimiento Browniano. Una
construcción basada en convergencia débil y en la aproximación por
caminatas aleatorias puede ser encontrada en \cite{billingsley}.
Otro es el método del espacio de Hilbert \cite{karatzas}. La idea
original de esta construcción fue dada por Paul Levy en 1948 y
simplificada posteriormente por Ciesielski en 1961.
\end{obse}

{} Si $X$ es una variable aleatoria normal con media cero y
varianza $\sigma$ entonces
\[E(X^{2n})=\frac{(2n)!}{2^n n!}\sigma^{2n}, \ \ \ E(X^{2n+1})=0, \ \ \ \ \ \ n=0,1,2,\ldots\]
En particular, para el movimiento Browniano estándar $(W_t)_{t\geq
0}$
\[E(|W_t-W_s|^4)=3|t-s|^2, \ \ \ \ \ 0\le s<t.\]
Usando esto junto con en el teorema de Kolmogorov-\u{C}entsov
\cite{karatzas} se obtiene
\begin{teorema} El movimiento Browniano unidimensional estándar $(W_t)_{t\geq 0}$ posee una versión continua.
\end{teorema}
{} De aquí en adelante haremos referencia a la versión continua
del movimiento Browniano.{}

Otra propiedad importante de las trayectorias del movimiento
Browniano es la no-diferenciabilidad: si una función $f:\R\to\R$
es continuamente diferenciable en $[s,t]$ y
\[\pi_n=\{s=t_{n,1},t_{n,2},\ldots,t_{n,m_n}=t\}, \ \ \ n\geq 1\]
es una sucesión de particiones de $[s,t]$ con
$|\!|\pi_n|\!|=\max_{1\le j\le
m_n}(t_{n,j}-t_{n.j-1})\xrightarrow[n\to\infty]\,0,$ entonces
\[\sum_{j=1}^{m_n}(f(t_{n,j})-f(t_{n,j-1}))^2\xrightarrow[n\to\infty]\,0.\]
Tal conclusión no se tiene para el movimiento Browniano. En cambio, se tiene lo siguiente
\begin{teorema}\label{v2browniano}
Si $\pi_n=\{s=t_{n,1},t_{n,2},\ldots,t_{n,m_n}=t\}, \ n\geq 1,$
una sucesión de particiones de $[s,t]$ con
$|\!|\pi_n|\!|\xrightarrow[n\to\infty]\,0,$ entonces
\[S_n=\sum_{j=1}^{m_n}(W_{t_{n,j}}-W_{t_{n,j-1}})^2\xrightarrow[n\to\infty]{L^2}t-s.\]
Si además $\sum_{n=1}^\infty|\!|\pi_n|\!|<+\infty,$ entonces
$S_n\xrightarrow[n\to\infty]{c.s}t-s.$
\end{teorema}
\begin{proof}
Para esta demostración, definimos $t_j=t_{n,j}, \ m=m_n.$ Con esta notación, se tiene
\[S_n-(t-s)=\sum_{j=1}^{m}\bigl[(W_{t_j}-W_{t_{j-1}})^2-(t_j-t_{j-1})\bigr].\]
Dado que los sumandos son independientes y de esperanza cero
\begin{align*}
E\bigl(|S_n-(t-s)|^2\bigr)
&=E\biggl[\sum_{j=1}^{m}\bigl[(W_{t_j}-W_{t_{j-1}})^2-(t_j-t_{j-1})\bigr]^2\biggr]\\
&=\sum_{j=1}^{m}E\bigl[(Y_j^2-1)(t_j-t_{j-1})\bigr]^2,
\end{align*}
donde \[Y_j:=\frac{W_{t_j}-W_{t_{j-1}}}{(t_j-t_{j-1})^{1/2}}.\]
Dado que los $Y_j$ están igualmente distribuidos con distribución
normal, se sigue
\[E\bigl(|S_n-(t-s)|^2\bigr)=E(|Y_1^2-1|^2)\sum_{j=1}^m(t_j-t_{j-1})^2
\le
E(|Y_1^2-1|^2)\cdot(t-s)|\!|\pi_n|\!|\xrightarrow[n\to\infty]\,0.\]
Además, se tiene la siguiente desigualdad
\[D_n:=2\sum_{k=1}^n(t_k-t_{k-1})^2\le 2(t-s)|\!|\pi_n|\!|.\]
Así, la condición $\sum_{n=1}^\infty|\!|\pi_n|\!|<+\infty$ implica que
$\sum_{n=1}^\infty D_n<+\infty$ y como consecuencia
$S_n\xrightarrow[n\to\infty]{c.s}t-s,$ ya que
$S_n\xrightarrow[n\to\infty]{L^2}t-s.$
\end{proof}

\begin{coro}
Si $\sum_{n=1}^\infty|\!|\pi_n|\!|<+\infty,$ y
$\pi_n\subseteq\pi_{n+1}$ para todo $n\geq 1,$ entonces
\[\sum_{k=1}^n|W_{t_{n,k}}-W_{t_{n,k-1}}|\xrightarrow[n\to\infty]{c.s}\infty.\]
En particular c.s. las trayectorias del movimiento Browniano
tienen variación infinita sobre cada intervalo compacto y como
consecuencia c.s. las trayectorias en ningún punto son derivables.
\end{coro}
\begin{proof}
Tenemos que
\begin{equation}\label{mbnodif1}
S_n\le\max_{1\le k\le
n}|W_{t_{n,k}}-W_{t_{n,k-1}}|\cdot\sum_{k=1}^n|W_{t_{n,k}}-W_{t_{n,k-1}}|,
\end{equation}
y por el teorema anterior
\begin{equation}\label{mbnodif2}
\sum_{k=1}^n|W_{t_{n,k}}-W_{t_{n,k-1}}|^2\xrightarrow[n\to\infty]{c.s}t-s\neq
0.
\end{equation}
Por la continuidad uniforme sobre $[s,t]$ del movimiento Browniano
se tiene que
\begin{equation}\label{mbnodif3}
\max_{1\le k\le
n}|W_{t_{n,k}}-W_{t_{n,k-1}}|\xrightarrow[n\to\infty]{c.s}t-s\neq
0.
\end{equation}
Usando las relaciones (\ref{mbnodif1}), (\ref{mbnodif2}) y
(\ref{mbnodif3}) se obtiene el resultado requerido.
\end{proof}

\subsection{La integral estocástica}
\subsubsection{Definición y propiedades}
Sea $(W_t)_{t\geq 0}$ un movimiento Browniano unidimensional
estándar con $W_0=0$ y sea $\{\F_t\}_{t\geq 0}$ la
filtración generada por $(W_t)_{t\geq 0}$ aumentada con los
conjuntos $\mathbf{P}-$nulos de $\Omega,$ es decir
\[\F_t=\sigma\bigl(\mathcal{N}\cup\{W_s: 0\leq s\le t\}\bigr),
\ \ \ t\geq 0 \] donde $\mathcal{N}=\{E\subseteq\Omega : \exists
G\in \F,\, G\supseteq E \text{ y } \mathbf{P}(G)=0\}.$ Se puede
probar que esta filtración satisface las condiciones usuales, ver
\cite{karatzas}.

{} Nuestro objetivo en esta sección es mostrar la existencia, en
cierto sentido, de la integral estocástica
\[\int_0^t X_s\,dW_s\]
para un tipo especial de procesos $(X_t)_{t\geq 0}.$ Debido a que
las trayectorias del movimiento Browniano son no-diferenciables y
tampoco de variación acotada, dicha integral no se puede definir
en el sentido Lebesgue-Stieljes. Aún así su construcción se
asemeja a la de la integral de Riemman, ya que se define primero
para procesos escalonados y luego por aproximación se extiende a
una clase más general de procesos.

\begin{defi}
Sean $0\le t_0<t<+\infty$ y $p\geq 1.$ Notaremos con $H^p[t_0,t]$
el espacio vectorial de los procesos $(X_s)_{s\geq 0}$ con valores
en $\R$ y adaptados a $\{\F_t\}_{t\geq 0}$ tales que
\[\int_{t_0}^t|X_s|^p\,ds<+\infty \ \ \ \text{c.s.}\]
y con $M^p[t_0,t]$ el subconjunto de $H^p[t_0,t]$ de los procesos
$(X_s)_{s\geq 0}$ que satisfacen
\[E\int_{t_0}^t|X_s|^p\,ds<+\infty.\]
\end{defi}

\begin{defi}
Diremos que $X\in H^p[t_0,t]$ es un \textit{proceso escalonado} o
\textit{proceso simple} si es de la forma
\[X_s=\sum_{i=0}^{n-1}e_i\mathbf{1}_{[t_i,t_{i+1})}(s)\] donde
$t_0<t_1<\cdots<t_n=t$ es una partición del intervalo $[t_0,t]$ y
cada $e_i$ es una variable aleatoria real $\F_{t_i}-$medible,
$0\le i\le n-1.$
\end{defi}

\begin{lema}[Aproximación por procesos escalonados]\label{aproxlema}
Sea $X\in H^2[t_0,t].$ Entonces existe una sucesión de procesos
escalonados $(X^n)_{n\geq 1}$ en $H^2[t_0,t]$ tal que
\begin{equation}\label{aproxsimples1}
\lim_{n\to\infty}\int_{t_0}^t|X_s-X_s^n|^2\,ds=0 \ \ \ \ \
\text{c.s.}
\end{equation}
\end{lema}
\begin{proof}
Sin pérdida de generalidad asumiremos $t_0=0.$ La prueba se hará
en 3 pasos:

{}{}\textbf{Paso 1.} Si $X$ es acotado y tiene trayectorias
continuas c.s., para cada $n\geq 1$ se define el proceso simple
\[X_s^n:=\sum_{k=0}^{n-1}X_{kt/n}\mathbf{1}_{A_{n,k}}(s), \ \ \ \
\ \text{donde} \
A_{n,k}=\bigl[\tfrac{kt}{n},\tfrac{(k+1)t}{n}\bigr), \ 0\le k\le
n.\] De la continuidad del proceso se tiene que
$X_s^n\xrightarrow[n\to\infty]{c.s}X_s,$ y por convergencia
dominada
\[\int_0^t|X_s^n-X_s|^2\,ds\xrightarrow[n\to\infty]{c.s.}0.\]

{}{}\textbf{Paso 2.} Si $X$ es acotada, definimos $G_s:=\int_0^s
X_r\,dr,$ y para cada $n\geq 1$
\[X_s^n:=\frac{G_s-G_{(s-1/n)\wedge 0}}{1/n}\]
el cual es continuo y acotado. El teorema fundamental del cálculo
implica que $X_s^n\xrightarrow{c.s}X_s,$ y de nuevo por
convergencia dominada se obtiene (\ref{aproxsimples1}).

{}{}\textbf{Paso 3.} Por último, para $X\in H^2[0,t]$ arbitrario
se define
\[X_s^n:=X_s\cdot\mathbf{1}_{\{|X_s|\le n\}}.\]
Cada proceso $X^n$ es acotado y satisface $|X_s^n-X_s|^2\le
2|X_s|^2.$ Usando el hecho de que $\int_0^t\!|X_s|^2 ds<+\infty$
c.s. y convergencia dominada obtenemos
\[\lim_{n\to\infty}\int_{t_0}^t|X_s-X_s^n|^2\,ds=0 \ \ \ \ \
\text{c.s.}\]
\end{proof}

\begin{lema}\label{aproxsimples2}
Sea $X\in M^2[t_0,t].$ Entonces existe una sucesión de procesos
simples $(X^n)_{n\geq 1}$ en $M^2[t_0,t]$ tal que
\begin{equation}
E\int_{t_0}^t|X_s-X_s^n|^2\,ds\xrightarrow[n\to\infty]\,0.
\end{equation}
\end{lema}
\begin{proof}
Dado que en particular $X\in H^2[t_0,t],$ por el lema anterior
existe una sucesión de procesos simples $(\widetilde{X}^n)_{n\geq
1}$ en $H^2[t_0,t]$ que satisfacen (\ref{aproxsimples1}). Para
cada $N>0$ definimos
\[\Pi_N(s):=
\begin{cases}
\,s               & \text{si}\ \,|s|\le N\\
N\frac{s}{|s|}  & \text{si}\ \,|s|>N
\end{cases} \]
Dado que $|\Pi_N(s)-\Pi_N(r)|\le|s-r|,$
\[\int_{t_0}^t|\Pi_N(X_s)-\Pi_N(\widetilde{X}_s^n)|^2\,ds\le
\int_{t_0}^t|X_s-\widetilde{X}_s^n|^2\,ds\xrightarrow[n\to\infty]{c.s}0\]
y de la desigualdad $(a+b)^2\le 2(a^2+b^2),$
\[\int_{t_0}^t|\Pi_N(X_s)-\Pi_N(\widetilde{X}_s^n)|^2\,ds\le 4tN^2.\]
Aplicando convergencia dominada en la variable $\omega$ tenemos
\begin{equation}\label{aproxsimples3}
E\int_{t_0}^t|\Pi_N(X_s)-\Pi_N(\widetilde{X}_s^n)|^2\,ds\xrightarrow[n\to\infty]\,0.
\end{equation}
Por otro lado
$|\Pi_N(X_s(\omega))-X_s(\omega)|\xrightarrow[N\to\infty]\,0$ y
$|\Pi_N(X_s(\omega))-X_s(\omega)|^2\le |X_s(\omega)|^2$ para todo
$(s,\omega)\in[t_0,t]\times\Omega,$ y de nuevo por convergencia
dominada, pero esta vez en las variables $s$ y $\omega$ se obtiene
\[E\int_{t_0}^t|\Pi_N(X_s)-X_s|^2\,ds\xrightarrow[N\to\infty]\,0.\]
De aquí se deduce que para todo entero positivo $k$ existe $N_k$
tal que
\[E\int_{t_0}^t|\Pi_{N_k}(X_s)-X_s|^2\,ds<\frac{1}{4k}\]
y de (\ref{aproxsimples3}) se sigue que existe $n_k$ tal que
\[E\int_{t_0}^t|\Pi_{N_k}(X_s)-\Pi_{N_k}(\widetilde{X}_s^{n_k})|^2\,ds<\frac{1}{4k}.\]
Tomando la sucesión de procesos simples
$X_s^k:=\Pi_{N_k}(\widetilde{X}_s^{n_k})$ obtenemos
\begin{align*}
E&\int_{t_0}^t|X_s-X_s^k|^2\,ds\\
&\le 2\,E\!\int_{t_0}^t|X_s-\Pi_{N_k}(X_s)|^2\,ds
+2\,E\!\int_{t_0}^t|\Pi_{N_k}(X_s)
-\Pi_{N_k}(\widetilde{X}_s^{n_k})|^2\,ds\xrightarrow[k\to\infty]\,0.
\end{align*}
\end{proof}

\begin{defi}[Integral estocástica para procesos escalonados]
Sea $X$ un proceso escalonado en $H^2[t_0,t]$ de la forma
\[X_s=\sum_{i=0}^{n-1}e_i\mathbf{1}_{[t_i,t_{i+1})}(s)\] donde
cada $e_i$ es $\F_{t_i}-$medible y $t_0<t_1<\cdots<t_n=t.$ La
variable aleatoria
\[\sum_{i=0}^{n-1}e_i(W_{t_{i+1}}-W_{t_i}),\]
es denotada por
\[\int_{t_0}^t X_s\,dW_s\]
y es llamada la \textit{integral estocástica} o \textit{integral
de Itô} de $X$ con respecto al movimiento Browniano $(W_s)_{s\geq
0}.$ Note que si $t_0=0$ entonces $\int_0^t\!X_s\,dW_s$ es
$\F_t-$medible.
\end{defi}

\begin{prop}\label{itolineal1}
Sean $X^1,X^2$ procesos escalonados en $H^2[t_0,t]$ y sean
$\alpha,\beta\in\R.$ Entonces $\alpha X^1+\beta X^2\in H^2[t_0,t]$
y
\[\int_{t_0}^t(\alpha X_s^1+\beta X_s^2)\,dW_s
=\alpha\int_{t_0}^t X_s^1\,dW_s+\beta\int_{t_0}^t X_s^2\,dW_s.\]
\end{prop}
\begin{proof}
Usaremos una partici\'{o}n $t_0<t_1<\cdots<t_n=t$ tal que
\[X^1=\sum_{i=0}^{n-1}e_i\mathbf{1}_{[t_j,t_{i+1})} \ \ \ \text{y}
\ \ \ X^2=\sum_{i=0}^{n-1}c_i\mathbf{1}_{[t_i,t_{i+1})}\] donde
$e_i$ y $c_i$ son variables aleatorias $\F_{t_i}$-medibles para
cada $i$ (si las particiones en las fórmulas de $X^1$ y $X^2$ son
distintas, podemos encontrar un refinamiento común de ambas
particiones). Entonces
\[\alpha X^1+\beta X^2=\sum_{j=0}^{n-1}(\alpha e_{j}+\beta
c_{j})\mathbf{1}_{[t_{j},t_{j+1})}\] y
\begin{align*}
\int_{t_0}^t(\alpha X_s^1+\beta X_s^2)\,dW_s &
=\sum_{j=0}^{n-1}(\alpha e_{j}+\beta c_{j})(W_{t_{j+1}}-W_{t_{j}})\\
&=\alpha\sum_{j=0}^{n-1}e_{j}(W_{t_{j+1}}-W_{t_{j}})+\beta\sum_{j=0}
^{n-1}c_{j}(W_{t_{j+1}}-W_{t_{j}})\\
&=\alpha\int\nolimits_{t_{0}}^{t}X_s^1\,dW_{s}+\beta\int\nolimits_{t_{0}
}^{t}X_s^2\,dW_{s}.
\end{align*}
\end{proof}

\begin{prop}\label{propitosimples}
Si $X$ es un proceso escalonado en $M^2[t_0,t]$ entonces
\begin{align}
E\int_{t_0}^t X_s\,dW_s&=0,\label{espcero1}\\
E\biggl[\Bigl|\int_{t_0}^t
X_s\,dW_s\Bigr|^2\biggr]&=E\int_{t_0}^t|X_s|^2\,ds.\label{isometria1}
\end{align}
\end{prop}
\begin{proof}
Si $X$ es de la forma
$X_s=\sum_{i=0}^{n-1}e_i\mathbf{1}_{[t_i,t_{i+1})}(s)$ entonces
\[E\int_{t_0}^t|X_s|^2\,ds=\sum_{i=0}^{n-1}E(e_i^2)\cdot(t_{i+1}-t_i)\]
que es finito por hipótesis, y en particular $E(|e_i|)<+\infty.$
Además $E(|W_{t_{i+1}}-W_{t_i}|)<+\infty.$ Dado que $e_i$ es
$\F_{t_i}-$ medible mientras que $W_{t_{i+1}}-W_{t_i}$ es
independiente de $\F_{t_i},$
\[E[e_i(W_{t_{i+1}}-W_{t_i})]=E(e_i)\cdot E(W_{t_{i+1}}-W_{t_i})=0, \ \ \ 0\le i\le n-1.\]
Sumando sobre $i$ se obtiene (\ref{espcero1}). Por otro lado
\begin{align*}
\Bigl|\int_{t_0}^t
X_s\,dW_s\Bigr|^2&=\sum_{i=0}^{n-1}\sum_{j=0}^{n-1}e_i
e_j(W_{t_{i+1}}-W_{t_i})(W_{t_{j+1}}-W_{t_j})\\
&=\sum_{i=0}^{n-1}e_i^2(W_{t_{i+1}}-W_{t_i})^2+\sum_{i<j}e_i
e_j(W_{t_{i+1}}-W_{t_i})(W_{t_{j+1}}-W_{t_j}).
\end{align*}
Si $0\le i<j\le n-1$ entonces $W_{t_{j+1}}-W_{t_j}$ es
independiente de $e_i e_j(W_{t_{i+1}}-W_{t_i}),$ luego
\[E[e_i
e_j(W_{t_{i+1}}-W_{t_i})(W_{t_{j+1}}-W_{t_j})=E[e_i
e_j(W_{t_{i+1}}-W_{t_i})(W_{t_{j+1}}-W_{t_j})\cdot
E[W_{t_{i+1}}-W_{t_i}]=0,\] y de nuevo por la independencia entre
$e_i$ y $W_{t_{i+1}}-W_{t_i},$
\[E[e_i^2(W_{t_{i+1}}-W_{t_i})^2]=E(e_i^2)\cdot
E[(W_{t_{i+1}}-W_{t_i})^2]=E(e_i^2)\cdot(t_{i+1}-t_i), \ 0\le i\le
n-1.\] Por lo tanto
\[E\biggl[\Bigl|\int_{t_0}^t
X_s\,dW_s\Bigr|^2\biggr]=\sum_{i=0}^{n-1}E(e_i^2)\cdot(t_{i+1}-t_i)
=E\int_{t_0}^t|X_s|^2\,ds.\]
\end{proof}

\begin{lema}\label{lemadef}
Para todo proceso escalonado $X\in H^2[t_0,t]$ y para todo
$\varepsilon>0,\, N>0,$
\[\mathbf{P}\biggl(\Bigl|\int_{t_0}^t
X_s\,dW_s\Bigr|>\varepsilon\biggr)
\le\mathbf{P}\biggl(\int_{t_0}^t|X_s|^2\,ds>N\biggr)+\frac{N}{\varepsilon^2}.\]
\end{lema}
\begin{proof}
Para $X$ de la forma
$X_s=\sum_{i=0}^{n-1}e_i\mathbf{1}_{[t_i,t_{i+1})}(s)$ definimos
\[\phi_N(s)=
\begin{cases}
X_s, \ \ & \text{si}\ \,t_k\le s<t_{k+1}\ \ \text{y} \ \
\sum_{i=0}^k e_i^2(t_{i+1}-t_i)\le N,\\
0, \ \ & \text{si}\ \,t_k\le s<t_{k+1}\ \ \text{y} \ \
\sum_{i=0}^k e_i^2(t_{i+1}-t_i)>N.
\end{cases}\]
Entonces $\phi_N\in H^2[t_0,t]$ y
\[\int_{t_0}^t \phi_N^2(s)\,ds=\sum_{i=0}^\nu e_i^2(t_{i+1}-t_i)\]
donde $\nu$ es el entero más grande tal que $\sum_{i=0}^\nu
e_i^2(t_{i+1}-t_i)\le N, \ \nu\le n-1,$ y por esta razón
\[E\int_{t_0}^t \phi_N^2(s)\,ds\le N.\] Además $X_s-\phi_N(s)=0$
para todo $s\in[t_0,t]$ si $\int_{t_0}^t |X_s|^2\,ds\le N,$ luego
\[\mathbf{P}\biggl(\Bigl|\int_{t_0}^t
X_s\,dW_s\Bigr|>\varepsilon\biggr)
\le\mathbf{P}\biggl(\Bigl|\int_{t_0}^t
\phi_N(s)\,dW_s\Bigr|>\varepsilon\biggr)
+\mathbf{P}\biggl(\int_{t_0}^t|X_s|^2\,ds>N\biggr),\] y por la
desigualdad de Chebyshev
\begin{align*}
\mathbf{P}\biggl(\Bigl|\int_{t_0}^t
\phi_N(s)\,dW_s\Bigr|>\varepsilon\biggr)&=\mathbf{P}\biggl(\Bigl|\int_{t_0}^t
\phi_N(s)\,dW_s\Bigr|^2>\varepsilon^2\biggr)\le
\frac{1}{\varepsilon^2}\,E\biggl[\Bigl|\int_{t_0}^t
\phi_N(s)\,dW_s\Bigr|^2\biggr]\\
&=\frac{1}{\varepsilon^2}\,E\int_{t_0}^t|\phi_N(s)|^2\,ds\le
\frac{N}{\varepsilon^2},
\end{align*} obteniendo el resultado.
\end{proof}
Ahora procedemos a definir la integral estocástica para cualquier
proceso $X$ en $H^2[t_0,t]\,:$ por el lema \ref{aproxsimples1}
existe una sucesión $(X^n)_{n\geq 1}$ de procesos simples en
$H^2[t_0,t]$ tal que
\[\int_{t_0}^t|X_s^n-X_s|^2\,ds\xrightarrow[n\to\infty]{c.s.}0.\]
En particular
\[\int_{t_0}^t|X_s^n-X_s|^2\,ds\xrightarrow[n\to\infty]{\mathbf{P}}0.\]
Usando la desigualdad $(a+b)^2\le 2a^2+2b^2$ se obtiene
\[\int_{t_0}^t|X_s^n-X_s^m|^2\,ds\le 2\int_{t_0}^t|X_s^n-X_s|^2\,ds
+2\int_{t_0}^t|X_s-X_s^m|^2\,ds\]
y por lo tanto
\begin{equation}\label{itocauchy1}
\int_{t_0}^t|X_s^n-X_s^m|^2\,ds\xrightarrow[n,m\to\infty]{\mathbf{P}}0.
\end{equation}
Por el lema anterior para todo $\varepsilon>0, \ \rho>0,$
\[\mathbf{P}\biggl(\Bigl|\int_{t_0}^t
(X_s^n-X_s^m)\,dW_s\Bigr|>\varepsilon\biggr)
\le\mathbf{P}\biggl(\int_{t_0}^t|X_s^n-X_s^m|^2\,ds
>\varepsilon^2\rho\biggr)+\rho,\]
y en virtud de (\ref{itocauchy1}) se tiene que
\[\limsup_{n,m\to\infty}\mathbf{P}\biggl(\Bigl|\int_{t_0}^t
X_s^n\,dW_s-\int_{t_0}^t X_s^m\,dW_s\Bigr|>\varepsilon\biggr)
\le\rho.\] Como $\rho>0$ es arbitrario se sigue entonces que la
sucesión $\int_{t_0}^t\!X_s^n\,dW_s, \ n\geq 1$ es de Cauchy en
probabilidad y por lo tanto existe una variable aleatoria real,
que denotaremos por $\int_{t_0}^t X_s\,dW_s,$ tal que
\[\int_{t_0}^t\!X_s^n\,dW_s\xrightarrow[n\to\infty]{\mathbf{P}}\int_{t_0}^t\!X_s\,dW_s.\]

\begin{defi}
A la variable aleatoria $\int_{t_0}^t X_s\,dW_s$ la llamaremos la
\textit{integral estocástica} o \textit{integral de Itô} de $X$
con respecto al movimiento Browniano $(W_s)_{s\geq 0}.$
\end{defi}
La anterior definición es independiente de la sucesión
$(X^n)_{n\geq 1}.$ En efecto: si $(\widetilde{X}^n)_{n\geq 1}$ es
otra sucesión en $H^2[t_0,t]$ que converge a $X$ en el sentido que
\[\int_0^t|\widetilde{X}_s^n-X_s|^2\,ds\xrightarrow[n\to\infty]{c.s.}0\]
entonces la sucesión $(Y^n)_{n\geq 1}$ definida por
\[Y^{2n}=X^n, \ \ Y^{2n+1}=\widetilde{X}^n\] también converge a $X$
en el mismo sentido, y por lo que hemos probado
$\int_{t_0}^t\!Y_s^n\,dW_s$ es una sucesión que converge en
probabilidad. Se sigue entonces que los límites (en probabilidad)
de $\int_{t_0}^t\!X_s^n\,dW_s$ y de
$\int_{t_0}^t\!\widetilde{X}_s^n\,dW_s$ son iguales c.s. (pues son
subsucesiones de una sucesión que converge en probabilidad).

\begin{teorema} Sean $X,\bar{X}$ procesos en $H^2[t_0,t]$
y sean $\alpha,\beta\in\R.$ Entonces
\begin{equation}\label{itolineal2}
\int_{t_0}^t(\alpha X_s+\beta\bar{X}_s)\,dW_s =\alpha\int_{t_0}^t
X_s\,dW_s+\beta\int_{t_0}^t\bar{X}_s\,dW_s.
\end{equation}
\end{teorema}
\begin{proof}
Sean $(X^n)_{n\geq 1}$ y $(\bar{X}^n)_{n\geq 1}$ sucesiones de
procesos simples en $H^2[t_0,t]$ tal que
\[\int_{t_0}^t|X_s^n-X_s|^2\,ds\xrightarrow[n\to\infty]{c.s.}0 \ \ \ \text{y} \ \ \
\int_{t_0}^t|\bar{X}_s^n-\bar{X}_s|^2\,ds\xrightarrow[n\to\infty]{c.s.}0.\]
Aplicando la proposición \ref{itolineal1}
\[\int_{t_0}^t(\alpha X_s^n+\beta\bar{X}_s^n)\,dW_s
=\alpha\int_{t_0}^t
X_s^n\,dW_s+\beta\int_{t_0}^t\bar{X}_s^n\,dW_s, \ \ \forall n\geq
1\] y tomando el límite en probabilidad a cada lado se obtiene
(\ref{itolineal2}).
\end{proof}

\begin{teorema}\label{propitom2}
Si $X$ es un proceso cualquiera en $M^2[t_0,t]$ entonces
\begin{align}
E\int_{t_0}^t X_s\,dW_s&=0,\label{espcero2}\\
E\biggl[\Bigl|\int_{t_0}^t
X_s\,dW_s\Bigr|^2\biggr]&=E\int_{t_0}^t|X_s|^2\,ds.\label{isometria2}
\end{align}
\end{teorema}
La última igualdad suele llamarse \textbf{Isometría de
Itô.}
\begin{proof} Por el lema \ref{aproxsimples2} existe una
sucesión de procesos simples $(X^n)_{n\geq 1}$ en $M^2[t_0,t]$ tal
que
\[E\int_{t_0}^t|X_s^n-X_s|^2\,ds\xrightarrow[n\to\infty]\,0.\]
Por la proposición \ref{propitosimples}
\[E\biggl[\Bigl|\int_{t_0}^t
X_s^n\,dW_s-\int_{t_0}^t
X_s^m\,dW_s\Bigr|^2\biggr]=E\int_{t_0}^t|X_s^n-X_s^m|^2\,ds\xrightarrow[n,m\to\infty]\,0.\]
Entonces $\int_{t_0}^t X_s^n\,dW_s$ es una sucesión de Cauchy en
$L^2$ y por lo tanto converge a un límite, que debe ser igual c.s.
a $\int_{t_0}^t X_s\,dW_s,$ debido a que la convergencia en $L^2$
implica la convergencia en probabilidad. En particular
\[E\int_{t_0}^t X_s\,dW_s=\lim_{n\to\infty}\,E\!\int_{t_0}^t\!
X_s^n\,dW_s=0,\] y por continuidad de la norma
\[E\biggl[\Bigl|\int_{t_0}^t
X_s\,dW_s\Bigr|^2\biggr]=\lim_{n\to\infty}\,E\biggl[\Bigl|\int_{t_0}^t
X_s^n\,dW_s\Bigr|^2\biggr]=\lim_{n\to\infty}\,E\!\int_{t_0}^t|X_s|^2\,ds
=E\int_{t_0}^t|X_s|^2\,ds\]
\end{proof}

\begin{teorema}\label{teoremaprob}
Si $X\in H^2[t_0,t]$ entonces para todo $\varepsilon>0,\, N>0,$
\[\mathbf{P}\biggl(\Bigl|\int_{t_0}^t
X_s\,dW_s\Bigr|>\varepsilon\biggr)
\le\mathbf{P}\biggl(\int_{t_0}^t|X_s|^2\,ds>N\biggr)+\frac{N}{\varepsilon^2}.\]
\end{teorema}
\begin{proof}
Dado que en un espacio de probabilidad la convergencia c.s.
implica la convergencia en probabilidad, por el lema
\ref{aproxlema} existe una sucesión de procesos simples
$(X^n)_{n\geq 1}$ en $H^2[t_0,t]$ tal que
\begin{equation}\label{simplesprob}
\int_{t_0}^t|X_s^n-X_s|^2\,ds\xrightarrow[n\to\infty]{\mathbf{P}}0
\end{equation}
y por definición de la integral estocástica
\begin{equation}\label{simplesprob1}
\int_{t_0}^t\!X_s^n\,dW_s\xrightarrow[n\to\infty]{\mathbf{P}}\int_{t_0}^t\!X_s\,dW_s.
\end{equation}
Aplicando el lema \ref{lemadef} a cada $X^n$ tenemos
\[\mathbf{P}\biggl(\Bigl|\int_{t_0}^t
X_s^n\,dW_s\Bigr|>\varepsilon'\biggr)
\le\mathbf{P}\biggl(\int_{t_0}^t|X_s^n|^2\,ds>N'\biggr)+\frac{N'}{(\varepsilon')^2}.\]
Tomando el límite cuando $n\to\infty$ y usando
(\ref{simplesprob}), (\ref{simplesprob1}) se obtiene
\[\mathbf{P}\biggl(\Bigl|\int_{t_0}^t
X_s\,dW_s\Bigr|>\varepsilon\biggr)
\le\mathbf{P}\biggl(\int_{t_0}^t|X_s|^2\,ds>N\biggr)+\frac{N'}{(\varepsilon')^2}.\]
para todo $\varepsilon'>\varepsilon, \ N<N'.$ Haciendo
$\varepsilon'\uparrow\varepsilon$ y $N'\downarrow N$ se obtiene lo
requerido.
\end{proof}

\begin{teorema}\label{aproxprob}
Sea $X\in H^2[t_0,t]$ y sea $(X^n)_{n\geq 1}$ una sucesión en
$H^2[t_0,t]$ tal que
\[\int_{t_0}^t|X_s^n-X_s|^2\,ds\xrightarrow[n\to\infty]{\mathbf{P}}0,\]
entonces
\[\int_{t_0}^t\!X_s^n\,dW_s\xrightarrow[n\to\infty]{\mathbf{P}}\int_{t_0}^t\!X_s\,dW_s.\]
(note que los procesos $X^n$ no son necesariamente procesos
simples).
\end{teorema}
\begin{proof}
Por el teorema anterior, para todo $\varepsilon>0, \ \rho>0,$
\[\mathbf{P}\biggl(\Bigl|\int_{t_0}^t
(X_s^n-X_s)\,dW_s\Bigr|>\varepsilon\biggr)
\le\mathbf{P}\biggl(\int_{t_0}^t|X_s^n-X_s|^2\,ds>\varepsilon^2\rho\biggr)+\rho.\]
Haciendo $n\to\infty$ y usando la hipótesis y el hecho de que
$\rho>0$ es arbitrario se obtiene el resultado.
\end{proof}

\begin{teorema}\label{aproxsumas}
Sea $X\in H^2[t_0,t]$ con trayectorias continuas c.s., y sea
$(\pi_n)_{n\geq 1}$ una sucesión de particiones
\[t_0=t_{n,0}<t_{n,1}<\cdots<t_{n,m_n}=t\]
del intervalo $[t_0,t],$ donde $m_n$ el número de puntos en la
partición $\pi_n,$ tal que $|\!|\pi_n|\!|\xrightarrow[n\to\infty]\,0.$
Entonces
\[\sum_{k=0}^{m_n-1}X_{t_{n,k}}(W_{t_{n,k+1}}
-W_{t_{n,k}})\xrightarrow[n\to\infty]{\mathbf{P}}\int_{t_0}^t
X_s\,dW_s.\]
\end{teorema}
\begin{proof}
Introduzcamos la sucesión de procesos escalonados $(X^n)_{n\geq
1}$ definidos por
\[X_t^n:=X_{t_{n,k}}, \ \ \ \text{si} \ t\in [t_{n,k},t_{n,k+1}), \ \ 0\le k\le m_n-1.\]
Entonces $X_t^n\xrightarrow[n\to\infty]\,X_t$ uniformemente en
$t\in[t_0,t)$ c.s. Por lo tanto
\[\int_{t_0}^t|X_s^n-X_s|^2\,ds\xrightarrow[n\to\infty]{\mathbf{c.s}}0,\]
y por el teorema \ref{aproxprob}
\[\sum_{k=0}^{m_n-1}X_{t_{n,k}}(W_{t_{n,k+1}}-W_{t_{n,k}})
=\int_{t_0}^t\!X_s^n\,dW_s\xrightarrow[n\to\infty]{\mathbf{P}}\int_{t_0}^t\!X_s\,dW_s.\]
\end{proof}

\begin{ejem}\label{ejemplo1}
Sea $(\pi_n)_{n\geq 1}$ una sucesión de particiones
\[0=t_{n,0}<t_{n,1}<\cdots<t_{n,n}=t\]
de $[0,t]$ tal que $|\!|\pi_n|\!|\xrightarrow[n\to\infty]\,0.$ Usando
la identidad $a(b-a)=\frac{1}{2}(b^2-a^2-(b-a)^2)$ y el teorema
\ref{aproxsumas} se sigue que
\begin{align*}
\int_{t_0}^t
W_s\,dW_s&=\lim_{n\to\infty}\sum_{k=0}^{n-1}W_{t_{n,k}}(W_{t_{n,k+1}}
-W_{t_{n,k}})\\
&=\tfrac{1}{2}\lim_{n\to\infty}\sum_{k=0}^{n-1}(W_{t_{n,k+1}}^2-
W_{t_{n,k}}^2-(W_{t_{n,k+1}}-W_{t_{n,k}})^2)\\
&=\tfrac{1}{2}W_t^2-\tfrac{1}{2}W_0^2
-\tfrac{1}{2}\lim_{n\to\infty}\sum_{k=0}^{n-1}(W_{t_{n,k+1}}-W_{t_{n,k}})^2
\end{align*}
donde $\lim_{n\to\infty}$ es tomado como el límite en
probabilidad. Por el teorema \ref{v2browniano} el último límite
converge en $L^2,$ y por lo tanto en probabilidad, a $t.$ En
conclusión
\[\int_0^tW_s\,dW_s=\tfrac{1}{2}W_t^2-\tfrac{1}{2}t.\]
\end{ejem}

\subsubsection{La integral de Itô como proceso estocástico}
Sea $X\in H^2[0,T]$ y consideremos el proceso estocástico
\[I(t):=
\begin{cases}
\,0, \ \ \ & \text{si} \ t=0,\\
\int_0^t X_s\,dW_s, \ \ \ & \text{si} \ 0<t\le T.
\end{cases}\]
Note que $I(t)$ es $\F_t-$medible.

\begin{lema}
Si $\xi$ y $\xi_1,\xi_2,\ldots$ son variables aleatorias cuadrado
integrables tales que $\xi_{n}\rightarrow \xi$ en $L^{2}$ cuando
$n\rightarrow\infty,$ entonces
\[E[\xi_{n}\mid \mathcal{G}]\xrightarrow[n\to\infty]{L^2}E[\xi\mid\mathcal{G}]\]
para cualquier $\sigma$-álgebra $\mathcal{G}$ sobre $\Omega$
contenida en $\mathcal{F}$.
\end{lema}
\begin{proof}
Por la desigualdad de Jensen
\[\Bigl|E[\xi_{n}\mid\mathcal{G}]-E[\xi\mid\mathcal{G}]\Bigr|^2
=\bigl|E[\xi_{n}-\xi\mid \mathcal{G}]\bigr|^2 \le
E\bigl[\left|\xi_{n}-\xi\right|^2\mid\mathcal{G}\bigr]\] lo cual
implica que
\begin{align*}
E\Bigl[\bigl|E[\xi_{n}\mid \mathcal{G}]-E[\xi\mid
\mathcal{G}]\bigr| ^{2}\Bigr]&\leq
E\Bigl[E\bigl[\left|\xi_{n}-\xi\right|^2\mid \mathcal{G}\bigl]\Bigr]  \\
&=E\left[\left|\xi_{n}-\xi\right|^2\right]\xrightarrow[n\to\infty]\,0.
\end{align*}
\end{proof}

\begin{teorema}\label{propmartingala}
Si $X\in M^2[0,T]$ entonces el proceso $I(t), t\in [0,T],$ es una
martingala con respecto a $\{\F_t\}_{t\in [0,T]}.$
\end{teorema}
\begin{proof}
Sean $0\leq s<t\le T$ y supongamos primero que $X$ es escalonado
de la forma $X_s=\sum_{j=0}^{n-1}e_{j}1_{[t_{j,}t_{j+1})}(s)$
donde $0=t_{0}<t_{1}<\cdots<t_{k}=s<t_{k+1}<\cdots<t_{n}=t.$ Tal
partición siempre puede ser obtenida adicionando el punto $s$ si
es necesario. Entonces \[E\left[\int_0^t
X_r\,dW_r\Bigl|\F_s\right]=\sum_{j=0} ^{n-1}E\left[\left.
e_{j}(W_{t_{j+1}}-W_{t_{j}})\right|\F_s\right].\]
Si $j<k$ entonces $e_{j}$ y $W_{t_{j+1}}-W_{t_{j}}$ son $\mathcal{F}_{s}%
$-medibles y
\[
E\left[\left. e_{j}(W_{t_{j+1}}-W_{t_{j}})\right| \F_s\right]
=e_{j}(W_{t_{j+1}}-W_{t_{j}}),\] si $j\geq k$ entonces
$\F_s\subseteq\F_{t_{j}}$ y
\begin{align*}
E\left[
e_{j}(W_{t_{j+1}}-W_{t_{j}})\mid\F_s\right]&=E\left[E[e_{j}(W_{t_{j+1}}-W_{t_{j}})\mid
\F_{t_{j}}]\bigl|\F_s\right]\\
& =E\left[e_{j}E[W_{t_{j+1}}-W_{t_{j}}\mid\F_{t_{j}
}]\bigl|\F_{s}\right]  \\
& =E[e_{j}|\F_s]\cdot E[W_{t_{j+1}}-W_{t_{j}}]=0
\end{align*}
pues $e_{j}$ es $\F_{t_j}-$medible y $W_{t_{j+1}}-W_{t_{j}}$ es
independiente de $\mathcal{F}_{t_{j}}.$ De esto se tiene que
\[E\left[\int_0^t
X_r\,dW_r\Bigl|\F_s\right]=\sum_{j=0}^{k-1}%
e_{j}(W_{t_{j+1}}-W_{t_{j}})=\int_{0}^{s}X_r\,dW_{r}.\] Para $X\in
M^2[0,T]$ arbitrario, gracias al lema \ref{aproxsimples2} existe
una sucesión de procesos simples $(X^n)_{n\geq 1}$ en $M^2[0,T]$
tal que
\[E\int_0^T|X_r^n-X_r|^2\,dr\xrightarrow[n\to\infty]\,0,\]
y como vimos en la demostración del teorema \ref{propitom2},
\[\int_0^t X_r^n\,dW_r\xrightarrow[n\to\infty]{L^2}\int_0^t X_r\,dW_r\]
y por el lema anterior
\[\int_0^s X_r^n\,dW_r=E\left[\int_0^t
X_r^n\,dW_r\Bigl|\F_s\right]\xrightarrow[n\to\infty]{L^2}E\left[\int_0^t
X_r\,dW_r\Bigl|\F_s\right].\] Por unicidad del límite en $L^2$ se
obtiene
\[E\left[\int_0^t X_r\,dW_r\Bigl|\F_s\right]=\int_{0}^{s}X_r\,dW_{r} \ \ \ \ \text{c.s.}\]
\end{proof}

\begin{lema}\label{lemacasitodo}
Sean $X,\widetilde{X}\in H^2[0,T]$ tales que
$X_t(\omega)=\widetilde{X}_t(\omega)$ para todo
$(t,\omega)\in[0,T]\times\Omega_0,$ $\Omega_0\subset\Omega.$
Entonces
\begin{equation}\label{lemacont}
\int_0^t X_s\,dW_s=\int_0^t\widetilde{X}_s\,dW_s, \ \ \ \text{para
casi todo} \ \omega\in\Omega_0,
\end{equation}
y para todo $t\in[0,T].$
\end{lema}
\begin{proof}
Sean $(X^n)_{n\geq 1}, \ (\widetilde{X}^n)_{n\geq 1}$ sucesiones
de procesos simples en $H^2[0,T]$ tales que
\[\int_0^t|X_s^n-X_s|^2\,ds\xrightarrow[n\to\infty]{c.s.}0, \ \ \ \
\int_0^t|\widetilde{X}_s^n-\widetilde{X}_s|^2\,ds\xrightarrow[n\to\infty]{c.s.}0.\]
De la hipótesis podemos escoger dichas sucesiones tales que si
$\omega\in\Omega_0$ entonces {}
$X_t^n(\omega)=\widetilde{X}_t^n(\omega)$ para todo $t\in[0,T].$
Por lo tanto, de la definición de la integral de Itô para procesos
escalonados tenemos
\[\int_0^t X_s^n\,dW_s=\int_0^t\widetilde{X}_s^n\,dW_s, \ \ \ \text{en} \ \Omega_0,\]
para todo $t\in[0,T].$ Tomando el límite cuando $n\to\infty$ se
sigue (\ref{lemacont}).
\end{proof}

\begin{teorema}\label{teoremacont}
Si $X\in H^2[0,T]$ entonces el proceso $I(t), \ t\in [0,T],$
tiene una versión con trayectorias continuas.
\end{teorema}
\begin{proof}
\textbf{Paso 1.} Suponga primero que $X\in M^2[0,T].$ Sea
$(X^n)_{n\geq 1}$ una sucesión de procesos simples en $M^2[0,T]$
tal que
\[E\int_0^T|X_r^n-X_r|^2\,dr\xrightarrow[n\to\infty]\,0.\]
Por la continuidad de $(W_t)_{t\geq 0}$ los procesos
\[I_n(t):=\int_0^t X_s^n\,dW_s, \ \ \ t\in [0,T],\]
tienen trayectorias continuas c.s. y por el teorema
\ref{propmartingala} $I_n-I_m$ es una martingala para cada $n,m.$
Usando la primera desigualdad de Doob para martingalas tenemos
\begin{align*}
\mathbf{P}\Bigl(\sup_{\ t\in [0,T]}|I_n(t)-I_m(t)|>\varepsilon\Bigr)
&\le\frac{1}{\varepsilon^2}E\biggl[\Bigl|\int_0^T[I_n(t)-I_m(t)]\,dW_t\Bigr|^2\biggr]\\
&=\frac{1}{\varepsilon^2}\,E\!\int_0^T|I_n(t)-I_m(t)|^2\,dt\xrightarrow[n,m\to\infty]\,0.
\end{align*}
Tomando $\varepsilon=1/2^k$ se puede escoger una sucesión
$(n_k)_{k\geq 1}$ tal que
\[\mathbf{P}\Bigl(\sup_{\ t\in [0,T]}|I_{n_k}(t)-I_{n_{k+1}}(t)|>\frac{1}{2^k}\Bigr)<\frac{1}{k^2}, \
\ \ \ k\geq 1.\] Dado que $\sum_{k=0}^\infty k^{-2}<+\infty,$ el
lema de Borel-Cantelli implica que
\[\mathbf{P}\biggl(\limsup_{k\to\infty}\,\Bigl(\sup_{\ t\in [0,T]}|I_{n_k}(t)-I_{n_{k+1}}(t)|
>\frac{1}{2^k}\Bigr)\biggr)=0\] o en forma equivalente
\[\mathbf{P}\biggl(\liminf_{k\to\infty}\,\Bigl(\sup_{\ t\in [0,T]}|I_{n_k}(t)-I_{n_{k+1}}(t)|
\le\frac{1}{2^k}\Bigr)\biggr)=1\] Es decir que para casi todo
$\omega\in\Omega$ existe un $k_0(\omega)\geq 1$ suficientemente
grande tal que
\[|I_{n_k}(t)-I_{n_{k+1}}(t)|\le\frac{1}{2^k}, \ \ \forall t\in[0,T], \ \
\text{si} \ k\geq k_0,\]
y por lo tanto, con probabilidad 1, $I_{n_k}(\cdot)$ converge
uniformemente en $[0,T].$ El límite $J(t)$ es por lo tanto
continuo en $t\in[0,T]$ para casi todo $\omega\in\Omega.$ Dado que
\[\int_0^t X_r^n\,dW_r\xrightarrow[n\to\infty]{L^2}\int_0^t X_r\,dW_r\]
por unicidad del límite se sigue que $I(t)=J(t)$ c.s. para todo
$t\in[0,T].$

{}{}\textbf{Paso 2.} Considere ahora el caso general $X\in
H^2[0,T].$ Para $N>0$ definimos el proceso
$X^N_t:=X_t\cdot\mathbf{1}_{[0,N]}\left(\int_0^t
|X_s|^2\,ds\right).$ Entonces $X^N\in M^2[0,T]$ y por la primer
parte de la demostración el proceso
\[J_N(t):=\int_0^t X_s^N\,dW_s, \ \ \ t\in [0,T],\] posee una
versión continua. Sea
$\Omega_N:=\bigl\{\omega\in\Omega:\int_0^T|X_s(\omega)|^2\,ds<N\bigr\}.$
Si $M>N$ entonces $X^N_t(\omega)=X_t^M(\omega)$ para todo
$(t,\omega)\in[0,T]\times\Omega_N.$ Por el lema anterior se sigue
que $J_N(t)=J_M(t), \ \forall t\in[0,T]$ y para casi todo
$\omega\in\Omega_N$ si $N<M,$ por lo tanto
\[\widetilde{J}(t):=\lim_{M\to\infty}J_M(t)\]
tiene trayectorias continuas para casi todo $\omega\in\Omega_N.$
Dado que $\Omega_N\uparrow\Omega$ entonces
$\mathbf{P}(\Omega_N)\uparrow 1$ si $N\to\infty,$ y así
$\widetilde{J}(t), \ t\in [0,T],$ es un proceso con trayectorias
continuas c.s. Dado que para cada $t\in(0,T],$
\[\mathbf{P}\biggl(\int_{t_0}^t
|X_s-X_s^M|^2\,ds>0\biggr)=\mathbf{P}\biggl(\int_{t_0}^t
|X_s|^2\,ds>M\biggr)\xrightarrow[M\to\infty]\,0,\] tenemos, por el
teorema \ref{aproxprob}, que
\[J_M(t)\xrightarrow[M\to\infty]{\mathbf{P}}\int_0^t X_s\,dW_s=I(t).\]
En consecuencia, $I(t)$ posee la versión continua
$\widetilde{J}(t).$
\end{proof}

\begin{obse}
De ahora en adelante siempre nos referiremos a la versión continua
de la integral estocástica.
\end{obse}

\begin{teorema}
Sea $X\in H^2[0,T].$ Entonces para todo $\varepsilon>0, \ N>0,$
\[\mathbf{P}\biggl(\sup_{\, t\in [0,T]}\Bigl|\int_0^t
X_s\,dW_s\Bigr|>\varepsilon\biggr)
\le\mathbf{P}\biggl(\int_0^t|X_s|^2\,ds>N\biggr)+\frac{N}{\varepsilon^2}.\]
\end{teorema}
\begin{proof}
Con la misma notación del paso 2 del teorema anterior tenemos
\begin{align*}
\mathbf{P}&\biggl(\sup_{\, t\in [0,T]}\Bigl|\int_0^t
X_s\,dW_s\Bigr|>\varepsilon\biggr)\\
&\le\mathbf{P}\biggl(\sup_{\, t\in [0,T]}\Bigl|\int_0^t
X_s\,dW_s-\int_0^t
X_s^N\,dW_s\Bigr|>0\biggr)+\mathbf{P}\biggl(\sup_{\, t\in [0,T]}\Bigl|\int_0^t X_s^N\,dW_s\Bigr|>\varepsilon\biggr).
\end{align*}
Dado que $J_N(t), \, t\in [0,T],$ es una martingala continua,
usando la primera desigualdad de Doob para martingalas se obtiene
\begin{equation}\label{probsupint1}
\mathbf{P}\biggl(\sup_{\, t\in [0,T]}\Bigl|\int_0^t
X_s^N\,dW_s\Bigr|>\varepsilon\biggr)\le\frac{1}{\varepsilon^2}E\biggl[\Bigl|\int_0^T
X_s^N\,dW_s\Bigr|^2\biggr]\le\frac{1}{\varepsilon^2}\,E\!\int_0^T|X_s^N|^2\,ds\le
\frac{N}{\varepsilon^2}.
\end{equation}
Ahora, sobre el conjunto
$\Omega_N=\bigl\{\omega\in\Omega:\int_0^T|X_s(\omega)|^2\,ds\le
N\bigr\}$ tenemos que $X_t=X_t^N$ para $t\in [0,T].$ Podemos
entonces aplicar el lema \ref{lemacasitodo} y obtener
\[\int_0^t X_s\,dW_s=\int_0^t X_s^N\,dW_s, \ \ \ \text{para casi todo} \
\omega\in\Omega_N, \ \ \forall t\in [0,T].\]
Dado que ambas integrales son procesos continuos, la última
relación se tiene para todo $t\in [0,T]$ y para todo
$\omega\in\Omega'_N\subseteq\Omega_N$ con
$\mathbf{P}(\Omega_N\setminus\Omega'_N)=0,$ es decir
\begin{equation}\label{probsupint2}
\mathbf{P}\biggl(\sup_{\, t\in [0,T]}\Bigl|\int_0^t
X_s\,dW_s-\int_0^t X_s^N\,dW_s\Bigr|>0\biggr)\le
\mathbf{P}(\Omega\setminus\Omega_N)=\mathbf{P}\biggl(\int_0^t|X_s|^2\,ds>N\biggr).
\end{equation}
Combinando las desigualdades (\ref{probsupint1}) y
(\ref{probsupint2}) se sigue el resultado.
\end{proof}

{}{} Una consecuencia inmediata del anterior teorema es el
\begin{coro}\label{corosup}
Sea $X\in H^2[0,T]$ y sea $(X^n)_{n\geq 1}$ una sucesión en
$H^2[0,T]$ tal que
\[\int_0^T|X_s^n-X_s|^2\,ds\xrightarrow[n\to\infty]{\mathbf{P}}0,\]
entonces
\[\sup_{t\in [0,T]}\left|\int_0^t\!X_s^n\,dW_s
-\int_0^t\!X_s\,dW_s\right|\xrightarrow[n\to\infty]{\mathbf{P}}0.\]
\end{coro}

{} Ahora que sabemos que la integral estocástica vista como
función del límite superior es una martingala continua cuando el
integrando está en $M^2[0,T],$ podemos usar las desigualdades de
Doob para obtener las siguientes estimaciones que serán de mucha
importancia en los siguientes capítulos
\begin{teorema}\label{estimacionint}
Si $X\in M^2[0,T]$ entonces
\begin{align}
E\biggl[\sup_{\, t\in [0,T]}\Bigl|\int_0^t
X_s\,dW_s\Bigr|\biggr]&\le 3\,E\Biggl[\biggl(\int_0^T
|X_s|^2\,ds\biggr)^{1/2}\Biggr]\label{estimacionito1}\\[0.1cm]
E\biggl[\sup_{\, t\in [0,T]}\Bigl|\int_0^t
X_s\,dW_s\Bigr|^2\biggr]&\le
4\,E\!\int_0^T\!|X_s|^2\,ds\label{estimacionito2}
\end{align}
\end{teorema}
\begin{proof}
La desigualdad (\ref{estimacionito2}) se sigue directamente de la
segunda desigualdad de Doob y de la isometría de Itô. Para la
primera desigualdad considere los tiempos de parada
\[\tau_c:=\inf\Bigl\{t\geq 0:\int_0^t\!|X_s|^2\,ds\geq
c^2\Bigr\}, \ \ c>0.\] Aplicando la primera desigualdad de Doob a
la martingala $\int_0^{t\wedge\tau_c}X_s\,dW_s, \ t\in [0,T],$
\begin{align*}
\mathbf{P}&\biggl(\sup_{\ t\in [0,T]\wedge\tau_c}\Bigl|\int_0^t
X_s\,dW_s\Bigr|^2\geq c^2\biggr)=\mathbf{P}\biggl(\sup_{\, 0\le
t\le
T}\Bigl|\int_0^{t\wedge\tau_c}X_s\,dW_s\Bigr|^2\geq
c^2\biggr)\\[0.2cm]
&\le\frac{1}{c^2}\,E\biggl[\Bigl|\int_0^{T\wedge\tau_c}\!|X_s|^2\,ds\Bigr|^2\biggr]
=\frac{1}{c^2}\,E\!\int_0^{T\wedge\tau_c}\!|X_s|^2\,ds
\le\frac{1}{c^2}E\left[c^2\wedge\left(\textstyle\int_0^T|X_s|^2\,ds\right)\right],
\end{align*}
lo que implica
\begin{align}
\mathbf{P}\biggl(\sup_{\, t\in [0,T]}&\Bigl|\int_0^t
X_s\,dW_s\Bigr|>c\biggr)=\mathbf{P}\biggl(\sup_{\, t\in [0,T]}\Bigl|\int_0^t X_s\,dW_s\Bigr|^2>c^2\biggr)\notag\\
&\le\mathbf{P}(\tau_c<T)+\mathbf{P}\biggl(\sup_{\, t\in [0,T]\wedge\tau_c}\Bigl|\int_0^t X_s\,dW_s\Bigr|>c^2\biggr)\notag\\
&\le\mathbf{P}\biggl(\int_0^T\!|X_s|^2\,ds\geq c^2\biggr)+
\frac{1}{c^2}\,%
E\left[c^2\wedge\left(\textstyle\int_0^T|X_s|^2\,ds\right)\right].\label{estimacionito3}
\end{align}
Si $\xi$ es una variable aleatoria real positiva, por el teorema
de Fubini
\[\int_0^{+\infty}\mathbf{P}(\xi\geq
c)\,dc=\int_0^{+\infty}\int_{\Omega}\mathbf{1}_{\{\xi\geq
c\}}\,d\mathbf{P}\,dc=\int_{\Omega}\int_0^{+\infty}
\mathbf{1}_{[0,\xi]}(c)\,dc\,d\mathbf{P}=\int_{\Omega}\xi\,d\mathbf{P}=E(\xi)\]
y
\begin{align*}
\int_0^{+\infty}\frac{1}{c^2}E[c^2\wedge\xi^2]\,dc
&=\int_0^{+\infty}\int_{\Omega}\frac{1}{c^2}(c^2\wedge\xi^2)\,d\mathbf{P}\,dc\\
&=\int_0^{+\infty}\int_{\Omega}\frac{1}{c^2}(c^2\mathbf{1}_{\{\xi\geq
c\}}+\xi^2\mathbf{1}_{\{\xi<c\}})\,d\mathbf{P}\,dc\\
&=\int_{\Omega}\int_0^{+\infty}
\mathbf{1}_{[0,\xi]}(c)\,dc\,d\mathbf{P}+\int_{\Omega}\xi^2\int_0^{+\infty}
\frac{1}{c^2}\mathbf{1}_{(\xi,\infty)}(c)\,dc\,d\mathbf{P}\\
&=E(\xi)+\int_{\Omega}\xi^2\int_\xi^{+\infty}
\frac{1}{c^2}\,dc\,d\mathbf{P}\\
&=E(\xi)+\int_{\Omega}\xi^2\frac{1}{\xi}\,d\mathbf{P}=2\,E(\xi).
\end{align*}
Haciendo $\xi^2=\int_0^T\!|X_s|^2\,ds$ y tomando la integral entre
$0$ y $+\infty$ en (\ref{estimacionito3}) se obtiene
(\ref{estimacionito1}.)
\end{proof}

Si $\tau$ es una variable real positiva, la integral $\int_0^\tau
X_s\,dW_s$ se entenderá como la variable aleatoria $I(\tau)$ donde
$I(t)=\int_0^t X_s\,dW_s.$ Podemos entonces preguntarnos: si
$\tau\le T$ c.s., en qué casos se tendrá
\[\int_0^\tau X_s\,dW_s=\int_0^T X_s\mathbf{1}_{[0,\tau)}(s)\,dW_s \ \ \text{?}\]
Para que la integral estocástica de la derecha tenga sentido, el
proceso ${1}_{[0,\tau)}(s), \ s\geq 0,$ debe ser adaptado, lo cual
sugiere que $\tau$ sea por lo menos un tiempo de parada con
respecto a $\{\F_t\}_{t\geq 0}.$ Necesitamos primero el siguiente
lema

\begin{lema}\label{lemazeta}
Sea $X\in H^2[t_0,t]$ y sea $\zeta$ una variable aleatoria real
acotada y $\F_{t_0}-$medible. Entonces $\zeta X\in H^2[t_0,t]$ y
\[\int_{t_0}^t\zeta X_s\,dW_s=\zeta\int_{t_0}^t X_s\,dW_s.\]
\end{lema}
\begin{proof}
Si $X$ es un proceso escalonado de la forma
$X_s=\sum_{i=0}^{n-1}e_i\mathbf{1}_{[t_i,t_{i+1})}(s)$ con
$t_0<t_1<\cdots<t_n=t$ y donde cada $e_i$ es $\F_{t_i}-$medible,
entonces $\zeta e_i$ es también $\F_{t_i}-$medible para $1\le i\le
n$ y
\[\zeta\int_{t_0}^t X_s\,dW_s=\zeta\sum_{i=0}^{n-1}e_i\bigl(W_{t_{i+1}}-W_{t_i}\bigr)=
\sum_{i=0}^{n-1}\zeta
e_i\bigl(W_{t_{i+1}}-W_{t_i}\bigr)=\int_{t_0}^t\zeta X_s\,dW_s.\]
Aproximando con procesos escalonados se obtiene fácilmente la
igualdad para el caso general.
\end{proof}

\begin{prop}\label{intparada}
Sea $X\in H^2[0,T]$ y sea $\tau$ un $\F_t-$tiempo de parada tal
que $\tau\le T$ c.s., entonces
\[\int_0^\tau X_s\,dW_s=\int_0^T X_s\mathbf{1}_{[0,\tau)}(s)\,dW_s.\]
\end{prop}
\begin{proof}
Suponga que $\tau$ es un tiempo de parada de la forma
$\tau=\sum_{i=1}^n \mathbf{1}_{A_i}t_i,$ donde $\bigcup_{i=1}^n
A_i=\Omega, \ A_i\cap A_j=\varnothing$ para $i\ne j$ y cada
$A_i\in\F_{t_i}.$ Entonces
\[\int_0^T X_s\mathbf{1}_{\{s\geq\tau\}}\,ds
=\int_0^T \biggl(\,
\sum_{i=1}^n\mathbf{1}_{A_i}\mathbf{1}_{\{s\geq
t_i\}}\biggr)X_s\,dW_s=\sum_{i=1}^n\int_{t_i}^T\mathbf{1}_{A_i}X_s\,dW_s\]
y por el lema anterior, debido a que cada $\mathbf{1}_{A_i}$ es
$\F_{t_i}-$medible
\begin{align*}
\int_0^T X_s\mathbf{1}_{\{s\geq\tau\}}\,ds
&=\sum_{i=1}^n\mathbf{1}_{A_i}\int_{t_i}^T X_s\,dW_s
=\sum_{i=1}^n\mathbf{1}_{A_i}\bigl[I(T)-I(t_i)\bigr]\\
&=I(T)-\sum_{i=1}^n\mathbf{1}_{A_i}I(t_i)=I(T)-I\Bigl(\,
\textstyle{\sum_{i=1}^n\mathbf{1}_{A_i}t_i}\Bigr)=\displaystyle\int_{\tau}^T
X_s\,dBs,
\end{align*}
es decir
\[\int_0^\tau X_s\,dW_s=\int_0^T
X_s\mathbf{1}_{\{s<\tau\}}\,dW_s=\int_0^T
X_s\mathbf{1}_{[0,\tau)}(s)\,dW_s.\] Ahora, un tiempo de parada
$\tau$ cualquiera puede ser aproximado por la sucesión decreciente
de tiempos de parada
\[\tau_n=\sum_{k=0}^{2^n}\tfrac{(k+1)T}{2^n}\,\mathbf{1}_{A_{n,k}}\]
donde
$A_{n,k}=\Bigl\{\omega\in\Omega:\frac{kT}{2^n}\le\tau(\omega)<\frac{(k+1)T}{2^n}\Bigr\},
\ k=0,1,\ldots,2^n, \ n\geq 1,$ de tal forma que
$\tau_n\searrow\tau$ c.s. En consecuencia, por la continuidad en
el límite superior de la integral estocástica
\[\lim_{n\to\infty}\int_0^{\tau_n}X_s\,dW_s=\int_0^\tau X_s\,dW_s \ \ \ \text{c.s.}\]
Además, dado que $|X_s|^2\mathbf{1}_{[\tau,\tau_n)}(s)\le|X_s|^2$
y
$|X_s|^2\mathbf{1}_{[\tau,\tau_n)}(s)\xrightarrow[n\to\infty]\,0,
\forall s\in[0,T],$ por convergencia dominada
\[\int_0^T|X_s|^2\mathbf{1}_{[\tau,\tau_n)}(s)\xrightarrow[n\to\infty]{c.s.}0\]
y en particular
$\int_0^T|X_s|^2\mathbf{1}_{[\tau,\tau_n)}(s)\xrightarrow[n\to\infty]{\mathbf{P}}0.$
Usando el teorema \ref{teoremaprob}, para todo $\varepsilon>0, \
\rho>0,$
\begin{align}
\limsup_{n\to\infty}\mathbf{P}&\biggl(\Bigl|\int_0^T
X_s\mathbf{1}_{[0,\tau_n)}(s)\,dW_s -\int_0^T
X_s\mathbf{1}_{[0,\tau)}(s)\,dW_s\Bigr|>\varepsilon\biggr)\\
&\le\rho+\lim_{n\to\infty}\mathbf{P}\biggl(\int_0^T|X_s|^2\mathbf{1}_{[\tau,\tau_n)}(s)\,ds
>\rho^2\varepsilon\biggr)=\rho
\end{align}
y como $\rho>0$ es arbitrario,
\[\int_0^T X_s\mathbf{1}_{[0,\tau_n)}(s)\,dW_s\xrightarrow[n\to\infty]{\mathbf{P}}
\int_0^T X_s\mathbf{1}_{[0,\tau)}(s)\,dW_s.\] Podemos entonces
tomar una subsucesión (que también denotaremos por $\tau_n)$ tal
que
\[\int_0^{\tau_n}X_s\,dW_s
=\int_0^T
H_s\mathbf{1}_{[0,\tau_n)}(s)\,dW_s\xrightarrow[n\to\infty]{c.s.}\int_0^T
X_s\mathbf{1}_{[0,\tau)}(s)\,dW_s\] y por unicidad del límite c.s.
se tiene el resultado.
\end{proof}

\begin{teorema}\label{intiguales}
Sean $X$ y $K$ procesos reales tales que $X,K\in M^2[0,T]$ para
todo $T>0,$ y sea $\tau$ una variable aleatoria positiva tal que
$X_s=K_s,$ si $s<\tau.$ Entonces
\[\int_0^t X_s\,dW_s=\int_0^t K_s\,dW_s, \ \ \ \text{c.s. \ si} \ t<\tau.\]
\end{teorema}
\begin{proof}
Sea $T>0$ fijo y sean $X^n,K^n$ dos sucesiones de procesos
escalonados tales que
\[E\int_0^T |X_s^n-X_s|^2\,ds\xrightarrow[n\to\infty]\,0, \ \ \ \
\ E\int_0^T|K_s^n-K_s|^2\,ds\xrightarrow[n\to\infty]\,0.\] Gracias
a la hipótesis dichas sucesiones se pueden escoger tales que
$X_s^n=K_s^n,$ si $s<\tau\wedge T;$ usando esto en la definición
de integral estocástica para procesos escalonados se sigue
fácilmente que
\[\int_0^t X_s^n\,dW_s=\int_0^t K_s^n\,dW_s, \ \ \
\text{c.s. \ si} \ t<\tau\wedge T.\] Por el corolario
\ref{corosup}
\[\sup_{t\in [0,T]}\left|\int_0^t\!X_s^n\,dW_s
-\int_0^t\!X_s\,dW_s\right|\xrightarrow[n\to\infty]{\mathbf{P}}0.\]
luego existe una subsucesión, que seguiremos notando por
$(X^n)_{n\geq 1},$ tal que
\[\sup_{t\in [0,T]}\left|\int_0^t\!X_s^n\,dW_s
-\int_0^t\!X_s\,dW_s\right|\xrightarrow[n\to\infty]{c.s}0,\] es
decir
\[\int_0^t X_s^n\,dW_s\xrightarrow[n\to\infty]{c.s.}\int_0^t
X_s\,dW_s, \ \ \ \text{para todo} \ t\le T\] y análogamente
\[\int_0^t K_s^n\,dW_s\xrightarrow[n\to\infty]{c.s.}\int_0^t
K_s\,dW_s, \ \ \ \text{para todo} \ t\le T.\] Por unicidad del
límite c.s. en $0\le t<\tau\wedge T$ se sigue que
\[\int_0^t X_s\,dW_s=\int_0^t K_s\,dW_s, \ \ \ \text{c.s. \ si} \ t<\tau\wedge T,\]
y haciendo $T\to\infty$ se obtiene para todo $t<\tau.$
\end{proof}

\begin{prop}\label{cambiovar}
Sea $X\in M^2[0,s].$ Entonces para todo $t\le s,$
\[\int_t^s X_r\,dW_r=\int_0^{s-t}X_{t+r}\,d(W_{t+r}-W_t).\]
\end{prop}
\begin{proof}
Supongamos primero que $X$ es un proceso escalonado de la forma
\[X_s=\sum_{i=0}^{n-1}e_i\mathbf{1}_{[t_i,t_{i+1})}(s)\] con
$0=t_0<t_1<\cdots<t_n=s.$ Podemos suponer sin pérdida de
generalidad que existe $1\le m\le n-1$ tal que $t_m=t.$ Así
entonces para $t\le t+r\le s$
\[X_{t+r}=\sum_{j=m}^{n-1}e_j\mathbf{1}_{[t_j,t_{j+1})}(t+r)
=\sum_{j=m}^{n-1}e_j\mathbf{1}_{[t_j-t,t_{j+1}-t)}(r), \ \ \ 0\le
r\le s-t\] y
\[\int_0^{s-t}X_{t+r}\,d(W_{t+r}-W_t)
=\sum_{j=m}^{n-1}e_j\bigl[W_{t_{j+1}}-W_t-(W_{t_j}-W_t)\bigr]
=\sum_{j=m}^{n-1}e_j\bigl(W_{t_{j+1}}-W_{t_j}\bigr)=\int_t^s
X_r\,dW_r.\] Aproximando con procesos escalonados se obtiene el
resultado para el caso general.
\end{proof}

\subsubsection{Procesos de Itô, diferencial estocástica y Fórmula de
Itô}
El ejemplo \ref{ejemplo1} ilustra que la definición básica de
la Integral de Itô no es muy útil cuando tratamos de evaluar una
integral dada. Esto es similar al caso de las integrales de
Riemman ordinarias, donde usamos el teorema fundamental del
cálculo y la regla de la cadena para hacer cálculos explícitos, en
vez de la definición básica.

{} En este contexto, sin embargo, no tenemos una teoría de
diferenciación. Pero aún así es posible establecer una versión de
la regla de la cadena para el cálculo estocástico. Antes de
estudiar dicha fórmula necesitamos introducir una clase especial
de procesos:

\begin{defi}
Sea $T>0$ fijo. Un proceso estocástico $(X_{t})_{t\in [0,T]}$ es
llamado un \textit{Proceso de Itô} si tiene trayectorias continuas
c.s. y es de la forma
\begin{equation}\label{procesoito}
X_{t}=X_{0}+\int_{0}^{t}h_s\,ds+\int_{0}^{t}G_s\,dW_{s}
\end{equation}
donde $X_0$ es $\F_0-$medible, y $(h_t)_{t\in [0,T]}$ y
$(G_t)_{t\in [0,T]}$ son procesos reales adaptados a
$\{\F_t\}_{t\in [0,T]}$ tales que
\[\int_0^T|h_s|\,ds<+\infty, \ \ \ \int_0^T|G_s|^2\,ds<+\infty, \ \ \ \text{c.s.}\]
es decir $h\in H^1[0,T]$ y $G\in H^2[0,T].$
\end{defi}

\noindent\textbf{Notación.} Para un proceso de Itô se acostumbra
escribir (\ref{procesoito}) como
\begin{equation}
dX_{t}=h_t\,dt+G_t\,dW_{t}
\end{equation}
y llamar a $dX_{t}$ la \textit{diferencial estocástica} de
$X_{t}.$ Esto es conocido como la \textit{notaci\'{o}n diferencial
de It\^{o} }y es una forma eficiente de escribir la ecuaci\'{o}n
(\ref{procesoito}). Se debe tener en cuenta que la diferencial
estoc\'{a}stica no tiene un significado matem\'{a}tico explícito y
sólo debe ser entendida en el contexto de la ecuaci\'{o}n
(\ref{procesoito}).

\begin{ejem}\label{ejemplo2}
En virtud del ejemplo \ref{ejemplo1} se tiene que el proceso
$(W_t^2)_{t\geq 0}$ es un proceso de Itô con diferencial
estocástica $d(W_t^2)=dt+2W_t\,dW_t.$
\end{ejem}

\begin{ejem}\label{ejemplo3}
Sea $(\pi_n)_{n\geq 1}$ una sucesión de particiones
\[0=t_{n,0}<t_{n,1}<\cdots<t_{n,n}=t\]
de $[0,t]$ tal que $|\!|\pi_n|\!|\xrightarrow[n\to\infty]\,0.$ Por el
teorema \ref{aproxsumas},
\[\sum_{k=0}^{n-1}t_{n,k}(W_{t_{n,k+1}}
-W_{t_{n,k}})\xrightarrow[n\to\infty]{\mathbf{P}}\int_0^t
s\,dW_s,\] y claramente
\[\sum_{k=0}^{n-1}W_{t_{n,k+1}}(t_{n,k+1}
-t_{n,k})\xrightarrow[n\to\infty]\,\int_0^t W_s\,ds,\] para todo
$\omega$ para el cual $t\mapsto W_t(\omega)$ es continua. La suma
de los lados de la izquierda es igual a
\[\sum_{k=0}^{n-1}(t_{n,k+1}W_{t_{n,k+1}}-W_{t_{n,k}})=tW_t,\]
luego $d(tW_t)=W_t\,dt+t\,dW_t.$
\end{ejem}

\begin{defi}
Sea $(X_t)_{t\in [0,T]}$ como en (\ref{procesoito}), y sea
$(K_t)_{t\in [0,T]}$ tal que
\[\sup_{t\in [0,T]}|K_s|<+\infty, \ \ \ \text{c.s.}\]
Definimos \ $K_t\,dX_t:=K_t h_t\,dt+K_t G_t\,dW_t.$
\end{defi}

\begin{teorema}[Regla del producto]
Si $X_t^i=h_t^i\,dt+G_t^i\,dW_t \ (i=1,2),$ entonces
\[d(X_t^1 X_t^2)=X_t^1\,dX_t^2+X_t^2\,dX_t^1+G_t^1 G_t^2\,dt\]
o en forma equivalente
\begin{equation}\label{reglaproducto2}
X_t^1 X_t^2=X_0^1 X_0^2+\int_0^t(X_s^1 h_s^2+X_s^2 h_s^1+G_s^1
G_s^2)\,ds+\int_0^t(X_s^1G_s^2+X_s^2G_s^1)\,dW_s.
\end{equation}
\end{teorema}
\begin{proof}
Sea $t\geq 0$ y suponga primero que $h^i$ y $G^i$ son procesos escalonados de la forma
\[h_s^i=\sum_{j=0}^{n-1}c_{ij}\mathbf{1}_{[t_j,t_{j+1})}(s), \ \ \ \ \
G_s^i=\sum_{j=0}^{n-1}e_{ij}\mathbf{1}_{[t_j,t_{j+1})}(s), \ \ \ \
i=1,2\] donde $0=t_0<t_1<\cdots<t_n=t$ y $c_{ij},e_{ij}$ son
$\F_{t_j}-$medibles, $i=1,2.$ Dado que
\[X_s^i=X_{t_j}^i+c_{ij}(s-t_j)+e_{ij}(W_s-W_{t_j}), \ \ \text {si} \ s\in[t_j,t_{j+1}), \ \
0\le j\le n-1, \ \ i=1,2\] y
$\zeta_{ij}:=X_{t_j}^i-c_{ij}t_j-e_{1j}W_{t_j}$ es
$\F_{t_j}-$medible, por el lema \ref{lemazeta} tenemos
\begin{align*}
\int_{t_j}^{t_{j+1}}&(X_s^1 h_s^2+X_s^2 h_s^1+G_s^1
G_s^2)\,ds+\int_{t_j}^{t_{j+1}}(X_s^1G_s^2+X_s^2G_s^1)\,dW_s\\
=&\,c_{2j}\zeta_{1j}(t_{j+1}-t_j)
+\tfrac{1}{2}c_{2j}c_{1j}(t_{j+1}^2-t_j^2)+c_{2j}e_{1j}\int_{t_j}^{t_{j+1}}W_s\,ds\\
&+c_{1j}\zeta_{2j}(t_{j+1}-t_j)
+\tfrac{1}{2}c_{1j}c_{2j}(t_{j+1}^2-t_j^2)+c_{1j}e_{2j}\int_{t_j}^{t_{j+1}}W_s\,ds\\
&+e_{1j}e_{2j}(t_{j+1}-t_j)\\
&+e_{2j}\zeta_{1j}(W_{t_{j+1}}-W_{t_j})
+e_{2j}c_{1j}\int_{t_j}^{t_{j+1}}s\,dW_s+e_{2j}e_{1j}\int_{t_j}^{t_{j+1}}W_s\,dW_s\\
&+e_{1j}\zeta_{2j}(W_{t_{j+1}}-W_{t_j})
+e_{1j}c_{2j}\int_{t_j}^{t_{j+1}}s\,dW_s+e_{1j}e_{2j}\int_{t_j}^{t_{j+1}}W_s\,dW_s\\
=&\,c_{2j}\zeta_{1j}(t_{j+1}-t_j)+c_{1j}\zeta_{2j}(t_{j+1}-t_j)+c_{1j}c_{2j}(t_{j+1}^2-t_j^2)\\
&+e_{2j}\zeta_{1j}(W_{t_{j+1}}-W_{t_j})+e_{1j}\zeta_{2j}(W_{t_{j+1}}-W_{t_j})\\
&+(e_{1j}c_{2j}+c_{1j}e_{2j})\left\{\int_{t_j}^{t_{j+1}}W_s\,ds
+\int_{t_j}^{t_{j+1}}s\,dW_s\right\}\\
&+e_{1j}e_{2j}\left\{t_{j+1}-t_j+2\int_{t_j}^{t_{j+1}}W_s\,dW_s\right\}.
\end{align*}
para cada $0\le j\le n-1.$ De los ejemplos \ref{ejemplo2} y
\ref{ejemplo3} se tiene que
\begin{align*}
t_{j+1}W_{t_{j+1}}-t_j
W_{t_j}&=\int_{t_j}^{t_{j+1}}W_s\,ds+\int_{t_j}^{t_{j+1}}s\,dW_s\\
W_{t_{j+1}}^2-W_{t_j}^2&=t_{j+1}-t_j+2\int_{t_j}^{t_{j+1}}W_s\,dW_s,
\end{align*}
luego
\begin{equation}\label{reglaproducto1}
\begin{split}
\int_{t_j}^{t_{j+1}}&(X_s^1 h_s^2+X_s^2 h_s^1+G_s^1
G_s^2)\,ds+\int_{t_j}^{t_{j+1}}(X_s^1G_s^2+X_s^2G_s^1)\,dW_s\\
=&\,c_{2j}\zeta_{1j}(t_{j+1}-t_j)+c_{1j}\zeta_{2j}(t_{j+1}-t_j)+c_{1j}c_{2j}(t_{j+1}^2-t_j^2)\\
&+e_{2j}\zeta_{1j}(W_{t_{j+1}}-W_{t_j})+e_{1j}\zeta_{2j}(W_{t_{j+1}}-W_{t_j})\\
&+(e_{1j}c_{2j}+c_{1j}e_{2j})(t_{j+1}W_{t_{j+1}}-t_j
W_{t_j})+e_{1j}e_{2j}(W_{t_{j+1}}^2-W_{t_j}^2),
\end{split}
\end{equation}
Por otro lado
\begin{align*}
X_{t_{j+1}}^1&X_{t_{j+1}}^2-X_{t_j}^1 X_{t_j}^2\\
=&\,(\zeta_{1j}+c_{1j}t_{j+1}+e_{1j}W_{t_{j+1}})(\zeta_{2j}+c_{2j}t_{j+1}+e_{2j}W_{t_{j+1}})\\
&-(\zeta_{1j}+c_{1j}t_j+e_{1j}W_{t_j})(\zeta_{2j}+c_{2j}t_j+e_{2j}W_{t_j})\\[0.2cm]
=&\,\zeta_{1j}\zeta_{2j}+c_{2j}\zeta_{1j}t_{j+1}+e_{2j}\zeta_{1j}W_{t_{j+1}}
+c_{1j}\zeta_{2j}t_{j+1}+c_{1j}c_{2j}t_{j+1}^2+c_{1j}e_{2j}W_{t_{j+1}}t_{j+1}\\
&+e_{1j}\zeta_{2j}W_{t_{j+1}}+c_{2j}e_{1j}W_{t_{j+1}}t_{j+1}+e_{1j}e_{2j}W_{t_{j+1}}^2\\
&-\zeta_{1j}\zeta_{2j}-c_{2j}\zeta_{1j}t_j-e_{2j}\zeta_{1j}W_{t_j}
-c_{1j}\zeta_{2j}t_j-c_{1j}c_{2j}t_j^2-c_{1j}e_{2j}W_{t_j}t_j\\
&-e_{1j}\zeta_{2j}W_{t_j}-c_{2j}e_{1j}W_{t_j}t_j-e_{1j}e_{2j}W_{t_j}^2\\[0.2cm]
=&\,c_{2j}\zeta_{1j}(t_{j+1}-t_j)+c_{1j}\zeta_{2j}(t_{j+1}-t_j)
+c_{1j}c_{2j}(t_{j+1}^2-t_j^2)\\
&+e_{2j}\zeta_{1j}(W_{t_{j+1}}-W_{t_j})+e_{1j}\zeta_{2j}(W_{t_{j+1}}-W_{t_j})\\
&+(e_{1j}c_{2j}+c_{1j}e_{2j})(t_{j+1}W_{t_{j+1}}-t_j
W_{t_j})+e_{1j}e_{2j}(W_{t_{j+1}}^2-W_{t_j}^2)
\end{align*}
y esto coincide con (\ref{reglaproducto1}), obteniendo así la
regla del producto en el intervalo $[t_j,t_{j+1}),$ y sumando
sobre $j$ se obtiene para el caso $h,G$ escalonados. Para el caso
general considere sucesiones $(h^{i,n})_{n\geq 1}, \
(G^{i,n})_{n\geq 1}$ de procesos escalonados tales que
\begin{align*}
\int_0^T|G_s^{i,n}-G_s|^2\,ds&\xrightarrow[n\to\infty]{c.s}0\\
\int_0^T|h_s^{i,n}-h_s|\,ds&\xrightarrow[n\to\infty]{c.s}0
\end{align*}
para $i=1,2.$ Sea
\[X_t^{i,n}:=X_0^i+\int_0^t h_s^{i,n}\,ds+\int_0^t G_s^{i,n}\,dW_s, \ \ \
n\geq 1, \ \ i=1,2.\] Por el corolario \ref{corosup}
\[\sup_{t\in [0,T]}|X_t^{i,n}-X_t^i|\le\int_0^T|h_s^{i,n}-h_s|\,ds
+\sup_{t\in [0,T]}\Bigl|\int_0^t(G_s^{i,n}-G_s)\,dW_s\Bigr|\xrightarrow[n\to\infty]{\mathbf{P}}0,\]
luego existe una subsucesión, que seguiremos notando por
$X^{i,n},$ tal que $X_t^{i,n}\to X_t^i$ uniformemente en $[0,T],$
c.s. Por esta razón y por el teorema \ref{aproxprob}, usando la
desigualdad
\begin{align*}
|X_s^{i,n}G_s^{j,n}-X_s^iG_s^i|^2&\le
2|X_s^{i,n}G_s^{j,n}-X_s^iG_s^{j,n}|^2
+2|X_s^iG_s^{j,n}-X_s^iG_s^i|^2\\
&=2|X_s^{i,n}-X_s^i|^2\cdot|G_s^{j,n}|^2
+2|G_s^{j,n}-G_s^i|^2\cdot|X_s^i|^2, \ \ \ \ i,j=1,2.
\end{align*}
Se sigue que
\[\int_0^tX_s^{i,n}G_s^{j,n}\,dW_s\xrightarrow[n\to\infty]{\mathbf{P}}\int_0^tX_s^iG_s^i\,dW_s,
\ \ \ t\in[0,T], \ \ i,j=1,2.\] Claramente también se tiene que
\begin{align*}
\int_0^tX_s^{i,n}h_s^{j,n}\,ds&\xrightarrow[n\to\infty]{c.s}\int_0^t
X_s^ih_s^j\,ds, \ \ \ \ i,j=1,2\\
\int_0^tG_s^{1,n}G_s^{2,n}\,ds&\xrightarrow[n\to\infty]{c.s}\int_0^t
G_s^1G_s^2\,ds, \ \ \ \ \forall t\in[0,T].
\end{align*}
Escribiendo (\ref{reglaproducto2}) para $h^{i,n},G^{i,n}$ y
$X^{i,n},$ y tomando el límite en probabilidad cuando{}
$n\to\infty,$ se sigue la regla del producto para el caso general.
\end{proof}

\begin{teorema}[Fórmula de Itô]
Sea $dX_t=h_t\,dt+G_t\,dW_t,$ y sea $F:[0,\infty)\times\R\to\R$ de
clase $\C^{1,2}$ (es decir $\C^1$ en la variable $t\geq 0$ y
$\C^2$ en la variable $x\in\R).$ Entonces el proceso $F(t,X_t), \,
t\geq 0,$ tiene diferencial estocástica dada por
\[dF(t,X_t)=\left\{\tfrac{\partial F}{\partial t}(t,X_t)
+\tfrac{\partial F}{\partial x}(t,X_t)h_t
+\tfrac{1}{2}\tfrac{\partial^2 F}{\partial
x^2}(t,X_t)G_t^2\right\}dt+\tfrac{\partial F}{\partial
x}(t,X_t)G_t\,dW_t.\]
\end{teorema}
{} Note que si $W_t$ fuera continuamente diferenciable en $t,$
entonces (por la regla de la cadena para derivadas totales) el
término $\frac{1}{2}\frac{\partial^2 F}{\partial
x^2}(t,X_t)G_t^2\,dt$ no aparecería.
\begin{proof}
{}{}\textbf{Paso 1.} Usando inducción y el teorema anterior se
prueba fácilmente que
\begin{equation}\label{formulaito1}
d(X_t^m)=\left\{mX_t^{m-1}h_t+\tfrac{1}{2}m(m-1)X_t^{m-2}G_t^2\right\}\,dt
+mX_t^{m-1}G_t\,dW_t,
\end{equation}
para todo entero $m\geq 2.$ Por linealidad de la diferencial
estocástica se obtiene entonces
\begin{equation}\label{formulaito2}
dQ(X_t)=\left\{Q'(X_t)h_t+\tfrac{1}{2}Q''(X_t)G_t^2\right\}\,dt+Q'(X_t)G_t\,dW_t
\end{equation}
para cualquier polinomio $Q(x).$

{}{}\textbf{Paso 2.} Sea $P(t,x)=g(t)Q(x)$ donde $Q(x)$ es un
polinomio en $x\in\R$ y $g(t)$ es continuamente diferenciable para
$t\geq 0.$ Dado que $dg(t)=g'(t)dt,$ de nuevo por la regla del
producto y (\ref{formulaito2}),
\begin{align*}
dP(t,X_t)&=Q(X_t)\,dg(t)+g(t)\,dQ(X_t)\\
&=\left\{g'(t)Q(X_t)+g(t)Q'(X_t)h_t+\tfrac{1}{2}g(t)Q''(X_t)G_t^2\right\}\,dt
+g(t)Q'(X_t)G_t\,dW_t,
\end{align*}
es decir
\begin{equation}\label{formulaito3}
\begin{split}
P&(t,X_t)-P(0,X_0)\\
&=\int_0^t\left\{\tfrac{\partial P}{\partial
s}(s,X_s)+\tfrac{\partial P}{\partial
x}(s,X_s)h_s+\tfrac{1}{2}\tfrac{\partial^2 P}{\partial
x^2}(s,X_s)G_s^2\right\}ds+\int_0^t\tfrac{\partial P}{\partial
x}(s,X_s)G_s\,dW_s,
\end{split}
\end{equation}
para $t\in [0,T].$

{}{}\textbf{Paso 3.} Por linealidad, la fórmula
(\ref{formulaito3}) aún es valida si
\begin{equation}\label{formulaito4}
P(t,x)=\sum_{i=1}^m g_i(t)Q_i(x),
\end{equation}
donde los $Q_i(x)$ son polinomios y los $g_i(t)$ son continuamente
diferenciables. Sea ahora $(P_n)_{n\geq 1}$ una sucesión de
polinomios en $x$ y $t$ tales que
\begin{align*}
P_n(t,x)&\xrightarrow[n\to\infty]\,F(t,x),
\ \ \ \ \ \ \ \ \ \ \ \ \ \frac{\partial P_n}{\partial
t}(t,x)\xrightarrow[n\to\infty]\,\frac{\partial F}{\partial t}(t,x),\\[0.3cm]
\frac{\partial P_n}{\partial
x}(t,x)&\xrightarrow[n\to\infty]\,\frac{\partial F}{\partial x}(t,x),
\ \ \ \ \ \ \ \ \ \ \frac{\partial^2 P_n}{\partial
x^2}(t,x)\xrightarrow[n\to\infty]\,\frac{\partial^2 F}{\partial x^2}(t,x),
\end{align*}
uniformemente sobre conjuntos compactos de $[0,\infty)\times\R.$
Dado que cada polinomio $P_n$ se puede escribir de la forma
(\ref{formulaito4}), por (\ref{formulaito3}) tenemos
\begin{equation}\label{formulaito5}
\begin{split}
&P_n(t,X_t)-P_n(0,X_0)\\
&=\int_0^t\left\{\tfrac{\partial P_n}{\partial
s}(s,X_s)+\tfrac{\partial P_n}{\partial
x}(s,X_s)h_s+\tfrac{1}{2}\tfrac{\partial^2 P_n}{\partial
x^2}(s,X_s)G_s^2\right\}ds+\int_0^t\tfrac{\partial P_n}{\partial
x}(s,X_s)G_s\,dW_s,
\end{split}
\end{equation}
para $t\in [0,T].$ Dado que
\[\int_0^t\left|\tfrac{\partial P_n}{\partial
x}(s,X_s)G_s-\tfrac{\partial F}{\partial
x}(s,X_s)G_s\right|^2 ds\xrightarrow[n\to\infty]{c.s}0,\]
por el teorema \ref{aproxprob} se tiene que
\[\int_0^t\tfrac{\partial P_n}{\partial
x}(s,X_s)G_s\,dW_s\xrightarrow[n\to\infty]{\mathbf{P}}\int_0^t\tfrac{\partial F}{\partial
x}(s,X_s)G_s\,dW_s.\]
Además
\begin{align*}
\int_0^t&\left\{\tfrac{\partial P_n}{\partial
s}(s,X_s)+\tfrac{\partial P_n}{\partial
x}(s,X_s)h_s+\tfrac{1}{2}\tfrac{\partial^2 P_n}{\partial
x^2}(s,X_s)G_s^2\right\}ds\\
&\xrightarrow[n\to\infty]{c.s}\int_0^t\left\{\tfrac{\partial
F}{\partial s}(s,X_s)+\tfrac{\partial F}{\partial
x}(s,X_s)h_s+\tfrac{1}{2}\tfrac{\partial^2 F}{\partial
x^2}(s,X_s)G_s^2\right\}ds.
\end{align*}
Tomando el límite en probabilidad en (\ref{formulaito5}) cuando
$n\to\infty$ se sigue que
\begin{align*}
F&(t,X_t)-F(0,X_0)\\
=&\int_0^t\left\{\tfrac{\partial F}{\partial
s}(s,X_s)+\tfrac{\partial F}{\partial
x}(s,X_s)h_s+\tfrac{1}{2}\tfrac{\partial^2 F}{\partial
x^2}(s,X_s)G_s^2\right\}ds+\int_0^t\tfrac{\partial F}{\partial
x}(s,X_s)G_s\,dW_s,
\end{align*}
lo cual prueba el teorema.
\end{proof}

\begin{defi} Un proceso $m-$dimensional
$W_t=(W_t^1,\ldots,W_t^m)^*, \ t\geq 0,$ será llamado un
movimiento Browniano $m-$dimensional si sus componentes $W_t^i$
son movimientos Brownianos reales independientes.
\end{defi}

Sea $(W_t)_{t\geq 0}$ un movimiento Browniano $m-$dimensional
con $W_0=0$  y sea $\{\F_t\}_{t\geq 0}$ la filtración
generada por $(W_t)_{t\geq 0}$ aumentada con los conjuntos
$\mathbf{P}-$nulos de $\Omega,$ es decir
\[\F_t=\sigma\bigl(\mathcal{N}\cup\{W_s: 0\leq s\le t\}\bigr),
\ \ \ t\geq 0 \] donde $\mathcal{N}=\{E\subseteq\Omega : \exists
G\in \F,\, G\supseteq E \text{ y } \mathbf{P}(G)=0\}.$

\begin{defi}
Notaremos con $H^{p,d\times m}[t_0,t]$ el espacio vectorial de los
procesos $(X_s)_{s\geq 0}$ con valores en el espacio de matrices
$\R^{d\times m}$ y adaptados a $\{\F_t\}_{t\geq 0}$ tales que
$X^{ij}\in H^p[t_0,t]$ para cada $1\le i\le d, \ 1\le j\le m.$
Análogamente se define $M^{p,d\times m}[t_0,t].$
\end{defi}

\begin{defi}
Para $X\in H^{p,d\times m}[t_0,t]$ definimos la integral de Itô
$d-$dimensional como el vector columna
\[\int_{t_0}^tX_s\,dW_s=\left(\sum_{j=1}^m\int_{t_0}^t X_s^{ij}\,dW_s^j\right)_{1\le i\le d}\]
\end{defi}

{} La siguiente es la generalización de la Fórmula de Itô al caso
$d-$dimensional (ver \cite{friedman2}):
\begin{teorema}
Sea $(X_t)_{t\geq 0}$ un proceso de Itô $d-$dimensional de la
forma
\[X_t=X_0+\int_0^th_s\,ds+\int_0^tG_s\,dW_s\]
con $h\in H^{1,d\times 1}[0,T]$ y $G\in M^{2,d\times m}[0,T],$ y
sea $F:[0,\infty)\times\R^d\to\R$ de clase $\C^{1,2}$ (es decir
$\C^1$ en la variable $t\geq 0$ y $\C^2$ en la variable
$x\in\R^d).$ Entonces el proceso $F(t,X_t), \, t\geq 0,$ tiene
diferencial estocástica dada por
\[dF(t,X_t)=\left\{\tfrac{\partial F}{\partial t}(t,X_t)
+\nabla_xF(t,X_t)\cdot h_t+\tfrac{1}{2}Tr\bigl(D_x^2F(t,X_t)\cdot
G_tG_t^*\bigr)\right\}dt+\nabla_xF(t,X_t)\cdot G_t\,dW_t,\] donde
\[\nabla_x F=\Bigl(\frac{\partial F}{\partial x_1
},\cdots,\frac{\partial F}{\partial x_d}\Bigr)\] es el gradiente
de $F$ con respecto a $x$ y $D_x^2F$ es la matriz de tamaño
$d\times d$ de las segundas derivadas con respecto a x, es decir,
$(D_x^2F)_{ij}=\frac{\partial^2 F}{\partial x_i\partial x_j}, \
1\le i,j\le d.$
\end{teorema}

\subsubsection{Aplicaciones de la Fórmula de Itô}
Recuerde que $\{\F_t\}_{t\geq 0}$ es la filtración generada por un
movimiento Browniano $m-$dimen-sional $(W_t)_{t\geq 0},$ aumentada
con los conjuntos $\mathbf{P}-$nulos de $\Omega.$

\begin{teorema}[Teorema de Representación de
Martingalas] Sean $T>0$ fijo y sea $(M_t)_{t\in [0,T]}$ una
martingala con respecto a $\{\F_t\}_{t\in [0,T]}$ tal que
\[E(|M_t|^2)<+\infty, \ \ \ \forall t\in[0,T].\] Entonces
existe un proceso $X\in M^{2,1\times m}[0,T]$ tal que
\[M_t=M_0+\int_0^t X_s\,dW_s, \ \ \ \text{c.s.}\]
\end{teorema}
\begin{proof}
Sea $\mathcal{H}:=\{Z\in L^2(\Omega,\F_T,\mathbf{P}):E(Z)=0\}$
dotado con el producto interno usual de
$L^2(\Omega,\F_T,\mathbf{P})$
\[\bigl\langle Z^1,Z^2\bigr\rangle=E\bigl(Z^1\cdot Z^2\bigr).\]
Dado que la aplicación $X\longmapsto\int_0^T\! X_s\,dW_s$ de
$M^{2,m}[0,T]$ en $\mathcal{H}$ es una isometría, el espacio
\[\V:=\left\{\textstyle\int_0^T\!
X_s\,dW_s:X\in M^{2,1\times m}[0,T]\right\}\] es un subespacio
cerrado de $\mathcal{H},$ luego $\mathcal{H}=\V\oplus\V^{\bot},$
donde $\V^{\bot}$ es el \textit{complemento ortogonal} de $\V$ en
$\mathcal{H}:$
\[\V^{\bot}=\{N\in\mathcal{H}:E(Z\cdot N)=0, \ \forall N\in\V\}.\]
Si probamos que $\V=\mathcal{H},$ toda variable aleatoria cuadrado
integrable $\F_T-$medible con esperanza cero podrá entonces ser
representada como la integral estocástica de un proceso en
$M^{2,1\times m}[0,T].$ En particular, la v.a.
$M_T-M_0\in\mathcal{H}$ tendrá la representación
\[M_T-M_0=\int_0^T X_s\,dW_s, \ \ \ \ \ X\in M^{2,1\times m}[0,T].\]
Tomando valor esperado con respecto a $\F_t$ y usando la propiedad
de martingala de la integral estocástica se sigue
\[M_t-M_0=\int_0^t X_s\,dW_s, \ \ \ \ \  t\in [0,T]\]
obteniendo así la representación deseada. Dado que
$\mathcal{H}=\V\oplus\V^{\bot},$ probar que $\V=\mathcal{H}$ es
equivalente a verificar $\V^{\bot}=\{0\}:$

{}{}\textbf{Paso 1.} Sea $Z\in\V^{\bot}$ y definamos la martingala
$Z_t:=E\bigl(Z\mid\F_t\bigr), \ t\in[0,T].$ Dado que
\[\int_0^\tau X_s\,dW_s=\int_0^T X_s\mathbf{1}_{[0,\tau]}(s)\,dW_s,\]
para todo tiempo de parada $\tau\le T$ y para todo $X\in
M^{2,m}[0,T],$ entonces $\int_0^\tau\!X_s\,dW_s\in\V$ y
\begin{align*}
0=E\left[Z\cdot\int_0^\tau X_s\,dW_s\right]
=E\Biggl[E\left[Z\cdot\int_0^\tau
X_s\,dW_s\,\Bigl|\,\F_\tau\right]\Biggr]&=E\left[\int_0^\tau
X_s\,dW_s\cdot E\bigl[Z\mid\F_\tau\bigr]\right]\\
&=E\left[Z_\tau\!\int_0^\tau X_s\,dW_s\right].
\end{align*}
Dado que $E[Z_0]=E\bigl[E[Z\mid\F_0]\bigr]=E[Z]=0,$ por el teorema
\ref{lemarepmart} el proceso
\[Z_t\cdot\int_0^t\!X_s\,dW_s, \ \ \ \ t\in [0,T],\] resulta ser una
martingala con respecto a $\{\F_t\}_{t\in [0,T]}$, para todo
$X\in M^{2,m}[0,T].$

{}{}\textbf{Paso 2.} Para un $\theta\in\R^m$ fijo pero arbitrario
defina
\[f(t,x):=\exp(i\theta^*x+\tfrac{1}{2}|\theta|^2 t)
=\underbrace{\exp(\tfrac{1}{2}|\theta|^2t)\cdot\cos(\theta^*x)}_{f_1(t,x)}
+i\cdot\underbrace{\exp(\tfrac{1}{2}|\theta|^2t)\cdot\sin(\theta^*x)}_{f_2(t,x)},\]
con $(t,x)\in[0,T]\times\R^m,$ y defina $M_t^\theta:=f(t,W_t), \
t\in[0,T].$ Entonces $|M_t^\theta|=\exp(\tfrac{1}{2}|\theta|^2t).$
Dado que
\[\frac{\partial f_1}{\partial t}(t,x)=\tfrac{1}{2}|\theta|^2 f_1(t,x), \ \ \ \
\frac{\partial f_1}{\partial x_i}(t,x)=-f_2(t,x)\theta_i, \ \ \ \
\frac{\partial^2 f_1}{\partial
x_ix_j}(t,x)=-f_1(t,x)\theta_i\theta_j,\] aplicando Fórmula de Itô
$m-$dimensional a $f_1(t,W_t)$ obtenemos
\[d\bigl[f_1(t,W_t)\bigr]=\Bigl\{\tfrac{1}{2}|\theta|^2 f_1(t,W_t)
-\tfrac{1}{2}\sum_{j=1}^m
f_1(t,W_t)\theta_j^2\Bigr\}dt-f_2(t,W_t)\theta^*dW_t=-f_2(t,W_t)\theta^*dW_t.\]
De forma completamente análoga se deduce que
$d\bigl[f_2(t,W_t)\bigr]=f_1(t,W_t)\theta^*dW_t,$ luego
\[d(f(t,W_t)=\left\{-f_2(t,W_t)+i\cdot f_1(t,W_t)\right\}\theta^*dW_t
=i\cdot f(t,W_t)\theta^*dW_t\] es decir
\[M_t^\theta=1+\sum_{j=1}^m i\cdot\int_0^t \theta_jM_s^\theta\,dW_s^j, \ \ \ t\in [0,T].\]
Como $|M_t^\theta|\le\exp(\tfrac{1}{2}|\theta|^2T)$ si $t\in[0,T]$
para un $\theta\in\R^m$ fijo, las integrales estocásticas en la
última igualdad son martingalas con respecto a $\{\F_t\}_{t\in [0,T]},$ y en particular, $(M_t^\theta)_{t\in [0,T]}$ también lo
es.

{}{}\textbf{Paso 3.} Por los pasos 1 y 2 de la prueba
$(Z_tM_t^\theta)_{t\in [0,T]}$ es una martingala con respecto a
$\{\F_t\}_{t\in [0,T]}.$ Entonces
\[E\bigl[Z_tM_t^\theta\mid \F_s\bigr]=Z_sM_s^\theta
=Z_s\exp(i\theta^*W_s+\tfrac{1}{2}|\theta|^2 s), \ \ \ \ \ 0\le
s\le t\le T.\] Multiplicando a ambos lados por
$\exp(-i\theta^*W_s-\tfrac{1}{2}|\theta|^2 t)$ y usando que
$\exp(-i\theta^*W_s)$ es $\F_s-$medible se sigue que
\[E\bigl[Z_t\exp\bigl(i\theta^*(W_t-W_s)\bigr)\mid \F_s\bigr]
=Z_s\exp\bigr(-\tfrac{1}{2}|\theta|^2(t-s)\bigl), \ \ \ \ \ 0\le
s\le t\le T.\] En particular, si $0=t_0<t_1<\cdots<t_n\le T, \
\Delta_k:=W_{t_k}-W_{t_{k-1}}$ y $\theta^k\in\R^m, \
k=1,\ldots,n,$ se obtiene
\begin{align*}
E\biggl[Z_T\cdot\exp\Bigl(i\sum_{k=1}^n(\theta^k)^*\Delta_k\Bigr)\biggr]
&=E\Biggl[E\biggl[Z_T\cdot
\exp\Bigl(i\sum_{k=1}^n(\theta^k)^*\Delta_k\Bigr)
\,\Bigl|\,\F_{t_{n-1}}\biggr]\Biggr]\\
&=\exp\bigl(-\tfrac{1}{2}|\theta^n|^2(t_n-t_{n-1})\bigr)\cdot
E\biggl[Z_{t_{n-1}}\cdot\exp\Bigl(i\sum_{k=1}^{n-1}(\theta^k)^*\Delta_k\Bigr)\biggr]\\
&=\cdots=\exp\Bigl(-\tfrac{1}{2}\sum_{k=1}^n|\theta^k|^2(t_k-t_{k-1})\Bigr)\cdot
E(Z_0)=0,
\end{align*}
luego \[Z_T\bot\exp\Bigl(i\sum_{k=1}^n(\theta^k)^*\Delta_k\Bigr),
\ \ \forall\theta^k\in\R^m, \ k=1,\ldots,n.\] Como $Z_T=Z$ es
$\F_T-$medible, podemos definir la medida $\mathbf{Q}$ sobre
$\F_T$ via $d\mathbf{Q}=Z\cdot d\mathbf{P},$ es decir
$\mathbf{Q}(A)=E[\mathbf{1}_A Z], \forall A\in\F_T.$ Entonces
\[0=E\biggl[Z_T\cdot\exp\Bigl(i\sum_{k=1}^n(\theta^k)^*\Delta_k\Bigr)\biggr]
=\int_\Omega\exp\Bigl(i\sum_{k=1}^n(\theta^k)^*\Delta_k\Bigr)\,d\mathbf{Q}\]
para todo $\theta^k\in\R^m, \ k=1,\ldots,n.$ La unicidad de la
transformada de Fourier implica que
\[\mathbf{Q}\bigr|_{\sigma(W_{t_1},\ldots,W_{t_n})}=0.\]
Como $\F_T$ es generada por vectores finito dimensionales de la
forma $(W_{t_1},\ldots,W_{t_n}),$ se sigue que
$\mathbf{Q}\bigl|_{\F_T}=0,$ lo cual conduce a $Z\equiv 0$ c.s.
\end{proof}

\begin{teorema}\label{burkholder}
Sea $X\in M^{2,1\times m}[0,T]$ tal que
\[E\int_0^T|X_t|^{2p}\,dt<+\infty\]
para algún $p\geq 1.$ Entonces
\[E\biggl[\sup_{\, t\in [0,T]}\Bigl|\int_0^t
X_s\,dW_s\Bigr|^{2p}\biggr]\le C_p T^{p-1}
E\!\int_0^T|X_s|^{2p}\,ds\] donde $C_p=[4p^3/(2p-1)]^{p}.$
\end{teorema}
\begin{proof}
Aplicando fórmula de Itô al proceso $M_t=\int_0^t X_s\,dW_s$ (que
en notación diferencial se escribe $dM_t=X_t\,dW_t)$ con la
función $F(x)=x^{2p}$ obtenemos
\[d\bigl(M_s^{2p}\bigr)=p(2p-1)(M_s)^{2p-2}|X_s|^2\,ds+2p(M_s)^{2p-1}X_s\,dW_s.\]
Para cada $n\geq 1$ defina los tiempos de parada
$\tau_n:=\inf\{t\geq 0: |M_t|\geq n\}.$ Integrando la anterior
expresión entre $0$ y $t\wedge\tau_n$
\[(M_{t\wedge\tau_n})^{2p}=p(2p-1)\int_0^{t\wedge\tau_n}(M_s)^{2p-2}|X_s|^2\,ds
+2p\int_0^{t\wedge\tau_n}(M_s)^{2p-1}X_s\,dW_s.\] Por la
definición de los tiempos de parada, para cada $n\geq 1$ se tiene
\[E\int_0^t \bigl(M_s^{2p-1}|X_s|\bigr)^2\mathbf{1}_{[0,\tau_n)}(s)\,ds
\le n^{4p-2}\,E\!\int_0^t|X_s|^2\,ds<+\infty\] luego, por la
proposición \ref{intparada}
\[E\int_0^{t\wedge\tau_n}(M_s)^{2p-1}X_s\,dW_s
=E\int_0^t (M_s)^{2p-1}X_s\mathbf{1}_{[0,\tau_n)}(s)\,dW_s=0\] y
\[E\bigl[(M_{t\wedge\tau_n})^{2p}\bigr]
=p(2p-1)E\!\int_0^{t\wedge\tau_n}(M_s)^{2p-2}|X_s|^2\,ds.\]
Entonces la aplicación $t\mapsto
E\bigl[(M_{t\wedge\tau_n})^{2p}\bigr]$ es monótona creciente, y
por la desigualdad de Hölder
\begin{align*}
E\bigl[(M_{T\wedge\tau_n})^{2p}\bigr]&=p(2p-1)E\!\int_0^{T\wedge\tau_n}
(M_s)^{2p-2}|X_s|^2\,ds\\
&=p(2p-1)E\!\int_0^T (M_{s\wedge\tau_n})^{2p-2}|X_{s\wedge\tau_n}|^2\,ds\\
&\le p(2p-1)\biggl\{E\!\int_0^T
(M_{s\wedge\tau_n})^{2p}\,ds\biggr\}^{\frac{2p-2}{2p}}
\cdot\biggl\{E\!\int_0^T
|X_{s\wedge\tau_n}|^{2p}\,ds\biggr\}^{\frac{2}{2p}}\\
&\le p(2p-1)\biggl\{E\!\int_0^T
(M_{T\wedge\tau_n})^{2p}\,ds\biggr\}^{\frac{p-1}{p}}
\cdot\biggl\{E\!\int_0^{T\wedge\tau_n}
|X_s|^{2p}\,ds\biggr\}^{\frac{1}{p}}\\
&\le
p(2p-1)T^{\frac{p-1}{p}}\Bigl\{E\bigl[(M_{T\wedge\tau_n})^{2p}\bigr]\Bigr\}^{\frac{p-1}{p}}
\cdot\biggl\{E\!\int_0^T|X_s|^{2p}\,ds\biggr\}^{\frac{1}{p}}
\end{align*}
dividiendo por
$\Bigl\{E\bigl[(M_{T\wedge\tau_n})^{2p}\bigr]\Bigr\}^{\frac{p-1}{p}}$
a ambos lados
\[\Bigl\{E\bigl[(M_{T\wedge\tau_n})^{2p}\bigr]\Bigr\}^{1/p}
\le p(2p-1)T^{\frac{p-1}{p}}\biggl\{E\!\int_0^T
|X_s|^{2p}\,ds\biggr\}^{1/p}\] elevando a la $p-$ésima potencia y
usando el lema de Fatou
\begin{align*}
E\bigl[(M_T)^{2p}\bigr]&=E\biggl[\liminf_{n\to\infty}(M_{T\wedge\tau_n})^{2p}\biggr]\\
&\le\liminf_{n\to\infty} E\bigl[(M_{T\wedge\tau_n})^{2p}\bigr]\le
[p(2p-1)]^{p}T^{p-1}\, E\!\int_0^T |X_s|^{2p}\,ds.
\end{align*}
Finalmente usamos la segunda desigualdad de Doob para obtener
\[E\biggl[\sup_{t\in [0,T]}(M_t)^{2p}\biggr]
\le\left(\frac{2p}{2p-1}\right)^{2p}E\bigl[M_T^{2p}\bigr]\le[4p^3/(2p-1)]^{p}T^{p-1}\,
E\!\int_0^T|X_s|^{2p}\,ds.\]
\end{proof}
\begin{obse}
Para el caso en que $(X_t)_{t\geq 0}$ toma valores en el espacio
de matrices $\R^{d\times m},$ la constante $C_p$ cambia, ya que
para la integral $d-$dimensional se tiene
\[\Bigl|\int_0^t X_s\,dW_s\Bigr|^{2p}
=\biggl[\sum_{i=1}^d\Bigl|\int_0^t
X_s^{(i)}\,dW_s\Bigr|^2\biggr]^{p} \le d^{p-1}
\sum_{i=1}^d\Bigl|\int_0^t X_s^{(i)}\,dW_s\Bigr|^{2p}\] donde
$X_s^{(i)}$ es la $i-$ésima fila de $X_s,$ luego
\begin{equation}\label{des1}
E\biggl[\sup_{t\in [0,T]}\Bigl|\int_0^t
X_s\,dW_s\Bigr|^{2p}\biggr]\le(dT)^{p-1}[4p^3/(2p-1)]^{p}\,
E\!\int_0^T |\!|X_s|\!|^{2p}\,ds.
\end{equation}
(recuerde que
$|\!|z|\!|^2=Tr[zz^*]=\sum_{i=1}^d\sum_{j=1}^m|z_{ij}|^2,$ para
$z\in\R^{d\times m}).$
\end{obse}

\subsection{Semigrupos y procesos de Markov}
Sea $\ba$ un espacio de \textit{Banach} y sea $L(\ba)$ el espacio
de \textit{Banach} de todos los operadores lineales y continuos
(acotados) sobre $\ba$ con respecto a la norma
\[|\!|T|\!|=\sup_{|\!|f|\!|\le 1}|\!|Tf|\!|.\]
Diremos que un operador $T\in L(\ba)$ es {\it densamente definido}
si el dominio de $T,$ que notaremos por $D(T),$ es denso en $\ba$.
\begin{defi}
Una familia $(T_t)_{t\geq 0}\subset L(\ba)$ con $D(T_t)=\ba$ para
todo $t\geq 0$ se llama {\it Semigrupo fuertemente continuo} o
\textit{$C_0-$semigrupo} si
\begin{description}
\item[(i)]\,$T_0=I$ y $\lim_{t\to 0^+}T_t f=f, \forall f\in\ba$ \ (continuidad
fuerte),
\item[(ii)] $T_t T_s=T_{t+s}, \ \forall s,t\geq 0$ \ (propiedad de semigrupo).
\end{description}
Si además $|\!|T_t|\!|\le 1$ (e.d. $T_t$ es una contracción) para todo
$t\geq 0$ el semigrupo se llama \textit{$C_0-$semigrupo de
contracciones.}
\end{defi}
\begin{defi}
Sea $(T_t)_{t\geq 0}$ un $C_0-$semigrupo sobre $\ba.$ Se define el
\textit{generador infinitesimal} $\A:D(\A)\subseteq\ba\to\ba$ del
semigrupo por
\begin{equation}\label{6.2.1}
\begin{split}
D(\A)&=\Bigl\{f\in B: \lim_{t\downarrow 0}\tfrac{1}{t}(T_tf-f)
\text{ existe en} \ \ba\Bigr\}\\
\A f&=\lim_{t\downarrow 0}\tfrac{1}{t}(T_tf-f).
\end{split}
\end{equation}
\end{defi}
{} El generador infinitesimal es importante porque bajo ciertas
condiciones caracteriza completamente al $C_0-$semigrupo (Teoremas
de Hille-Yosida y Lumer-Phillips \cite{pazy}).

{} Antes de ver las propiedades más importantes del generador
infinitesimal necesitamos primero la siguiente definición
\begin{defi}
Una función $f:[a,b]\to\ba$ se dice \textit{Riemann-integrable} si
el límite
\[\lim_{n\to\infty}\sum_{k=1}^n f(s_k)(t_k-t_{k-1})\]
existe, es independiente de las particiones
$\pi_n:a=t_0,t_1<\cdots<t_n=b,$ donde{} $|\!|\pi_n|\!|=\max_{1\le k\le
n}(t_k-t_{k-1})\to\,0,$ y también es independiente de los puntos
intermedios $s_k\in[t_{k_1},t_k].$ Denotaremos el límite por
$\int_a^b\,f(t)\,dt.$
\end{defi}

\begin{teorema}
Sea $\A$ el generador infinitesimal de un $C_0-$semigrupo
$(T_t)_{t>0}$ sobre $\ba.$

{}{} (a) Si $f\in D(\A)$ y $t\geq 0$ entonces $T_t f\in D(\A)$ y
\[\frac{d(T_t f)}{dt}=\A T_t f=T_t\A f.\] En particular se tiene
la igualdad
\begin{equation}\label{semiprop1}
T_t f-f=\int_0^t\A T_sf\,ds=\int_0^t T_s\A f\,ds.
\end{equation}

{}{} (b) Si $f\in\ba$ y $t\geq 0$ entonces $\int_0^t T_s f\,ds\in
D(\A)$ y
\[T_t f-f=\A\int_0^t T_sf\,ds.\]

{}{} (c) El generador infinitesimal $\A:D(\A)\subseteq\ba\to\ba$
es un operador cerrado y además $D(\A)$ es denso en $\ba.$
\end{teorema}
\begin{proof}
(a) Para $f\in D(\A)$ y $t\geq 0,$ por continuidad de $T_t$
tenemos
\begin{align*}
\lim_{\ h\to 0^+}\frac{1}{h}(T_hT_tf-T_tf)&=\lim_{\ h\to
0^+}\frac{1}{h}(T_{t+h}-T_tf)
=\lim_{\ h\to 0^+}T_t\bigl(\tfrac{1}{h}(T_hf-f)\bigr)\\
&=T_t\Bigl(\lim_{\ h\to 0^+}\tfrac{1}{s}(T_hf-f)\Bigr)=T_t\A f.
\end{align*}
Entonces $T_tf\in D(\A), \ \A T_t f=T_t\A f$ y $\dfrac{d^+(T_t
f)}{dt}=\A T_t f=T_t\A f.$

{} Falta ver que $\dfrac{d^-(T_t f)}{dt}=T_t\A f.$ Para $0<h\le t$
se tiene que
\begin{align*}
\frac{1}{-h}(T_{t-h}f-T_t f)-T_t\A f&=\frac{1}{h}(T_t
f-T_{t-h}f)-T_t\A f\\&=\frac{1}{h}(T_{t-h}T_h
f-T_{t-h}f)-T_{t-h}\A f+T_{t-h}\A f-T_{t-h}T_h\A f.
\end{align*}
Usando la definición del generador infinitesimal $\A$ y la
continuidad fuerte del semigrupo se obtiene entonces que
\begin{align*}
\left\|\frac{1}{-h}(T_{t-h}f-T_tf)-T_t\A
f\right\|&\le\left\|T_{t-h}\bigl(\tfrac{1}{h}(T_h f-f)-\A
f\bigr)\right\|+\left\|T_{t-h}(\A f-T_h f\A f)\right\|\\
&\le |\!|T_{t-h}|\!|\cdot\left\|\tfrac{1}{h}(T_h f-f)-\A
f\right\|+|\!|T_{t-h}|\!|\cdot\left\|\A f-T_h f\A f\right\|\\
&\le\left\|\tfrac{1}{h}(T_h f-f)-\A f\right\|+\left\|\A f-T_h f\A
f\right\|\xrightarrow[h\to 0^+]\,0,
\end{align*}
es decir
\[\frac{d^-(T_t f)}{dt}=\lim_{h\uparrow 0}\frac{1}{h}(T_{t+h}f-T_t f)
=\lim_{h\downarrow 0}\frac{1}{-h}(T_{t-h}f-T_t f)=T_t\A f.\] La
igualdad (\ref{semiprop1}) se sigue fácilmente usando que
\[\int_0^t\frac{d(T_s f)}{ds}\,ds=T_t f-T_0f=T_tf-f.\]
(b) Usando la continuidad y linealidad de los operadores $T_h, \
h>0,$ y el teorema fundamental del cálculo obtenemos
\begin{align*}
\frac{1}{h}&\,\biggl[T_h\!\int_0^t\!T_sf\,ds-\int_0^t
T_sf\,ds\biggr]\\
&=\frac{1}{h}\biggl[\int_0^tT_hT_sf\,ds-\int_0^t T_sf\,ds\biggr]\\
&=\frac{1}{h}\biggl[\int_0^tT_{h+s}f\,ds-\int_0^t
T_sf\,ds\biggr]\\
&=\frac{1}{h}\biggl[\int_h^{t+h}T_sf\,ds-\int_0^t
T_sf\,ds\biggr]\\
&=\frac{1}{h}\int_t^{t+h}T_sf\,ds-\frac{1}{h}\int_0^h
T_sf\,ds\xrightarrow[h\to 0^+]\,T_tf-T_0f=T_tf-f
\end{align*}
(c) Dado que para cada $f\in\ba$
\[f=T_0 f=\lim_{t\downarrow 0}\frac{1}{t}\int_0^tT_sf\,ds\]
la parte (b) implica que $D(\A)$ es denso en $\ba.$ Para ver que
$\A$ es cerrado, sea $(f_n)_{n\geq 1}\subset D(\A)$ tal que
$f_n\to f$ y $\A f_n\to g$ en $\ba.$ Por (\ref{semiprop1}) se
sigue que
\[T_t f_n-f_n=\int_0^tT_s\A f_n\,ds\]
para cada $t>0.$ Haciendo $n\to\infty$ se obtiene $T_t
f-f=\int_0^tT_sg\,ds.$ Dividiendo por $t>0$
\[\frac{1}{t}(T_tf-f)=\frac{1}{t}\int_0^tT_sg\,ds\xrightarrow[t\to 0^+]\,T_0g=g\]
lo que significa que $f\in D(\A)$ y $\A f=g.$
\end{proof}
{}
\begin{defi}
Sea $p(t,x,s,A)$ una función no negativa definida para $0\le t\le
s<+\infty,{} x\in\R^d$ y $A\in\B(\R^d)$ tal que
\begin{description}
  \item[(i)] \, $x\mapsto p(t,x,s,A)$ es Borel-medible, para cada
  $s\geq t, \ A\in\B(\R^d);$
  \item[(ii)] $A\mapsto p(t,x,s,A)$ es una medida de probabilidad, para cada $s\geq t, \ %
  x\in\R^d;$
  \item[(iii)] $p$ satisface la \textit{ecuación de Chapman-Kolmogorov}
\[p(t,x,s,A)=\int_{\R^d}p(u,y,s,A)\,p(t,x,u,dy), \ \ \ \ \text{para todo} \ \ t< u< s, \
  x\in\R^d, \ A\in\B(\R^d).\]
\end{description}
Diremos entonces que $p$ es una \textit{probabilidad de
transición,} o \textit{función de transición de probabilidad.}
\end{defi}
\begin{defi}
Si una probabilidad de transición $p$ sólo depende de $s-t,$ es
decir \[p(t,x,s,A)=p(0,x,s-t,A), \ \ \text{para todo} \ \ s\geq t,
\ x\in\R^d, \ A\in\B(\R^d),\] decimos que $p$ es
\textit{homogénea} y escribiremos $p(t,x,A)$ en vez de
$p(0,x,t,A), \ t\geq 0.$ En este caso la ecuación de
Chapman-Kolmogorov se escribe:
\[p(t+s,x,A)=\int_{\R^d}p(t,y,A)\,p(s,x,dy), \ \ \ t,s>0, \ x\in\R^d, \ A\in\B(\R^d).\]
\end{defi}

\begin{obse}
Es costumbre también requerir que la medida $p(0,x,\cdot)$
coincida con la medida de Dirac $\epsilon_x$ en $\R^d,$ y que
$t\mapsto p(t,x,A)$ sea continua en cero, esto es {}$\lim_{t\to
0^+}p(t,x,A)=\epsilon_x(A).$

{}{} Recuerde que la medida de Dirac $\epsilon_x$ se define como
$\epsilon_x(A):=\mathbf{1}_A(x).$
\end{obse}

\begin{prop}
Sea $p(t,x,A)$ una función de transición de probabilidad homogénea
definida para $t\geq 0, \ x\in\R^d$ y $A\in\B(\R^d).$ Entonces la
familia de operadores $(T_t)_{t\geq 0}$ definida por
\[(T_tf)(x):=\int_{\R^d}f(y)\,p(t,x,dy), \ \ \ x\in\R^d\]
es un $C_0-$semigrupo de contracciones sobre el espacio de Banach
\[B_0:=\Bigl\{f:\R^d\to\R \mid f \ \text{es acotada, borel
medible y} \ |\!|T_tf-f|\!|_\infty\xrightarrow[t\downarrow
0]\,0\,\Bigr\}\] con la norma
$|\!|f|\!|_\infty:=\sup_{x\in\R^d}|f(x)|.$
\end{prop}
\begin{proof}
La propiedad de semigrupo y la continuidad fuerte se obtienen
directamente de la ecuación de Chapman-Kolmogorov para el caso
homogéneo y de la definición de $B_0$ respectivamente. Dado que
\[\bigl|(T_tf)(x)\bigr|\le\int_{\R^d}|f(y)|\,p(t,x,dy)\le|\!|f|\!|_\infty,\]
claramente $T_t$ es una contracción para todo $t\geq 0.$
\end{proof}

\begin{defi}
Sea $(X_t)_{t\geq 0}$ un proceso estocástico con valores en $\R^d$
adaptado a una filtración $\{\F_t\}_{t\geq 0}.$ Diremos que
$(X_t)_{t\geq 0}$ es un $\F_t-$\textit{proceso de Markov} (resp.
$\F_t-$\textit{proceso de Markov homogéneo}) con probabilidad de
transición $p(t,x,s,A)$ (resp. probabilidad de transición
homogénea $p(t,x,A)$) si para todo $s>t$ y $A\in\B(\R^d)$ se tiene
\begin{equation}\label{markov1}
\mathbf{P}(X_s\in A\mid \F_t)=p(t,X_t,s,A)
\end{equation}
\[\text{(resp.} \ (\mathbf{P}(X_{s}\in A\mid \F_t)=p(s-t,X_t,A)).\]
\end{defi}

\begin{obse}
Si tomamos la esperanza con respecto a $X_t$ en (\ref{markov1}) se
obtiene
\[\mathbf{P}(X_s\in A\mid X_t)=p(t,X_t,s,A)\]
y esto permite interpretar a la probabilidad de transición
$p(t,x,s,A)$ como la probabilidad de que al tiempo $s$ el proceso
$X$ esté en $A$ dado que al tiempo $t$ estuvo en $x.$
\end{obse}

\begin{defi}
Una probabilidad de transición homogénea $p(t,x,A)$ se dice
\textit{estocásticamente continua} si
\begin{equation}\label{estcontinua}
\lim_{t\to 0^+}p(t,x,B(x,\varepsilon))=1
\end{equation}
para todo $\varepsilon>0$ y $x\in\R^d$ fijo. Si el límite en
(\ref{estcontinua}) se tiene uniformemente en $x$ para cada
$\varepsilon>0$, la probabilidad de transición se dice
\textit{uniformemente estocásticamente continua.}
\end{defi}

\begin{lema}
Sea $(X_t)_{t\geq 0}$ un $\F_t-$proceso de Markov homogéneo con
valores en $\R^d$ y probabilidad de transición $p(t,x,A).$ Si
$(X_t)_{t\geq 0}$ tiene trayectorias continuas a derecha entonces
$p(t,x,A)$ es estocásticamente continua.
\end{lema}
\begin{proof}
Sea $(t_n)_{n\geq 1}\subset\R_+$ una sucesión tal que $t_n\searrow
0$ y defina la sucesión de eventos $E_n:=\{X_{t_n}\in
B(x,\varepsilon)\}, \ n\geq 1.$ Dado que
$X_0=\lim_{n\to\infty}X_{t_n},$
\[\{X_0\in B(x,\varepsilon)\}
\subseteq\bigcup_{n=1}^\infty\bigcap_{k=n}^\infty
E_n=\liminf_{n\to\infty}E_n.\] Por lo tanto, para todo $x\in\R^d$
y $\varepsilon>0,$
\begin{align*}
\liminf_{n\to\infty}p(t_n,x,B(x,\varepsilon))
&=\liminf_{n\to\infty}\mathbf{P}(X_{t_n}\in B(x,\varepsilon)\mid X_0=x)\\
&\geq\mathbf{P}(\liminf_{n\to\infty}E_n\mid X_0=x)\\
&\geq\mathbf{P}(X_0\in B(x,\varepsilon)\mid X_0=x)=1.
\end{align*}
\end{proof}

\begin{prop}
Sea $(X_t)_{t\geq 0}$ como en el lema anterior, y sea
$f:\R^d\to\R$ acotada y continua. Entonces, para cada $x\in\R^d,$
se tiene
\[\lim_{\ t\to 0^+}(T_t f)(x)=f(x).\]
\end{prop}
\begin{proof}
Debido a la continuidad de $f,$ para cada $x\in\R^d$ y para
$\delta>0,$ existe $B(x,\varepsilon)$ tal que si $y\in
B(x,\varepsilon)$ entonces $|f(y)-f(x)|<\delta.$ Por lo tanto
\[(T_tf)(x)-f(x)=\int_{B(x,\varepsilon)}[f(y)-f(x)]\,p(t,x,dy)
+\int_{B(x,\varepsilon)^c}[f(y)-f(x)]\,p(t,x,dy),\] y de esto se
sigue que
\begin{equation}\label{descorofuertcont}
|(T_t f)(x)-f(x)|\le\delta p(t,x,B(x,\varepsilon))+2|\!|f|\!|_\infty
p(t,x,B(x,\varepsilon)^c).
\end{equation}
Esta desigualdad junto con el lema anterior prueban la
proposición.
\end{proof}

También de la desigualdad (\ref{descorofuertcont}) se deduce el
\begin{coro}
Si $p(t,x,A)$ es uniformemente estocásticamente continua entonces
el semigrupo $(T_t)_{t\geq 0}$ es fuertemente continuo sobre
$\C_b(\R^d).$
\end{coro}

\begin{prop}\label{homogen}
Sea $(X_t)_{t\geq 0}$ un $\F_t-$proceso de Markov con espacio de
estados $\R^d$ y probabilidad de transición
\[p(t,x,s,A), \ \ \ s\geq t\geq 0, \ x\in\R^d, \ A\in\B(\R^d).\]
Entonces $(t,X_t)_{t\geq 0}$ es un $\F_t-$proceso de Markov
\textbf{homogéneo} con espacio de estados $[0,\infty)\times\R^d$ y
función de transición de probabilidad
\[q(s,(t,x),C\times D)=p(t,x,t+s,D)\cdot\mathbf{1}_C(t+s), \ \ \
C\in\B\bigl([0,+\infty)\bigr), \ D\in\B(\R^d).\]
\end{prop}
\begin{proof}
Veamos que $q$ así definida es una función de transición de
probabilidad: la condición \textbf{(i)} es inmediata. Dado que la
aplicación $C\mapsto\mathbf{1}_{C}(t+s)$ es precisamente la
\textit{medida de Dirac} $\epsilon_{t+s}$ concentrada en $t+s,$
entonces
\[C\times D\mapsto p(t,x,t+s,D)\cdot\mathbf{1}_{C}(t+s), \ \ \
C\in\B\bigl([0,+\infty)\bigr), \ D\in\B(\R^d)\] genera una medida
producto (y que es una medida de probabilidad) sobre
$[0,\infty)\times\R^d,$ para cada $t\geq s, \ x\in\R^d,$ lo cual
prueba \textbf{(ii)}. Para ver que $q$ satisface la ecuación de
Chapman-Kolmogorov usaremos la igualdad
\[\int_{[0,\infty)}f(r)\,d\epsilon_t(r)=f(t), \ \ \ \ \ \text{para todo} \ \
f:[0,\infty)\to\R, \ \ t\geq 0\] (ver demostración en
\cite{blanco} ejemplo 2.2.2, pág. 76). Si $s,r>0$ entonces
\begin{align*}
\int_{[0,\infty)\times\R^d}&q(s,(u,y),C\times
D)\,q(r,(t,x),du\otimes
dy)\\&=\int_{[0,\infty)\times\R^d}p(u,y,u+s,D)\cdot\mathbf{1}_C(u+s)\,
p(t,x,t+r,dy)\,d\epsilon_{t+r}(u)\\
&=\int_{\R^d}p(t,x,t+r,dy)\int_{[0,\infty)}p(u,y,u+s,D)\cdot%
\mathbf{1}_C(u+s)\,d\epsilon_{t+r}(u)\\
&=\int_{\R^d}p(t,x,t+r,dy)\,p(t+r,y,t+r+s,D)\cdot\mathbf{1}_C(t+r+s)\\
&=p(t,x,t+s+r,D)\cdot\mathbf{1}_C(t+r+s)=q(s+r,(t,x),C\times D).
\end{align*}
Por último, dado que $p$ satisface (\ref{markov1}), para $s>t$ se
tiene que
\begin{align*}
q(s-t,(t,X_t),C\times D)&=p(t,X_t,s,D)\cdot\mathbf{1}_C(s)
=\mathbf{P}(X_s\in
D\mid\F_t)\cdot\mathbf{1}_C(s)\\&=\mathbf{P}((s,X_s)\in C\times
D\mid\F_t).
\end{align*}
\end{proof}
\begin{obse}
Si $(X_t)_{t\geq 0}$ es un proceso de Markov, el semigrupo
$(T_s)_{s\geq 0}$ asociado a las probabilidades de transición del
proceso de Markov homogéneo $(t,X_t)_{t\geq 0}$ está dado por
\begin{align*}
(T_s f)(t,x)&=\int_{[0,\infty)\times\R^d}f(r,y)q(s,(t,x),dr\otimes
dy)\\
&=\int_{\R^d}\int_{[0,\infty)}f(r,y)\,p(t,x,t+s,dy)\,d\epsilon_{t+s}(r)\\
&=\int_{\R^d}f(t+s,y)\,p(t,x,t+s,dy), \ \ \ \ s\geq 0.
\end{align*}
\end{obse}
{} Para más resultados sobre semigrupos de operadores y su
conexión con procesos de Markov consultar
\cite{karlin,pazy,todorovic,tudor}.

\newpage\section[Ecuaciones diferenciales estocásticas y conexión con EDPs lineales de segundo orden]{Ecuaciones diferenciales
estocásticas y conexión con EDPs\\lineales de segundo orden}

\subsection{Ecuaciones diferenciales estocásticas}
Sea $(W_t)_{t\geq 0}$ un movimiento Browniano $m-$dimensional
definido sobre un espacio de probabilidad $(\Omega,\F,\mathbf{P})$
y sean
\[b:[0,+\infty)\times\R^d\rightarrow\R^d, \ \ \ \ \
\sigma:[0,+\infty)\times\R^d\rightarrow\R^{d\times m}\] funciones
medibles. Nuestro objetivo en esta sección es mostrar existencia y
unicidad de soluciones de \textit{ecuaciones diferenciales
estocásticas} (EDEs) de la forma
\begin{equation}\label{ede1}
dX_t=b(t,X_t)\,dt+\sigma(t,X_t)\,dW_t, \ \ \ X_0=\xi
\end{equation}
donde la \textit{condición inicial} $\xi$ es una v.a.
$d-$dimensional independiente de $(W_t)_{t\geq 0}.$ Hay que
recordar que la ecuación (\ref{ede1}) se debe interpretar
exclusivamente como una ecuación integral, a saber
\begin{equation}\label{ede2}
X_t=\xi+\int_0^tb(s,X_s)\,ds+\int_0^t\sigma (s,X_s)\,dW_s, \ \ \ \
t\geq 0,
\end{equation}
o de modo equivalente, usando componentes
\[X_t^i=\xi_i+\int_0^t W_i(s,X_s)\,ds+\sum_{j=1}^m
\int_0^t\sigma_{ij}(s,X_s)\,dW_s^j \ \ \ \ \ t\geq 0, \ \ 1\le
i\le d.\] A lo largo de esta sección $\{\F_t\}_{t\geq 0}$ denotará
la filtración generada por $\xi$ y $(W_t)_{t\geq 0},$ aumentada
con los conjuntos $\mathbf{P}-$nulos de $\Omega,$ es decir
\[\F_t=\sigma\bigl(\{\xi,W_s: 0\le s\le t\}\cup\mathcal{N}\bigr), \ \ \ t\geq 0 \] donde
$\mathcal{N}=\{E\subseteq\Omega : \exists G\in \F,\, G\supseteq E
\text{ y } \mathbf{P}(G)=0\}.$

\begin{defi} Diremos que un proceso
$(X_t)_{t\geq 0}$ con valores en $\R^d$ y trayectorias continuas
c.s. es solución (fuerte) de la EDE (\ref{ede1}) si satisface
\begin{enumerate}
  \item $(X_t)_{t\geq 0}$ es \textit{progresivamente medible} con respecto a $\{\F_t\}_{t\geq 0},$
  \item $\displaystyle\int_0^t\bigl(|b(s,X_s)|+|\!|\sigma(s,X_s)|\!|^2\bigr)\,ds<+\infty$ \
  c.s. para todo $t\geq 0,$
  \item $\displaystyle{X_t=\xi+\int_0^tb(s,X_s)\,ds+\int_0^t\sigma (s,X_s)\,dW_s,}
  \ \ \ t\geq 0$ \ c.s.
\end{enumerate}
\end{defi}
{} (recuerde que para $z\in\R^{d\times m}, \
|\!|z|\!|:=Tr[zz^*]^{1/2}).$

{}{} Si la matriz $\sigma$ (que suele llamarse \textit{matriz de
dispersión)} es idénticamente cero, la ecuación (\ref{ede2}) se
reduce a una \textit{ecuación integral ordinaria} (no estocástica,
excepto tal vez por la condición inicial) de la forma
\begin{equation}\label{eio}
X_t=X_0+\int_0^t b(s,X_s)\,ds, \ \ \ \ t\geq 0
\end{equation}
Usualmente, en la teoría de tales ecuaciones se impone la
hipótesis de que el campo vectorial $b(t,x)$ satisfaga una
condición de \textit{Lipschitz} local en la variable de espacio
$x\in\R^d$ y sea acotada sobre subconjuntos compactos de
$[0,\infty)\times\R^d.$ Estas condiciones garantizan que para un
$t>0$ suficientemente pequeño, las iteraciones de
\textit{Picard-Lindelöf}
\[X_t^{(0)}\equiv\xi, \ \ \ \ X_t^{(n+1)}=\xi+\int_0^t b(s,X_s^{(n)})\,ds, \ \ \ \ \ n\geq 0\]
convergen a la solución de (\ref{eio}), y que esta solución sea
única. En ausencia de estas condiciones la ecuación podría no
tener solución o tener infinitas soluciones. Por ejemplo, la
ecuación integral (determinística) unidimensional
\begin{equation}\label{eio2}
X_t=\int_0^t|X_s|^\alpha\,ds
\end{equation}
tiene una única solución para $\alpha\geq 1,$ a saber: $X_t\equiv
0.$ Sin embargo, para $0<\alpha<1,$ todas las funciones de la
forma
\[X_t=
\begin{cases}
  0, & 0\le t\le s,\\
  \displaystyle{\left(\frac{t-s}{\beta}\right)}^\beta, &t\geq s.
\end{cases}\]
con $\beta=1/(1-\alpha)$ y $s\geq 0$ arbitrario, son solución de
(\ref{eio2}).

Parece entonces razonable desarrollar la teoría de EDEs imponiendo
condiciones del tipo \textit{Lipschitz} en los coeficientes:

{}{} Diremos que $b$ y $\sigma$ satisfacen

{}{}\textbf{(a)} \textit{la condición de Lispchitz global,} si
para cada $T>0$ existe una constante $K_T>0$ tal que para todo
$t\in[0,T]$ y para todo par $x,x'\in\R^d$
\begin{equation}\label{lipglobal}
|b(t,x)-b(t,x')|+|\!|\sigma(t,x)-\sigma(t,x')|\!|\le K_T|x-x'|
\end{equation}

{}\textbf{(b)}  \textit{la condición de Lispchitz local,} si para
cada $T>0$ y para cada $n\geq 1$ existe una constante $K_{T,n}>0$
tal que para todo $t\in[0,T]$ y para todo par $x,x'\in B[0,n]$
\begin{equation}\label{liplocal}
|b(t,x)-b(t,x')|+|\!|\sigma(t,x)-\sigma(t,x')|\!|\le K_{T,n}|x-x'|
\end{equation}
Para demostrar existencia y unicidad bajo estas hipótesis también
necesitaremos que los coeficientes satisfagan la condición de
\textit{crecimiento lineal:} para cada $T>0$ existe una constante
$K_T>0$ tal que
\begin{equation}\label{clineal}
|b(t,x)|+|\!|\sigma(t,x)|\!| \le K_T\bigl(1+|x|\bigr), \ \ \ \forall
t\in[0,T], \ \ \forall x\in\R^d
\end{equation}
Esta condición nos permite establecer la siguiente estimación
\textit{a priori} de los momentos de soluciones de EDEs

\begin{lema}\label{apriori}
Sea $(X_t)_{t\geq 0}$ solución (fuerte) de la EDE (\ref{ede1}) y
suponga que existe un entero $p\geq 1$ tal que
$E\bigl(|\xi|^{2p}\bigr)<+\infty.$ Si los coeficientes $b$ y
$\sigma$ satisfacen la condición de crecimiento lineal
(\ref{clineal}), entonces
\begin{align*}
E\biggl[\sup_{\,t\in [0,T]}|X_t|^{2p}\biggr]&\le
Ce^{CT}\bigl\{1+E\bigl(|\xi|^{2p}\bigr)\bigr\}, \ \ \ \forall
T>0,\\
E\biggl[\sup_{\,t\le s\le T}|X_s-X_t|^{2p}\biggr]&\le
C'(T-t)^p\bigl\{1+E\bigl(|\xi|^{2p}\bigr)\bigr\}, \ \ \ t\in [0,T]
\end{align*}
donde $C,C'$ son constantes que dependen sólo de $T,K_T$ y $p.$
\end{lema}
\begin{proof}
Usando la desigualdad $|a+b+c|^{2p}\le
3^{2p-1}\bigl(|a|^{2p}+|b|^{2p}+|c|^{2p}\bigr),$ se obtiene para
$t\in [0,T]$
\begin{equation*}
|X_t|^{2p}\le 3^{2p-1}\biggl(|\xi|^{2p}+\Bigl|\int_0^t
b(r,X_r)\,dr\Bigr|^{2p} +\Bigl|\int_0^t
\sigma(r,X_r)\,dW_r\Bigr|^{2p}\biggr)
\end{equation*}
De la desigualdad de Hölder se sigue
\begin{align*}
\Bigl|\int_0^t
b(s,X_s)\,ds\Bigr|^{2p}=\Biggl[\sum_{i=1}^d\,\biggl(\int_0^t
W_i(s,X_s)\,ds\biggr)^2\Biggr]^p&\le t^p\bigg[\int_0^t|b(s,X_s)|^2
ds\biggr]^p\\
&\le t^{2p-1}\int_0^t|b(s,X_s)|^{2p}\,ds
\end{align*}
luego
\begin{equation}\label{apriori2}
|X_t|^{2p}\le
3^{2p-1}\biggl(|\xi|^{2p}+t^{2p-1}\int_0^t|b(r,X_r)|^{2p}\,dr+\Bigl|\int_0^t
\sigma(r,X_r)\,dW_r\Bigr|^{2p}\biggr)
\end{equation}
Si $\tau_n:=\inf\{s\geq 0:|X_s|\geq n\},$ el proceso
$H_s^n=\sigma(s,X_s)\mathbf{1}_{[0,\tau_n)}(s),$ $s\geq 0$
satisface
\begin{align*}
E\int_0^t |\!|H_s^n|\!|^{2p}\,ds\le E\int_0^t
|\!|\sigma(s,X_s)|\!|^{2p}\mathbf{1}_{[0,\tau_n)}(s)\,ds &\le
E\int_0^t
K_t^{2p}(1+|X_s|)^{2p}\mathbf{1}_{[0,\tau_n)}(s)\,ds\\
&\le tK_t^{2p}(1+n)^{2p}
\end{align*}
para todo $t\geq 0,$ luego podemos aplicar la desigualdad
(\ref{des1}) y obtener
\begin{align*}
E\biggl[\sup_{\,0\le s\le t}\Bigl|\int_0^{s\wedge\tau_n}
\sigma(r,X_r)\,dW_r\Bigr|^{2p}\biggr]&=E\biggl[\sup_{\,0\le s\le
t}\Bigl|\int_0^s
\sigma(r,X_r)\mathbf{1}_{[0,\tau_n)}(r)\,dW_r\Bigr|^{2p}\biggr]\\
&\le C_p t^{p-1}E\int_0^t
|\!|\sigma(r,X_r)|\!|^{2p}\mathbf{1}_{[0,\tau_n)}(r)\,dr
\end{align*}
con $C_p=d^{p-1}[4p^3/(2p-1)]^p.$ Tomando $t=s\wedge\tau_n$ en
(\ref{apriori2}) y usando la anterior desigualdad,
\begin{align*}
|X_{s\wedge\tau_n}|^{2p}\le 3^{2p-1}&\biggl(|\xi|^{2p}+
(s\wedge\tau_n)^{2p-1}\int_0^{s\wedge\tau_n}|b(r,X_r)|^{2p}\,dr+\Bigl|\int_0^{s\wedge\tau_n}
\sigma(r,X_r)\,dW_r\Bigr|^{2p}\biggr)\\
E\biggl[\sup_{\,0\le s\le t}|X_{s\wedge\tau_n}|^{2p}\biggr]&\le
3^{2p-1}\Bigl\{E\bigl(|\xi|^{2p}\bigr)+ t^{2p-1}\!E\int_0^t
|b(r,X_r)|^{2p}\mathbf{1}_{[0,\tau_n)}(r)\,dr\\
&\hspace{5.2cm}+C_p t^{p-1}E\!\int_0^t
|\!|\sigma(r,X_r)|\!|^{2p}\mathbf{1}_{[0,\tau_n)}(r)\,dr\Bigr\}\\
\le 3^{2p-1}&\Bigl\{E\bigl(|\xi|^{2p}\bigr)+(t^{2p-1}+C_p
t^{p-1})K_T^{2p}E\!\int_0^t
(1+|X_r|)^{2p}\mathbf{1}_{[0,\tau_n)}(r)\,dr\Bigr\}\\
\le
3^{2p-1}&\biggl\{E\bigl(|\xi|^{2p}\bigr)+2^{2p-1}L\,\Bigl\{T+E\!\int_0^t
|X_r|^{2p}\mathbf{1}_{[0,\tau_n)}(r)\,dr\Bigr\}\biggr\}, \ \ \ \
t\in [0,T]
\end{align*}
donde $L=(T^{2p-1}+C_p T^{p-1})K_T^{2p}.$ Dado que
$|X_r|^{2p}\mathbf{1}_{[0,\tau_n)}(r)\le\sup_{\,0\le s\le
r\wedge\tau_n}|X_s|^{2p}$ y
\[E\biggl[\sup_{\,0\le s\le t}|X_{s\wedge\tau_n}|^{2p}\biggr]
=E\biggl[\sup_{\,0\le s\le t\wedge\tau_n}|X_s|^{2p}\biggr]\] del
teorema de Fubini se tiene que la aplicación
$\varphi_n(t)=E\bigl[\sup_{\,0\le s\le
t\wedge\tau_n}|X_s|^{2p}\bigr]$ satisface
\[\varphi_n(t)\le C\,\Bigl\{E\bigl(|\xi|^{2p}\bigr)+1+\int_0^t\varphi_n(r)\,dr\Bigr\}, \
\ \ \ t\in [0,T]\] con $C=3^{2p-1}(1\vee 2^{2p-1}L\vee
2^{2p-1}LT).$ Por el lema de Gronwall (ver \cite{tudor}), {}
$\varphi_n(t)\le Ce^{Ct}\bigl\{1+E\bigl(|\xi|^{2p}\bigr)\bigr\}$
para $t\in [0,T], \ \forall n\geq 1,$ y por el lema de Fatou
\begin{align*}
E\biggl[\sup_{\,0\le s\le t}|X_s|^{2p}\biggr]
&=E\biggl[\,\liminf_{n\to\infty}\,\Bigl(\sup_{\ 0\le s\le
{t\wedge\tau_n}}|X_s|^{2p}\Bigr)\biggr]\\
&\le \liminf_{n\to\infty}\, E\biggl[\sup_{\,0\le s\le
{t\wedge\tau_n}}|X_s|^{2p}\biggr]\le
Ce^{Ct}\bigl\{1+E\bigl(|\xi|^{2p}\bigr)\bigr\}, \ \ \ t\in [0,T].
\end{align*}
lo cual prueba la primera desigualdad. Para la segunda
desigualdad, si $0\le s<u\le\ t\le T$ entonces
\[|X_u-X_s|^{2p}\le
2^{2p-1}(u-s)^{2p-1}\int_s^u|b(r,X_r)|^{2p}\,dr+2^{2p-1}\Bigl|\int_s^u
\sigma(r,X_r)\,dW_r\Bigr|^{2p}\] Tomando ahora
$\tau_n:=\inf\{u\geq s:|X_u|\geq n\},$ usando un razonamiento
análogo al anterior y la desigualdad $E\bigl(|X_r|^{2p}\bigr)\le
Ce^{CT}\bigl\{1+E\bigl(|\xi|^{2p}\bigr)\bigr\}$ se obtiene
\begin{align*}
|X_{u\wedge\tau_n}-X_s|^{2p}\le
2^{2p-1}&[(u\wedge\tau_n)-s]^{2p-1}\int_s^{u\wedge\tau_n}|b(r,X_r)|^{2p}\,dr
+2^{2p-1}\Bigl|\int_s^{u\wedge\tau_n}
\sigma(r,X_r)\,dW_r\Bigr|^{2p}\\
\sup_{s\le u\le t}|X_{u\wedge\tau_n}-X_s|^{2p}
&\le 2^{2p-1}(t-s)^{2p-1}\int_s^{t\wedge\tau_n}|b(r,X_r)|^{2p}\,dr\\
&\hspace{3.8cm}+2^{2p-1}\sup_{s\le u\le t}\Bigl|\int_s^u
\sigma(r,X_r)\mathbf{1}_{[s,\tau_n)}(r)\,dW_r\Bigr|^{2p}\\
E\biggl[\sup_{\, s\le u\le t}|X_{u\wedge\tau_n}-X_s|^{2p}\biggr]
&\le 4^{2p-1}(t-s)^{2p-1}\int_s^t
K_T^{2p}(1+Ce^{CT})\bigl\{1+E\bigl(|\xi|^{2p}\bigr)\bigr\}\,dr\\
&\hspace{1.2cm}+4^{2p-1}C_p(t-s)^{p-1}\int_s^t
K_T^{2p}(1+Ce^{CT})\bigl\{1+E\bigl(|\xi|^{2p}\bigr)\bigr\}\,dr\\
&=C'(t-s)^p\bigl\{1+E\bigl(|\xi|^{2p}\bigr)\bigr\}
\end{align*}
con $C'=4^{2p-1}K_T^{2p}(T^p+C_p)(1+Ce^{CT}),$ y de nuevo por el
lema de Fatou se sigue el resultado.
\end{proof}
{} Los resultados que veremos a continuación bajo condiciones del
tipo \textit{Lipschitz} constituyen la teoría estándar de EDEs
desarrollada por K. Itô.

\subsubsection{Unicidad y existencia de soluciones}
{} Diremos que la EDE (\ref{ede1}) satisface la \textit{unicidad
en trayectoria} o \textit{unicidad fuerte} si dos soluciones $X$ y
$X'$ son \textit{indistinguibles,} es decir, satisfacen
\[\mathbf{P}\bigl(X_t=X'_t, \ \forall t\geq 0)=1\]
Debido a que dichos procesos tienen trayectorias continuas c.s.,
bastará con mostrar que uno es \textit{modificación} del otro, es
decir que $\mathbf{P}(X_t=X'_t)=1,$ para todo $t\geq 0.$
\begin{teorema}\label{teorema1}
Sean $b^i:[0,\infty)\times\R^d\rightarrow\R^d,\
\sigma^i:[0,\infty)\times\R^d\rightarrow\R^{d\times m}$ funciones
medibles y localmente acotadas, $i=1,2.$ Sea $D\subset\R^d$ un
conjunto abierto y $T>0.$ Supongamos que
\[\sigma^1(t,x)=\sigma^2(t,x), \ \ b^1(t,x)=b^2(t,x), \ \ \ \forall(t,x)\in[0,T]\times D\]
y que existe una constante $K_T>0$ tal que para todo $t\in[0,T]$ y
para todo par $x,x'\in D$
\begin{equation}\label{lipglobal2}
\begin{split}
|b^i(t,x)-b^i(t,x')|+|\!|\sigma^i(t,x)-\sigma^i(t,x')|\!|\le
K_T|x-x'|\\
|\!|\sigma^i(t,x)|\!|+|b^i(t,x)|\le K_T\bigl(1+|x|\bigr), \ \ \ i=1,2.
\end{split}
\end{equation}
Sea $X^i$ solución (fuerte) de la EDE (\ref{ede1}) con condición
inicial $X_0^i$ cuadrado integrable y coeficientes $b^i, \sigma^i$
y sean
\[\tau_i=\inf\{t\geq 0:X_t^i\notin D\}, \ \ \ i=1,2\]
\[\tau=\inf\{t\geq 0:X_t^1\notin D \ \text{o} \ X_t^2\notin
D\}=\tau_1\wedge\tau_2.\] Si $X_0^1(\omega)=X_0^2(\omega)$ para
casi todo $\omega$ para el cual $X_0^1(\omega)\in D$ o
$X_0^2(\omega)\in D,$ es decir
\[\mathbf{P}\bigl(\bigl((X_0^1\in D)\cup(X_0^2\in D)\bigr)\cap(X_0^1=X_0^2)\bigr)=1\]
entonces $\mathbf{P}\bigl(X_{t\wedge\tau}^1=X_{t\wedge\tau}^2, \
\forall t\in[0,T]\bigr)=1.$

{}{} Si además (\ref{lipglobal2}) se cumple para todo par
$x,x'\in\R^d,$ entonces $\tau_1\wedge T=\tau_2\wedge T=\tau\wedge
T$ c.s. y en consecuencia, en el intervalo $[0,T]$ los procesos
$X^1$ y $X^2$ salen al mismo tiempo $\tau$ de $D$ y son iguales
hasta ese momento.
\end{teorema}
\begin{proof}
Para $t\in [0,T]$ se tiene
\begin{align*}
X_{t\wedge\tau}^1-X_{t\wedge\tau}^2 &=
\int_0^{t\wedge\tau}\bigl(b^1(s,X_s^1)-b^2(s,X_s^1)\bigr)\,ds
+\int_0^{t\wedge\tau}\bigl(b^2(s,X_s^1)-b^2(s,X_s^2)\bigr)\,ds \\
&+\int_0^{t\wedge\tau}\bigl(\sigma^1(s,X_s^1)-\sigma^2(s,X_s^1)\bigr)\,dW_s
+\int_0^{t\wedge\tau}\bigl(\sigma^2(s,X_s^1)-\sigma^2(s,X_s^2)\bigr)\,dW_s
\end{align*}
Como \
$\displaystyle{\int_0^{t\wedge\tau}\!\bigl(b^1(s,X_s^1)-b^2(s,X_s^1)\bigr)\,ds=0},$
\ y por el teorema \ref{intiguales}
\[\int_0^{t\wedge\tau}\bigl(\sigma^1(s,X_s^1)-\sigma^2(s,X_s^1)\bigr)\,ds=0\]
resulta que
\[X_{t\wedge\tau}^1-X_{t\wedge\tau}^2
=\int_0^{t\wedge\tau}\bigl(b^2(s,X_s^1)-b^2(s,X_s^2)\bigr)\,ds
+\int_0^{t\wedge\tau}\bigl(\sigma^2(s,X_s^1)-\sigma^2(s,X_s^2)\bigr)\,dW_s\]
Tomando valor esperado, usando la desigualdad $(a+b)^2\le 2
(|a|^2+|b|^2)$ y la desigualdad de Hölder obtenemos
\begin{align*}
E&\bigl(|X_{t\wedge\tau}^1-X_{t\wedge\tau}^2|^2\bigr)\\[0.2cm]
&\le
2E\biggl[\Bigl|\int_0^{t\wedge\tau}\bigl(b^2(s,X_s^1)-b^2(s,X_s^2)\bigr)\,ds\Bigr|^2\biggr]
+2E\biggl[\Bigl|\int_0^{t\wedge\tau}\bigl(\sigma^2(s,X_s^1)
-\sigma^2(s,X_s^2)\bigr)\,dW_s\Bigr|^2\biggr]\\
&\le 2tE\!\int_0^{t\wedge\tau}|b^2(s,X_s^1)-b^2(s,X_s^2)|^2\,ds
+2E\biggl[\Bigl|\int_0^t\bigl(\sigma^2(s,X_s^1)
-\sigma^2(s,X_s^2)\bigr)\mathbf{1}_{[0,\tau)}(s)\,dW_s\Bigr|^2\biggr]
\end{align*}
Gracias a la condición de crecimiento lineal en $D$ y al lema
\ref{apriori}, podemos aplicar Isometría de Itô a las integrales
estocásticas y obtener
\begin{align*}
E&\bigl(|X_{t\wedge\tau}^1-X_{t\wedge\tau}^2|^2\bigr)\\[0.2cm]
&\le 2tE\!\int_0^{t\wedge\tau}|b^2(s,X_s^1)-b^2(s,X_s^2)|^2\,ds
+2E\int_0^t|\!|\sigma^2(s,X_s^1)
-\sigma^2(s,X_s^2)|\!|^2\mathbf{1}_{[0,\tau)}(s)\,ds\\
&\le 2K_T^2(1+T)\, E\!\int_0^t
|X_s^1-X_s^2|^2\mathbf{1}_{[0,\tau)}(s)\,ds
=2K_T^2(1+T)\int_0^t E\bigl[|X_s^1-X_s^2|^2\mathbf{1}_{[0,\tau)}(s)\bigr]\,ds\\
&=2K_T^2(1+T)\int_0^t
E\bigl(|X_{s\wedge\tau}^1-X_{s\wedge\tau}^2|^2\bigr)\,ds, \ \ \
t\in [0,T]
\end{align*}
y por el lema de Gronwall se sigue que
$E\bigl(|X_{t\wedge\tau}^1-X_{t\wedge\tau}^2|^2\bigr)=0,$ para
todo $t\in[0,T],${} lo que implica
$\mathbf{P}\bigl(X_{t\wedge\tau}^1=X_{t\wedge\tau}^2\bigr)=1,$ \
$\forall t\in[0,T],$ y puesto que $X^1$ y $X^2$ tienen
trayectorias continuas c.s., se tiene que
\[\mathbf{P}\bigl(X_{t\wedge\tau}^1=X_{t\wedge\tau}^2, \ \forall t\in[0,T]\bigr)=1.\]
Supongamos ahora que (\ref{lipglobal2}) se cumple para todo par
$x,x'\in\R^d.$ En este caso el razonamiento anterior se repite
para la diferencia $X_{t\wedge\tau_1}^1-X^2_{t\wedge\tau_1}$
obteniendo
\[\mathbf{P}\bigl(X_{t\wedge\tau_1}^1=X_{t\wedge\tau_1}^2, \ \forall t\in[0,T]\bigr)=1.\]
Esto quiere decir que en en el intervalo $[0,T]$ los dos procesos
coinciden hasta que $X^1$ sale de $D,$ luego $\tau_1\wedge T\le
\tau_2\wedge T,$ y por simetría en el razonamiento $\tau_2\wedge
T\le \tau_1\wedge T.$
\end{proof}

\begin{coro}[Unicidad en trayectoria]\label{unicidad}
Sean $b$ y $\sigma$ que satisfacen la condición de Lipschitz
local. Entonces se cumple la unicidad en trayectoria para la EDE
(\ref{ede1}).
\end{coro}
\begin{proof}
Sean $X^1,X^2$ dos soluciones y sea
\[\tau_n:=\inf\{t\geq 0:
|X_t^1|\geq n \ \ \text{o} \ \ |X_t^2|\geq n\}, \ \ \ n\geq 1.\]
Entonces por el teorema anterior se tiene \
$\mathbf{P}\bigl(X_{t\wedge\tau_n}^1=X_{t\wedge\tau_n}^2, \
\forall t\in[0,T]\bigr)=1,$ y por lo tanto para $t\in [0,T],$
\begin{align*}
\mathbf{P}(X_t^1\neq X_t^2)\le\mathbf{P}(\tau_n\le t) &\le
\mathbf{P}\Bigl(\sup_{\ 0\le s\le T}|X_s^1|\geq
n\Bigr)+\mathbf{P}\Bigl(\sup_{\ 0\le s\le
T}|X_s^2|\geq n\Bigr)\\
&=\mathbf{P}\Bigl(\max_{\ 0\le s\le T}|X_s^1|\geq
n\Bigr)+\mathbf{P}\Bigl(\max_{\ 0\le s\le T}|X_s^2|\geq
n\Bigr)\xrightarrow[n\to\infty]\,0
\end{align*}
puesto que $X^1, X^2$ tiene trayectorias continuas c.s. En
consecuencia $X^1$ es una \textit{versión} de $X^2,$ y de nuevo
por continuidad son \textit{indistinguibles.}
\end{proof}
\begin{teorema}[Existencia]\label{existencia}
Sean $b:[0,\infty)\times\R^d\rightarrow\R^d,\
\sigma:[0,\infty)\times\R^d\rightarrow\R^{d\times m}$ funciones
medibles que satisfacen la condición de Lipschitz local
(\ref{liplocal}) y la condición de crecimiento lineal
(\ref{clineal}). Entonces existe una solución (fuerte) de la EDE
(\ref{ede1}).
\end{teorema}
\begin{proof}
\textbf{Paso 1.} Primero supongamos que
$E\bigl(|\xi|^2\bigr)<+\infty$ y que los coeficientes satisfacen
la condición de \textit{Lipschitz} global (\ref{lipglobal})\,:

{}{} Para cada $T>0$ denotemos con $\M_T^2(\R^d)$ el espacio
vectorial de los procesos $(X_t)_{0 \leq t\leq T}$ \textit{prog.
medibles} y con valores en $\R^d$ tales que
\begin{equation}\label{norma1}
|\!|X|\!|_{\M_T^2(\R^d)}^2 :=E\!\int_0^T\!|X_t|^2 dt<+\infty
\end{equation} y definamos la aplicación
$\Phi_T:\M_T^2(\R^d)\to\M_T^2(\R^d)$ que a cada $X\in\M_T^2(\R^d)$
le asigna el proceso
\begin{equation}
\Phi_T(X)_t:=\xi+\int_0^t b(s,X_s)\,ds+\int_0^t\sigma
(s,X_s)\,dW_s, \ \ \ \ t\in [0,T].
\end{equation}
Veamos que esta aplicación posee un punto fijo en $\M_T^2(\R^d).$
Dados $X,X'\in\M_T^2(\R^d),$ usando la desigualdad de Hölder y la
Isometría de Itô se obtiene
\begin{align*}
E\bigl(&|\Phi_T(X)_t-\Phi_T(X')_t|^2\bigr)\\
&\le
2E\biggl[\Bigl|\int_0^t\bigl(b(s,X_s)-b(s,X'_s)\bigr)\,ds\Bigr|^2\biggr]
+2E\biggl[\Bigl|\int_0^t\bigl(\sigma(s,X_s)
-\sigma(s,X'_s)\bigr)\,dW_s\Bigr|^2\biggr]\\
&\le 2tE\!\int_0^t|b(s,X_s)-b(s,X'_s)|^2\,ds
+2E\int_0^t|\!|\sigma(s,X_s)
-\sigma(s,X'_s)|\!|^2\,ds\\
&\le 2K_T^2(1+T)\,E\!\int_0^t\!|X_s-X'_s|^2\,ds, \ \ \ \ t\in [0,T]
\end{align*}
y debido a que $\int_s^T e^{-\lambda t}e^{\lambda
s}\,dt=\frac{1}{\lambda}\bigl(1-e^{-\lambda (T-s)}\bigr)\le
\frac{1}{\lambda}$ \,para todo $\lambda>0,$ se tiene que

\begin{align*}
E\!\int_0^T\!&e^{-\lambda t}|\Phi_T(X)_t-\Phi_T(X')_t|^2\,dt
=\int_0^T e^{-\lambda t}E\bigl(|\Phi_T(X)_t-\Phi_T(X')_t|^2\bigr)\,dt\\
&\le 2K_T^2(1+T)\int_0^T e^{-\lambda t}
E\!\int_0^t\!|X_s-X'_s|^2\,ds\,dt\\
&= 2K_T^2(1+T)\, E\!\int_0^T \!\int_0^t e^{-\lambda
t}|X_s-X'_s|^2\,ds\,dt\\
&=2K_T^2(1+T)\,E\!\int_0^T\!\int_s^T e^{-\lambda
t}|X_s-X'_s|^2\,dt\,ds\\
&= 2K_T^2(1+T)\,E\!\int_0^T \biggl(\int_s^T e^{-\lambda
t}e^{\lambda s}\,dt\biggr)e^{-\lambda
s}|X_s-X'_s|^2\,ds\\
&\le \frac{1}{\lambda}2K_T^2(1+T)\,E\!\int_0^T e^{-\lambda
s}|X_s-X'_s|^2\,ds
\end{align*}
Si $\lambda>2K_T^2(1+T),$ la aplicación $\Phi_T$ resulta ser una
\textit{contracción estricta} de $\M_T^2(\R^d)$ en sí mismo con la
norma
\begin{equation}\label{norma2}
|\!|X|\!|_\lambda^2:=E\!\int_0^T e^{-\lambda s}|X_s|^2\,ds
\end{equation}
Esta norma satisface
\[e^{-\lambda T}|\!|X|\!|_{\M_T^2(\R^d)}^2\le |\!|X|\!|_\lambda^2\le |\!|X|\!|_{\M_T^2(\R^d)}^2\]
y dado que $\M_T^2(\R^d)$ es es un espacio de Banach con la norma
usual (\ref{norma1}), también lo es con la norma (\ref{norma2}), y
por lo tanto $\Phi_T$ posee un único punto fijo, que notaremos por
$X^T,$ y es solución de la EDE (\ref{ede1}) en el intervalo
$[0,T].$

{} Si $T\le T'$ el proceso $X^{T'}$ restringido a $[0,T],$ es
decir $(X_t^{T'})_{t\in [0,T]},$ es una \textit{versión} de
$X^T,$ pues también pertenece a $\M_T^2(\R^d)$ y es punto fijo de
$\Phi_T.$ Podemos entonces definir el proceso continuo y
\textit{prog. medible} $(X_t)_{t\geq 0}$ por
\[X_t:=X_t^T \ \text{si} \ t\in [0,T]\]
y que claramente es solución de (\ref{ede1}).

{}{}\textbf{Paso 2.} Sea ahora $\xi$ no necesariamente cuadrado
integrable. Para cada $n\geq1$ defina
$\xi^n:=\xi\mathbf{1}_{\{|\xi|\le n\}}.$ Debido a que $|\xi^n|\le
n,$ se tiene por el paso 1 la existencia de un proceso continuo y
adaptado $(X_t^n)_{t\geq 0}$ tal que
\begin{equation}\label{paso2}
X_t^n=\xi^n+\int_0^t b(s,X_s^n)\,ds+\int_0^t\sigma
(s,X_s^n)\,dW_s, \ \ \ \ t\geq 0, \ \ \ n\geq 1
\end{equation}
y por el lema \ref{apriori}
\[E\biggl[\sup_{\,t\in [0,T]}|X_t^n|^2\biggr]<+\infty, \ \ \ \forall
T>0, \ \forall n\geq 1\] Sean $\lambda_n=\mathbf{1}_{\{|\xi|\le
n\}}$ y $T>0.$ Dado que $(\xi^{n+1}-\xi^n)\lambda_n=0$ y
$\lambda_n$ es $\F_0-$medible, por el lema \ref{lemazeta} se tiene
que para $t\in [0,T]$
\[(X_t^{n+1}-X_t^n)\lambda_n=\int_0^t \bigl[b(s,X_s^{n+1})-b(s,X_s^n)\bigr]\lambda_n\,ds
+\int_0^t\bigl[\sigma(s,X_s^{n+1})-\sigma(s,X_s^n)\bigr]\lambda_n\,dW_s\]
de donde resulta
\[|X_t^{n+1}-X_t^n|^2\lambda_n\le 2t\int_0^t|b(s,X_s^{n+1})-b(s,X_s^n)|^2\lambda_n\,ds
+2\Bigl|\int_0^t\bigl[\sigma(s,X_s^{n+1})-\sigma(s,X_s^n)\bigr]\lambda_n\,dW_s\Bigr|^2\]
tomando valor esperado obtenemos
\[E\bigl(|X_t^{n+1}-X_t^n|^2\lambda_n\bigr)
\le
2K^2(1+T)\int_0^tE\bigl(|X_s^{n+1}-X_s^n|^2\lambda_n\bigr)\,ds\] y
por el lema de Gronwall
$E\bigl(|X_t^{n+1}-X_t^n|^2\lambda_n\bigr)=0, \ \forall
t\in[0,T].$ Por lo tanto, para todo $t\in [0,T]$ se tiene
$X_t^{n+1}\lambda_n=X_t^n\lambda_n$ c.s. lo que implica por
continuidad de las trayectorias que
$\mathbf{P}\bigr(X_t^{n+1}\lambda_n=X_t^n\lambda_n, \ \forall
t\in[0,T]\bigl)=1, \ \forall n\geq 1.$ En particular
\[\mathbf{P}\Bigl(\sup_{\ t\in [0,T]}|X_t^{n+1}-X_t^n|^2>0\Bigr)
\le\mathbf{P}\bigl(|\xi|\le n\bigr)\xrightarrow[n\to\infty]\,0.\]
Sea $\{n_k\}_{k\geq 1}$ una sucesión tal que
$\mathbf{P}\bigl(|\xi|\le n_k\bigr)\le\frac{1}{2^k}, \ \forall
k\geq 1.$ Puesto que
\[\mathbf{P}\Bigl(\sup_{\ t\in [0,T]}|X_t^{n_{k+1}}-X_t^{n_k}|^2>0\Bigr)\le\frac{1}{2^k},\] se
obtiene por el lema de Borel-Cantelli que
\[\mathbf{P}\biggl(\limsup_{k\to\infty}\,\Bigl(\sup_{\ t\in [0,T]}|X_t^{n_{k+1}}
-X_t^{n_k}|^2>0\Bigr)\biggr)=0\] o en forma equivalente
\[\mathbf{P}\biggl(\liminf_{k\to\infty}\,\Bigl(\sup_{\ t\in [0,T]}|X_t^{n_{k+1}}
-X_t^{n_k}|^2=0\Bigr)\biggr)=1\] Es decir que para casi todo
$\omega\in\Omega$ existe un $k_0(\omega)\geq 1$ suficientemente
grande tal que $X_t^{n_{k+1}}(\omega)=X_t^{n_k}(\omega), \ \forall
k\geq k_0,$ para todo $t\in [0,T].$ Como $T>0$ es arbitrario,
podemos definir el proceso continuo y \textit{prog. medible}
$(X_t)_{t\geq 0}$ por
\[X_t(\omega):=\lim_{k\to\infty}X_t^{n_k}(\omega)=X^{n_{k_0}}(\omega), \ \ t\geq 0\]
luego $X_t=X_t^{n_{k_0}}$ c.s. para todo $t\geq 0.$ Si asumimos
además que $k_0(\omega)\geq|\xi(\omega)|$ entonces
$\xi^{n_{k_0}}=\xi.$ Tomando $n=n_{k_0}$ en (\ref{paso2})
obtenemos
\[X_t=\xi+\int_0^t b(s,X_s)\,ds+\int_0^t\sigma
(s,X_s)\,dW_s, \ \ \ \ t\geq 0.\]

{}{}\textbf{Paso 3.} Supongamos ahora que $b$ y $\sigma$
satisfacen la condición de \textit{Lipschitz} local
(\ref{liplocal}). Para $n\geq 1$ definimos
$W_n(t,x):=b(t,\Pi_n(x))$ y $\sigma_n(t,x):=\sigma(t,\Pi_n(x)),$
donde $\Pi_n:\R^d\to B[0,n]$ está dada por
\[\Pi_n(x):=\frac{n\wedge|x|}{|x|}x=
\begin{cases}
\,x               & \text{si}\ \,|x|\le n\\
\frac{n}{|x|}\,x  & \text{si}\ \,|x|> n
\end{cases} \]
Dado que $|\Pi_n(x)-\Pi_n(x')|\le|x-x'|,$ $W_n$ y $\sigma_n$
satisfacen la condición de \textit{Lipschitz} global
(\ref{lipglobal}) con constante de \textit{Lipschitz} $K_{T,n}.$
Por lo tanto, por el paso anterior, para cada $n\geq 1$ existe un
proceso continuo y \textit{prog. medible} $(X_t^n)_{t\geq 0}$ tal
que
\begin{equation}
X_t^n=\xi+\int_0^t W_n(s,X_s^n)\,ds+\int_0^t\sigma_n
(s,X_s^n)\,dW_s, \ \ \ \ t\geq 0
\end{equation}
Sea $\tau_n:=\inf\{t\geq 0:|X_t^n|\geq n\}, \ n\geq 1.$ Por el
teorema \ref{teorema1} resulta que c.s.
\[\tau_n\le\inf\{t\geq 0:|X_t^{n+1}|\geq n\}\le\tau_{n+1}, \ \ X_t^n=X_t^{n+1} \ \text{si} \ \
0\le t\le\tau_n.\] Probaremos que $\tau_n\nearrow\infty$ c.s., lo
cual permite definir el proceso continuo y \textit{prog. medible}
$(X_t)_{t\geq 0}$ por
\[X_t:=X_t^n \ \ \text{si} \ \ 0\le t\le \tau_n,\]
y que es solución de (\ref{ede1}). Definamos
$Y_t^n:=e^{-|\xi|}X_t^n,$ \ $\xi':=e^{-|\xi|}\xi,$
\[b'_n(t,y):=e^{-|\xi|}W_n(t,e^{|\xi|}y), \ \
\ \sigma'_n(t,y):=e^{-|\xi|}\sigma_n(t,e^{|\xi|}y)\] Dado que
$e^{-|\xi|}$ es $\F_0-$medible, por el lema \ref{lemazeta} tenemos
\begin{align*}
e^{-|\xi|}X_t^n&=e^{-|\xi|}\xi+\int_0^t
e^{-|\xi|}W_n(s,X_s^n)\,ds+\int_0^te^{-|\xi|}\sigma
(s,X_s^n)\,dW_s\\
&=e^{-|\xi|}\xi+\int_0^t
e^{-|\xi|}W_n(s,e^{|\xi|}Y_s^n)\,ds+\int_0^te^{-|\xi|}\sigma_n
(s,e^{|\xi|}Y_s^n)\,dW_s\\[0.2cm]
\text{es decir} \ \ \ \ \ \ \ \ \ \ Y_t^n&=\xi'+\int_0^t
b'_n(s,Y_s^n)\,ds+\int_0^t \sigma'_n (s,Y_s^n)\,dW_s, \ \ \ \
t\geq 0
\end{align*}
Como $b'_n$ y $\sigma'_n$ satisfacen la condición de crecimiento
(\ref{clineal}) con la misma constante $K_T$ y $|\xi'|\le 1,$ por
el lema \ref{apriori}
\[E\biggl[\sup_{\,t\in [0,T]}|Y_t^n|^2\biggr]\le
C_T<\infty, \ \forall n\geq 1\] $(C_T$ no depende de $n)$. Usando
la desigualdad de Markov, para todo $\delta,T>0$ se tiene
\begin{align*}
\mathbf{P}(\tau_n\le T)&=\mathbf{P}\Bigl(\sup_{\ t\in [0,T]}|X_t^n|\geq n\Bigr)=\mathbf{P}\Bigl(\sup_{\ t\in [0,T]}|Y_t^n|\geq ne^{-|\xi|}\Bigr)\\
&\le \mathbf{P}\Bigl(\sup_{\ t\in [0,T]}|Y_t^n|\geq
n\delta\Bigr)+\mathbf{P}(e^{-|\xi|}\le\delta)\le\frac{C_T}{n^2\delta^2}
+\mathbf{P}(e^{-|\xi|}\le\delta)
\end{align*}
luego $\limsup_{n\to\infty}\mathbf{P}(e^{-|\xi|}\le\delta)\searrow
0$ cuando $\delta\searrow 0.$ En consecuencia
$\mathbf{P}(\tau_n\le T)\to 0$ cuando $n\to\infty$ para todo $T>0$
y puesto que $0\le\tau_1\le\tau_2\le\cdots\le\tau_n\le\cdots\,,$
resulta que $\tau_n\nearrow\infty$ c.s.
\end{proof}

\begin{obse}\label{obsecondinicial}
Los resultados anteriores se extienden de manera obvia a EDEs de
la forma
\begin{equation}\label{edetxi}
X_s=\xi+\int_t^s b(r,X_r)\,dr+\int_t^s\sigma (r,X_r)\,dW_r, \ \ \
\ s\geq t
\end{equation}
con $t\geq 0$ fijo y $\xi$ una variable aleatoria independiente
del M.B. $(W_s-W_t)_{s\geq t}$ \,(ver proposición
\ref{cambiovar}). Notaremos la solución de (\ref{edetxi}) por
$(X_s^{t,\xi})_{s\geq t}. $
\end{obse}

\subsubsection{Propiedad de Markov de las soluciones}
De aquí en adelante asumiremos que los coeficientes son
globalmente \textit{Lipschitz} y satisfacen la condición de
\textit{crecimiento lineal.} Para cada
$(t,x)\in[0,+\infty)\times\R^d$ notaremos $(\X)_{s\geq t}$ la
solución de la EDE
\begin{equation}\label{edemarkov}
\X=x+\int_t^s b(r,X_r^{t,x})\,dr+\int_t^s\sigma
(r,X_r^{t,x})\,dW_r, \ \ \ \ s\geq t
\end{equation}

La siguiente proposición establece dependencia continua de
$X^{t,x}$ con respecto a $t$ y $x$ y será fundamental para
demostrar la propiedad de Markov de soluciones de EDEs.

\begin{prop}\label{propmarkov1} Sean
$0\le t'\le t\le T$ y $p\geq 1.$ Para todo par $x,x'\in\R^d$
existe una constante $C$ que depende de $T$ y de $K_T$ tal que
\[E\biggl[\sup_{\,t\le s\le
T}|X_s^{t',x'}-X_s^{t,x}|^{2p}\biggr]\le
C\bigl\{|x-x'|^{2p}+|t'-t|^{p}(1+|x'|)^{2p}\bigr\}\]
\end{prop}
\begin{proof}
Dado que
\[X_t^{t',x'}=x'+\int_{t'}^t b(r,X_r^{t',x'})\,dr+\int_{t'}^t\sigma
(r,X_r^{t',x'})\,dW_r\] para todo $t\le u\le s \le T$ se tiene que
\[X_u^{t',x'}-X_u^{t,x}=(X_t^{t',x'}-x)
+\int_t^u\bigl[b(r,X_r^{t',x'})-b(r,X_r^{t,x})\bigr]\,dr
+\int_t^u\bigl[\sigma(r,X_r^{t',x'})-\sigma(r,X_r^{t,x})\bigr]\,dW_r.\]
Por el lema \ref{apriori} se tiene que
\begin{align*}
E\bigl(|X_t^{t',x'}-x|^{2p}\bigr)
&\le 2^{2p-1} E\bigl(|X_t^{t',x'}-X_{t'}^{t',x'}|^{2p}+|x'-x|^{2p}\bigr)\\
&\le
2^{2p-1}\bigl\{C'(t-t')^p\bigl(1+|x'|^{2p}\bigr)+|x'-x|^{2p}\bigr\},
\end{align*}
y por la desigualdad de Hölder se tiene que
\begin{align*}
\sup_{\, t\le u\le s}\Bigl|\int_t^u
&\bigl[b(r,X_r^{t',x'})-b(r,X_r^{t,x})\bigr]\,dr
+\int_t^u\bigl[\sigma(r,X_r^{t',x'})-\sigma(r,X_r^{t,x})\bigr]\,dW_r\Bigr|^{2p}\\
&\le
2^{2p-1}(s-t)^{2p-1}\int_t^s|b(r,X_r^{t',x'})-b(r,X_r^{t,x})|^{2p}\,dr\\
&\hspace{2cm}+2^{2p-1}\sup_{t\le u\le s}
\Bigl|\int_t^u\bigl[\sigma(r,X_r^{t',x'})-\sigma(r,X_r^{t,x})\bigr]\,dW_r\Bigr|^{2p}.
\end{align*}
Usando la desigualdad (\ref{des1}), la condición de
\textit{Lipschitz} y el teorema de Fubini obtenemos
\begin{align*}
E\biggl[\sup_{\,t\le u\le s}|X_u^{t',x'}-X_u^{t,x}|^{2p}\biggr]
&\le 4^{2p-1}\bigl\{C'(t-t')^p\bigl(1+|x'|^{2p}\bigr)+|x'-x|^{2p}\bigr\}\\
&\hspace{1cm}+4^{2p-1}(T^{2p-1}+C_pT^{p-1})K_T^{2p}\int_t^s
E\bigl(|X_r^{t',x'}-X_r^{t,x}|^{2p}\bigr)\,dr
\end{align*}
y dado que $E\bigl[|X_r^{t',x'}-X_r^{t,x}|^{2p}\bigr]\le
E\bigl[\sup_{\,t\le u\le r}|X_r^{t',x'}-X_r^{t,x}|^{2p}\bigr],$
podemos aplicar lema de Gronwall y obtener
\[E\biggl[\sup_{\,t\le u\le s}|X_u^{t',x'}-X_u^{t,x}|^{2p}\biggr]
\le
4^{2p-1}e^{C_0(t-s)}\bigl\{C'(t-t')^p\bigl(1+|x'|^{2p}\bigr)+|x'-x|^{2p}\bigr\},\]
con $C_0=4^{2p-1}(T^{2p-1}+C_pT^{p-1})K_T^{2p}.$
\end{proof}

\begin{lema}\label{lemamarkov}
Sea $\psi:\R^d\times\Omega\to\R$ una función medible y acotada tal
que $\psi(x,\cdot)$ es independiente de $\F_t$ para cada
$x\in\R^d.$ Sea $\theta$ una variable aleatoria con valores en
$\R^d$ y $\F_t-$medible. Entonces
\[E[\psi(\theta(\cdot),\cdot)|\F_t]=g(\theta),\]
donde $g(x)=E[\psi(x,\cdot)].$
\end{lema}
\begin{proof}
Por el teorema de las clases monótonas (ver \cite{tudor}) es
suficiente tomar $\psi$ de la forma
$\psi(x,\omega)=\alpha(x)\beta(\omega),$ donde $\alpha:\R^d\to\R,
\ \beta:\Omega\to\R$ son medibles, acotadas y $\beta$ es
independiente de $\F_t.$ Para $\Omega'\in\F_t$ se tiene que
\begin{align*}
\int_{\Omega'}\psi(\theta(\cdot),\cdot)\,d\mathbf{P}
&=\int_{\Omega'}\alpha(\theta)\beta\,d\mathbf{P}
=E[\beta]\int_{\Omega'}\alpha(\theta)\,d\mathbf{P}\\
&=\int_{\Omega'}\alpha(\theta)E[\beta]\,d\mathbf{P}
=\int_{\Omega'}g(\theta)\,d\mathbf{P},
\end{align*}
donde $g(x)=\alpha(x)E[\beta].$
\end{proof}

\begin{prop}\label{propmarkov}
Sea $\xi$ una variable aleatoria independiente de $(W_t)_{t\geq
0}$ y sea $(X_s)_{s\geq 0}$ la solución de la EDE
\[X_s=\xi+\int_0^s b(r,X_r)\,dr+\int_0^s\sigma (r,X_r)\,dW_r, \ \ \
\ s\geq 0.\] Entonces el proceso $(X_s)_{s\geq 0}$ es un
$\F_t-$proceso de Markov con función de transición de probabilidad
\[p(t,x,s,A):=\mathbf{P}(\X\in A), \ \ \ \ s\geq t, \ x\in\R^d, \ A\in\B(\R^d)\]
donde $(X_s^{t,x})_{s\geq t}$ es la solución de la EDE
(\ref{edemarkov}) con condición inicial $X_t^{t,x}=x.$
\end{prop}
\begin{proof}
La aplicación $A\mapsto\mathbf{P}(X_s^{t,x}\in A)$ es la medida
transportada por $X_s^{t,x},$ luego es claramente una medida de
probabilidad sobre $\R^d.$ 

{}{} Veamos que para $s\geq t$ y $A\in\B(\R^d), \
\mathbf{P}(X_s\in A|\F_t)=p(t,X_t,s,A)\,:$

{}{} Sea $\psi(x,\omega):=\mathbf{1}_A(\X(\omega)).$ Por la
construcción de $(\X)_{s\geq t}$ en la demostración del teorema
\ref{existencia} se tiene que $\X$ es medible con respecto a la
$\sigma-$álgebra
\[\sigma\bigl(\{W_r-W_t: s\le r\le t\}\cup\mathcal{N}\bigr).\]
En consecuencia, por definición del M.B., $\psi(x,\cdot)$ es
independiente de $\F_t,$ y dado que los procesos $(X_s)_{s\geq t}$
y $(X_s^{t,X_t})_{s\geq t}$ satisfacen la misma EDE
\[X_s=X_t+\int_t^s b(r,X_r)\,dr+\int_t^s\sigma (r,X_r)\,dW_r, \ \ \
\ s\geq t\] resulta por unicidad que $X_s=X_s^{t,X_t}$ c.s. para
$s\geq t,$ luego
\[\psi(X_t(\omega),\omega)=\mathbf{1}_A(X_s^{t,X_t(\omega)}(\omega))
=\mathbf{1}_A(X_s(\omega)).\] Aplicando el lema \ref{lemamarkov}
con $\theta=X_t$ obtenemos
\[\mathbf{P}(X_s\in
A|\F_t)=E\bigl[\mathbf{1}_A(X_s)|\F_t\bigr]=E\bigl[\psi(X_t(\cdot),\cdot)|\F_t\bigr]=g(X_t)\]
donde $g(x)=E[\psi(x,\cdot)]=E[\mathbf{1}_A(\X)]=\mathbf{P}(\X\in
A).$

{} Por último veamos que $p$ satisface la \textit{ecuación de
Chapman-Kolmogorov}: dado que $p(t,x,s,\cdot)$ es una medida de
probabilidad, para todo $A\in\B(\R^d)$ se cumple
\[\int_{\R^d}\mathbf{1}_A(y)p(t,x,s,dy)=\int_{A}p(t,x,s,dy)=p(t,x,s,A)
=E\bigl[\mathbf{1}_A(\X)\bigr]\] y aproximando por funciones
simples, para toda función $f:\R^d\to\R$ acotada se cumple
\[E\bigl[f(\X)\bigr]=\int_{\R^d}f(y)p(t,x,s,dy).\]
Si $0\le t<u<s,$ tomando $f(y)=p(u,y,s,A)$ y usando
$p(u,X_u^{t,x},s,A)=\mathbf{P}(X_s^{t,x}\in A|\F_u)$ se obtiene
\begin{align*}
\int_{\R^d}p(u,y,s,A)p(t,x,u,dy)&=E\bigl[p(u,X_u^{t,x},s,A)\bigr]
=E\bigl[\mathbf{P}(X_s^{t,x}\in A|\F_u)\bigr]\\
&=\mathbf{P}(X_s^{t,x}\in A)=p(t,x,s,A).
\end{align*}
\end{proof}
{} Se puede probar, usando el concepto de \textit{unicidad en distribución}
(ver e.g. \cite{tudor}), que si los coeficientes $b$ y $\sigma$ no
dependen de $t,$ el proceso $(X_s)_{s\geq 0}$ solución de la EDE
(\ref{ede1}) es un proceso de Markov \textit{homogéneo} con
probabilidad de transición $p(t,x,A)=\mathbf{P}(X_t^x\in A),$
donde $(X_t^x)_{t\geq 0}$ es la solución de la EDE
\[X_t^x=x+\int_0^tb(X_s)\,ds+\int_0^t\sigma(X_s)\,dW_s, \ \ \ \ t\geq 0\]
con condición inicial $X_0^x=x.$

\begin{defi}
Un proceso de Markov $d-$dimensional con función de transición de
probabilidad $p(t,x,s,A)$ es llamado un \textit{proceso de
difusión} si

{}{}\textbf{(i)} para todo $\varepsilon > 0, \ t\geq 0, \
x\in\R^d,$
\begin{equation}\label{difusion1}
\lim_{h\to 0^+}\frac{1}{h}\int_{|y-x|>\varepsilon} p(t,x,t+h,dy)=0
\end{equation}

{}\textbf{(ii)} existen un vector $d-$dimensional $b(t,x)$ y una
matriz $a(t,x)$ de tamaño $d\times d$ tales que para todo
$\varepsilon > 0, \ t\geq 0, \ x\in\R^d,$
\begin{align}
\lim_{h\to 0^+}\frac{1}{h}\int_{|y-x|<\varepsilon}
(y_i-x_i)\,p(t,x,t+h,dy)&=W_i(t,x), \ \ \ \ \ 1\le i\le d
\label{difusion2}\\
\lim_{h\to 0^+}\frac{1}{h}\int_{|y-x|<\varepsilon}
(y_i-x_i)(y_j-x_j)\,p(t,x,t+h,dy)&=a_{ij}(t,x), \ \ \ \ \ 1\le
i,j\le d \label{difusion3}
\end{align}
Al vector $b(t,x)$ lo llamamos \textit{coeficiente de drift} y a
la matriz $a(t,x)$ la llamamos \textit{matriz de difusión.}
\end{defi}
\begin{lema}
Las siguientes condiciones implican la condiciones \textbf{(i)} y
\textbf{(ii)}:

{}{}\textbf{(i')} para algún $\delta>0, \ t\geq 0, \ x\in\R^d,$
\begin{equation}\label{lemadifusion1}
\lim_{h\to 0^+}\frac{1}{h}\int_{\R^d}
|x-y|^{2+\delta}p(t,x,t+h,dy)=0
\end{equation}

{}\textbf{(ii')} para cualquier $t\geq 0, \ x\in\R^d,$
\begin{align}
\lim_{h\to 0^+}\frac{1}{h}&\int_{\R^d}
(y_i-x_i)\,p(t,x,t+h,dy)=W_i(t,x), \ \ \ \ \ 1\le i\le d
\label{lemadifusion2}\\
\lim_{h\to 0^+}\frac{1}{h}&\int_{\R^n}
(y_i-x_i)(y_j-x_j)\,p(t,x,t+h,dy)=a_{ij}(t,x), \ \ \ \ \ 1\le
i,j\le d \label{lemadifusion3}
\end{align}
\end{lema}
\begin{proof}
Usando (\ref{lemadifusion1}) tenemos
\[\frac{1}{h}\int_{|y-x|>\varepsilon} p(t,x,t+h,dy)
\le\frac{1}{\varepsilon^{2+\delta}}\int_{\R^d}|y-x|^{2+\delta}p(t,x,t+h,dy)\xrightarrow[h\to
0^+]\,0\] que es (\ref{difusion1}). Por (\ref{lemadifusion1})
también tenemos, para $k=1,2,$
\[\frac{1}{h}\int_{|y-x|>\varepsilon} |y-x|^k p(t,x,t+h,dy)
\le\frac{1}{\varepsilon^{2+\delta-k}}\int_{\R^d}|y-x|^{2+\delta}p(t,x,t+h,dy)\xrightarrow[h\to
0^+]\,0\] y dado que
\begin{align*}
\int_{|y-x|<\varepsilon}&(y_i-x_i)\,p(t,x,t+h,dy)-\int_{\R^d}
(y_i-x_i)\,p(t,x,t+h,dy)\\
&=\int_{|y-x|\geq\varepsilon} (y_i-x_i)\,p(t,x,t+h,dy)
\le\int_{|y-x|\geq\varepsilon} |y-x|\,p(t,x,t+h,dy)
\end{align*}
\begin{align*}
\int_{|y-x|<\varepsilon}&(y_i-x_i)(y_j-x_j)\,p(t,x,t+h,dy)-\int_{\R^d}
(y_i-x_i)(y_j-x_j)\,p(t,x,t+h,dy)\\
&=\int_{|y-x|\geq\varepsilon} (y_i-x_i)(y_j-x_j)\,p(t,x,t+h,dy)\\
&\le \int_{|y-x|\geq\varepsilon} |y-x|^2\,p(t,x,t+h,dy)\\
\end{align*}
se tiene que (\ref{lemadifusion2}) y (\ref{lemadifusion3})
implican (\ref{difusion2}) y (\ref{difusion3}).
\end{proof}
\begin{teorema}\label{difusion}
Suponga que $b,\sigma$ son continuas en
$(t,x)\in[0,\infty)\times\R^d$ y satisfacen la condiciones de
Lipschitz global y crecimiento lineal. Entonces la solución de la
EDE (\ref{ede1}) es un proceso de difusión con drift $b(t,x)$ y
matriz de difusión {} $a(t,x)=\sigma(t,x)\sigma(t,x)^*.$
\end{teorema}
\begin{proof}
Como vimos en la demostración de la proposición \ref{propmarkov},
\[E\bigl[f(X_{t+h}^{t,x}-x)\bigr]=\int_{\R^d}f(y-x)p(t,x,t+h,dy), \ \ \ \ \forall h>0,\]
para toda función $f(z)$ continua con $|f(z)|\le K(1+|z|^\alpha)$
para algún $K,\alpha>0.$ En vista del lema anterior es suficente
probar
\begin{align}
\frac{1}{h}E\bigl(|X_{t+h}^{t,x}-x|^4\bigr)&\xrightarrow[h\to 0^+]\,0\label{demdifusion1}\\
\frac{1}{h}E\bigl[X_{t+h}^{t,x}-x\bigr]&\xrightarrow[h\to 0^+]\,b(t,x)\label{demdifusion2}\\
\frac{1}{h}E\bigl[\bigl(X_{t+h}^i-x_i\bigr)\bigl(X_{t+h}^j-x_j\bigr)\bigr]
&\xrightarrow[h\to 0^+]\,a_{ij}(t,x)\label{demdifusion3}
\end{align}
donde $X_{t+h}^i$ es la $i-$ésima componente de $X_{t+h}^{t,x}, \
1\le i\le d.$ Del lema \ref{apriori} se tiene que
\[\frac{1}{h}E\bigl(|X_{t+h}^{t,x}-X_t^{t,x}|^4\bigr)
\le\frac{1}{h}E\bigl(|X_{t+h}^{t,x}-x|^4\bigr) \le Ch^2(1+|x|^4)\]
lo cual da (\ref{demdifusion1}). De la condición de crecimiento
lineal de $\sigma$ y el lema \ref{apriori} se tiene que
\[E\int_t^{t+h}|\!|\sigma(s,X_s^{t,x})|\!|^2\,ds<+\infty\]
luego \,$E\!\displaystyle\int_t^{t+h}\!\sigma(t,X_s^{t,x})\,dW_s =
0$\, y
\[\frac{1}{h}E\bigl[X_{t+h}^{t,x}-x\bigr]=\frac{1}{h}E\int_t^{t+h}b(s,\X)\,ds
=\int_0^1 E\bigl[b(t+hr,X_{t+hr}^{t,x})\bigr]\,dr.\] Para cada
$h>0$ y $0\le s\le 1$ se define la variable aleatoria con valores
en $\R^d$
\[Y_h(r):=b\bigl(t+hr,X_{t+hr}^{t,x}\bigr).\] Del lema \ref{apriori} se sigue
que
\[X_{t+hr}^{t,x}\xrightarrow[h\to 0^+]{L^2}X_t^{t,x}=x, \ \ \ 0\le r\le
1\] luego por continuidad de $b(t,x)$ se tiene
\[Y_h(r)\xrightarrow[h\to 0]{L^2}b(t,x).\]
Por la desigualdad de Jensen se tiene en particular que
\[E[Y_h(r)]\xrightarrow[h\to 0]\,b(r,x), \ \ \forall
r\in[0,1].\] Usando de nuevo desigualdad de Jensen, el lema
\ref{apriori} y la condición de crecimiento lineal sobre $b(r,x),$
y tomando $0\le h\le 1$ se obtiene
\begin{align*}
\Bigl|E\bigl[b\bigl(t+hr,X_{t+hr}^{t,x}\bigr)\bigr]\Bigr| \le
E\Bigl[\bigl|b\bigl(t+hr,X_{t+hr}^{t,x}\bigr)\bigr|\Bigr]&\le
K_T\bigl\{1+E\bigl(|X_{t+hr}^{t,x}|\bigr)\bigr\}\\
&\le K_T\bigl\{1+C^{1/2}e^{CT/2}(1+|x|)\bigr\}
\end{align*} con $T=t+1,$ luego podemos aplicar convergencia dominada a la familia de
funciones $E[Y_h(s)], \ 0\le h\le 1,$ y obtener
\[\frac{1}{h}E\bigl[X_{t+h}^{t,x}-x\bigr]
=\int_0^1 E\bigl[Y_h(s)\bigr]\,ds\xrightarrow[h\to 0]\,\int_0^1
b(t,x)\,ds=b(t,x)\] lo cual prueba (\ref{demdifusion2}). Para
probar (\ref{demdifusion3}) notaremos $X_s^{t,x}$ simplemente por
$X_s$ y su $i-$ésima componente por $X_s^i, \ 1\le i\le d.$ Usando
fórmula de Itô con $F(t,z)=z_iz_j$ obtenemos
\[\frac{1}{h}\bigl\{E\bigl(X_{t+h}^i X_{t+h}^j\bigr)-x_i
x_j\bigr\}=\frac{1}{h}E\int_t^{t+h}\bigl[X_s^i W_j(s,X_s)+X_s^j
W_i(s,X_s)+a_{ij}(s,X_s)\bigr]\,ds\] (de nuevo la esperanza de la
integral estocástica es cero). Con un razonamiento similar al
anterior se demuestra
\[\frac{1}{h}\bigl\{E\bigl(X_{t+h}^i X_{t+h}^j\bigr)-x_i
x_j\bigr\}\xrightarrow[h\to 0^+]\,x_i W_j(t,x)+x_j
W_i(t,x)+a_{ij}(t,x).\] Se sigue entonces que
\begin{align*}
\lim_{\ h\to 0^+}&
\frac{1}{h}E\bigl[\bigl(X_{t+h}^i-x_i\bigr)\bigl(X_{t+h}^j-x_j\bigr)\bigr]\\
=&\,x_i W_j(t,x)+x_j W_i(t,x)+a_{ij}(t,x)-x_i\lim_{h\to
0^+}E\bigl(X_{t+h}^j-x_j\bigr)-x_j \lim_{h\to
0^+}E\bigl(X_{t+h}^i-x_i\bigr)\\
=&\,a_{ij}(t,x)
\end{align*}
donde hemos usado (\ref{demdifusion2}) en la última igualdad.
\end{proof}
\begin{obse}
Recíprocamente, se puede ver que si el coeficiente de
\textit{drift} y la matriz de \textit{difusión} de un proceso de
difusión $(X_t)_{t\in [0,T]}$ satisfacen ciertas condiciones
(entre otras, que $a(t,x)$ sea definida positiva), entonces existe
un M.B. $d-$dimensional $(W_t)_{t\geq 0}$ y una matriz
$\sigma(t,x)$ de tamaño $d\times m$ tal que $a=\sigma\sigma^*$ y
$X$ es solución de la EDE
\[X_t=\xi+\int_0^t b(s,X_s)\,ds+\int_0^t\sigma (s,X_s)\,dW_s, \ \ \
\ t\in [0,T].\] Para una demostración de dicho resultado ver
\cite{gihman}.
\end{obse}

Cuando los coeficientes $b$ y $\sigma$ dependen de $t,$ el proceso
de Markov solución de la EDE (\ref{ede1}) no es necesariamente
homogéneo. Sin embargo, por la proposición \ref{homogen} del
capítulo anterior el proceso $(s,X_s)_{s\geq 0}$ con espacio de
estados $[0,\infty)\times\R^d$ es un proceso de Markov homogéneo
con probabilidad de transición
\[q(s,(t,x),D\times C)=p(t,x,t+s,C)\cdot\mathbf{1}_D(t+s)\]
y semigrupo $(T_s)_{s\geq 0}$ dado por
\[\bigl(T_sf\bigr)(t,x)=\int_{\R^d}f(t+s,y)\,p(t,x,t+s,dy)
=E\bigl[f\bigl(t+s,X_{t+s}^{t,x}\bigr)\bigr], \ \ \ \ s\geq 0.\]
Si $f:\R^d\to\R$ no depende de $t,$ con ayuda del teorema
\ref{difusion} se puede calcular de manera explícita el generador
infinitesimal
\[\bigl(\A f\bigr)(t,x)=\lim_{h\to 0^+}\frac{\bigl(T_h f\bigr)(t,x)-f(x)}{h}, \ \ \ t\geq 0.\]
En efecto, dado que $\bigl(T_h
f\bigr)(t,x)=\int_{\R^d}f(y)\,p(t,x,t+h,dy),$
\begin{equation}\label{generador1}
\bigl(\A f\bigr)(t,x)=\lim_{h\to
0^+}\frac{1}{h}\int_{\R^d}\bigl[f(y)-f(x)\bigr]\,p(t,x,t+h,dy),
\end{equation}
y si $f$ es acotada y 2-veces diferenciable en $x,$ por la fórmula
de Taylor se tiene
\[f(y)-f(x)=\sum_{i=1}^d(y_i-x_i)\frac{\partial f}{\partial x_i}(x)
+\frac{1}{2}\sum_{i,j=1}^d (y_i-x_i)(y_j-x_j)\frac{\partial^2
f}{\partial x_i\partial x_j}(x)+o(|y-x|^2).\] Sustituyendo esto en
la integral (\ref{generador1}) para $|y-x|<\varepsilon,$ tomando
el límite cuando $h\downarrow 0$ y usando
(\ref{difusion1})-(\ref{difusion3}) se obtiene
\[\bigl(\A f\bigr)(t,x)=\frac{1}{2}\sum\limits_{i,j=1}^d
a_{ij}(t,x)\frac{\partial^2f}{\partial x_i\partial x_j}(x)+
\sum\limits_{j=1}^d W_i(t,x)\frac{\partial f}{\partial x_i}(x), \
\ (t,x)\in[0,\infty)\times\R^d.\] En la siguiente proposición
generalizamos la anterior expresión para el caso en que $f$
depende de $t$ y satisface ciertas condiciones de crecimiento con
respecto a sus derivadas:

{}{}\textbf{Notación.} Si $f:[0,\infty)\times\R^d\to\R$ es
diferenciable, para las derivadas de orden mayor usaremos la
notación
\[D_x^\alpha f=\frac{\partial^{|\alpha|}f}{\partial x_1^{\alpha_1}\cdots\partial
x_d^{\alpha_d}}\] donde $\alpha=(\alpha_1,\ldots,\alpha_d), \
\alpha_i=0,1,2,\ldots$ y $|\alpha|=\alpha_1+\cdots+\alpha_d.$

\begin{prop}\label{generador2}
Sea $\A$ el generador infinitesimal del semigrupo $(T_s)_{s\geq
0}$ y sea{} $f\in\C^{1,2}\bigl([0,\infty)\times\R^d\bigr)$ tal que
\begin{align*}
|D_x^\alpha f(t,x)|&\le C\bigl(1+|x|^\beta\bigr), \ \ \
0\le|\alpha|\le 2\\
\Bigl|\frac{\partial f}{\partial t}(t,x)\Bigr|&\le
C\bigl(1+|x|^\beta\bigr),
\end{align*}
para algunas constantes $C>0, \,\beta\geq 1$. Entonces $f\in
D(\A)$ y \[\bigl(\A f\bigr)(t,x)=\frac{\partial f}{\partial
t}(t,x)+\bigl(L_t f\bigr)(t,x),\] donde $L_t$ es el operador
diferencial parcial de segundo orden
\[\bigl(L_t f\bigr)(t,x)=\frac{1}{2}\sum\limits_{i,j=1}^d
a_{ij}(t,x)\frac{\partial^2f}{\partial x_i\partial x_j}(t,x)+
\sum\limits_{i=1}^d W_i(t,x)\frac{\partial f}{\partial x_i}(t,x),
\ \ (t,x)\in[0,\infty)\times\R^d.\]
\end{prop}
\begin{proof}
Aplicando fórmula de Itô al proceso $\bigl(X_s^{t,x}\bigr)_{s\geq
t}$ con la función $f,$ integrando entre $t$ y $t+h$ y teniendo en
cuenta $X_t^{t,x}=x$ obtenemos
\[f(t+h,X_{t+h}^{t,x})-f(t,x)=\int_t^{t+h}\Bigl(\frac{\partial f}{\partial s}
+L_s f\Bigr)\bigl(s,X_s^{t,x}\bigr)\,ds +\int_t^{t+h}(\nabla_x
f\cdot\sigma)\bigl(s,X_s^{t,x}\bigr)\,dW_s\] donde $\nabla_x
f=\bigl(\frac{\partial f}{\partial x_1},\cdots,\frac{\partial
f}{\partial x_d}\bigr).$ De las condiciones de crecimiento de
$\nabla_x f$ y $\sigma,$ y del lema \ref{apriori} se tiene que
\[E\!\int_t^{t+h}\!\bigl\|(\nabla_x
f\cdot\sigma)\bigl(s,X_s^{t,x}\bigr)\bigr\|^2\,ds\le C'\,
E\!\int_t^{t+h}\!\bigl(1+|X_s^{t,x}|^{\beta+1}\bigr)\,ds<+\infty\]
con $C'=2^{\beta+1}K_{t+h}C\sqrt{d},$ luego el valor esperado de
la integral estocástica es cero y
\begin{align}
E\bigl[f(t+h,X_{t+h}^{t,x})\bigr]-f(t,x)&=E\int_t^{t+h}(\tfrac{\partial
f}{\partial s}+L_s f)\bigl(s,X_s^{t,x}\bigr)\,ds\notag\\
&=h\int_0^1 E\bigl[(\tfrac{\partial f}{\partial s}+L_{t+hr}
f)\bigl(t+hr,X_{t+hr}^{t,x}\bigr)\bigr]\,dr\label{difusion4}
\end{align}
usando en la última igualdad el teorema de Fubini y el cambio de
variable $s=t+hr.$ Para $h>0$ y $0\le r\le 1,$ definimos
\[Y_h(r):=(\tfrac{\partial f}{\partial
r}+L_{t+hr} f)\bigl(t+hr,X_{t+hr}^{t,x}\bigr)\] Por el lema
\ref{apriori}
\[X_{t+hr}^{t,x}\xrightarrow[h\to 0^+]{L^2}X_t^{t,x}=x, \ \ \ 0\le r\le
1\] y por la continuidad de $\, a(t,x), \,b(t,x), \,\frac{\partial
f}{\partial x_i}, \, \frac{\partial^2 f}{\partial x_i\partial
x_j}$ y $\frac{\partial f}{\partial t}$
\[Y_h(r)\xrightarrow[h\to 0^+]{L^2}\frac{\partial f}{\partial t}(t,x)+L_t f(t,x).\]
y usando la desigualdad de Jensen se obtiene en particular
\[E[Y_h(r)]\xrightarrow[h\to 0^+]\,\frac{\partial f}{\partial t}(t,x)+L_t f(t,x), \ \ \forall
r\in[0,1].\] De las condiciones de crecimiento sobre $\sigma(s,x),
\,b(s,x), \, \frac{\partial f}{\partial x_i}$ y $\frac{\partial^2
f}{\partial x_i\partial x_j}$ se deduce la estimación
\begin{align}
|L_s f(s,x)|&\le \frac{1}{2}\sum_{i,j=1}^d
|a_{ij}(s,x)|\cdot\Bigl|\frac{\partial^2 f}{\partial x_i\partial
x_j}(s,x)\Bigr|
+\sum_{i=1}^d |W_i(s,x)|\cdot\Bigl|\frac{\partial f}{\partial x_i}(s,x)\Bigr|\notag\\
&\le\frac{1}{2}d^2K_T^2m C(1+|x|)^2(1+|x|^\beta)+CK_T\sqrt{d}(1+|x|)(1+|x|^\beta)\notag\\
&\le d^2K_T^2mC(1+|x|)^2(1+|x|)^\beta+2CK_T\sqrt{d}(1+|x|)(1+|x|)^\beta\notag\\
&\le C(md^2K_T^2+2K_T\sqrt{d})(1+|x|)^{\beta+2}\notag\\
&\le
2^{\beta+1}C(md^2K_T^2+2K_T\sqrt{d})(1+|x|^{\beta+2})\label{generador3}
\end{align}
para un $T>0$ suficientemente grande, luego
\[E\bigl[(L_{t+hr} f)\bigl(t+hr,X_{t+hr}^{t,x}\bigr)\bigr]\le
C_1\Bigl\{1+E\bigl(|X_{t+hr}^{t,x}|^{\beta+2}\bigr)\Bigr\} \le
C_1\bigl(2+|x|^{\beta+2}\bigr), \ \ \forall r\in[0,1]\] donde
$C_1=2^{\beta+1}C(md^2K_T^2+2K_T\sqrt{d}).$ Usando la condición de
crecimiento sobre $\frac{\partial f}{\partial s}$ se obtiene una
acotación similar para $E\bigl[\frac{\partial f}{\partial
s}(t+hr,X_{t+hr}^{t,x}\bigr)\bigr].$ Podemos entonces aplicar
convergencia dominada y usar (\ref{difusion4}) para obtener
\begin{align*}
\bigl(\A f\bigr)(t,x)&=\lim_{\ h\to 0^+}\frac{T_h
f(t,x)-f(t,x)}{h}
=\lim_{\ h\to 0^+}\frac{E\bigl[f(t+h,X_{t+h}^{t,x})\bigr]-f(t,x)}{h}\\
&=\lim_{\ h\to 0^+}\int_0^1 E\bigl[Y_h(r)\bigr]\,dr=\int_0^1
\bigl[\tfrac{\partial f}{\partial t}(t,x) +L_t
f(t,x)\bigr]\,ds\\
&=\frac{\partial f}{\partial t}(t,x)+L_t f(t,x)
\end{align*}
\end{proof}

\subsection[Diferenciabilidad con respecto a las condiciones
iniciales y la ecuación de\\Kolmogorov con condición final]{Diferenciabilidad con respecto a las condiciones
iniciales y la ecuación de Kolmogorov con condición final} Sea
$(W_t)_{t\geq 0}$ un movimiento Browniano $m-$dimensional definido
sobre un espacio de probabilidad $(\Omega,\F,\mathbf{P}),$ y sea
$\{\F_t\}_{t\geq 0}$ la filtración natural generada por
$(W_t)_{t\geq 0}$ y aumentada con los conjuntos $\mathbf{P}-$nulos
de $\Omega,$ es decir
\[\F_t=\sigma\bigl(\{W_s: 0\le s\le t\}\cup\mathcal{N}\bigr), \ \ \ t\geq 0 \] donde
$\mathcal{N}=\{E\subseteq\Omega : \exists G\in \F,\, G\supseteq E
\text{ y } \mathbf{P}(G)=0\}.$ La siguiente es una extensión del
teorema \ref{existencia} de existencia de soluciones de EDEs al
caso en el que los coeficientes son aleatorios
\begin{teorema}\label{extension}
Sea $(\varphi_t)_{t\geq 0}$ un proceso estocástico $\F_t-$prog.
medible con valores en $\R^d$ y sean
\[b:[0,+\infty)\times\R^d\times\Omega\rightarrow\R^d, \ \ \ \
\sigma:[0,+\infty)\times\R^d\times\Omega\rightarrow\R^{d\times
m}\time\Omega\] funciones medibles tales que para cada
$(t,x)\in[0,\infty)\times\R^d,$ las variables aleatorias
$b(t,x,\cdot)$ y $\sigma(t,x,\cdot)$ son $\F_t-$medibles.

{}{} Supongamos que existe una constante $K>0$ tal que para todo
$t\geq 0,$ $x,x'\in\R^d$ y $\omega\in\Omega$ se tiene
\begin{equation}\label{extension2}
\begin{split}
|b(t,x,\omega)-b(t,x',\omega)|+|\!|\sigma(t,x,\omega)-\sigma(t,x',\omega)|\!|&\le K|x-x'|\\
|b(t,x,\omega)|+|\!|\sigma(t,x,\omega)|\!|&\le K\bigl(1+|x|\bigr)
\end{split}
\end{equation}
y que el proceso $\varphi$ satisface \,$\sup\limits_{t\in [0,T]}E\bigl(|\varphi_t|^2\bigr)<+\infty,$ para todo $T>0.$

{}{} Entonces existe un proceso $(X_t)_{t\geq 0}$ $\F_t-$adaptado
y medible solución de la EDE con coeficientes aleatorios
\begin{equation}\label{extension3}
X_t=\varphi_t+\int_0^t b(s,X_s)\,ds+\int_0^t\sigma (s,X_s)\,dW_s,
\ \ \ \ t\geq 0 \ \ \text{c.s.}
\end{equation}
el cual satisface \,$\sup\limits_{t\in [0,T]}E\bigl(|X_t|^2\bigr)<+\infty,$ para todo $T>0.$

{}{} Si $X$ y $X'$ son soluciones de (\ref{extension3}) entonces
$X$ es versión de $X'.$ Si $\varphi$ es continuo entonces el
proceso $X$ es continuo y si $X$ y $X'$ son soluciones de
(\ref{extension3}) resulta que $X$ y $X'$ son indistinguibles.

{}{} Si además $\varphi$ es continuo y satisface
\,$E\bigl[\sup_{\, t\in [0,T]}|\varphi_t|^2\bigr]<+\infty,$
entonces
\[E\biggl[\sup\limits_{\, t\in [0,T]}|X_t|^2\biggr]<+\infty.\]
\end{teorema}
\begin{proof}
Primero veamos que si $X$ es solución de (\ref{extension3})
entonces
\[\sup_{t\in [0,T]}E\bigl[|X_t|^2\bigr]<+\infty.\]
Sean $\tau_n:=\{t\geq 0: |X_t|\geq n\},$ $n\geq 1.$ Usando la
condición de crecimiento lineal en (\ref{extension2}), la
Isometría de Itô y un argumento análogo al del lema \ref{apriori}
se obtiene

\begin{align*}
E\bigl(|X_{t\wedge\tau_n}|^2\bigr)&\le
3E\bigl(|\varphi_{t\wedge\tau_n}|^2\bigr)+
3tE\int_0^{t\wedge\tau_n}
|b(s,X_s)|^2\,ds+3E\!\int_0^{t\wedge\tau_n}
|\!|\sigma(s,X_s)|\!|^2\,ds\\&\le
3E\bigl(|\varphi_{t\wedge\tau_n}|^2\bigr)+
3(t+1)K\,E\int_0^t\bigl(1+|X_{s\wedge\tau_n}|\bigr)^2\,ds\\
&\le 3E\bigl(|\varphi_{t\wedge\tau_n}|^2\bigr)+ 6(T+1)KT+
6(T+1)K\int_0^t E\bigl(|X_{s\wedge\tau_n}|^2\bigr)\,ds, \ \ \ \
t\in [0,T].
\end{align*}
Por el lema de Gronwall existe una constante $C$ que depende de
$K$ y de $T$ tal que $E\bigl(|X_{t\wedge\tau_n}|^2\bigr)\le C
E\bigl(|\varphi_{t\wedge\tau_n}|^2\bigr)$ para todo $t\in[0,T], \
\forall n\geq 1,$ y por el lema de Fatou
\begin{align*}
E\bigl(|X_t|^2\bigr)&=E\biggl[\,\liminf_{n\to\infty}\,|X_{t\wedge\tau_n}|^2\biggr]\le
\liminf_{n\to\infty}\, E\bigl(|X_{t\wedge\tau_n}|^2\bigr)\\
&\le C\Bigl\{\liminf_{n\to\infty}\,
E\bigl(|\varphi_{t\wedge\tau_n}|^2\bigr)\Bigr\}\le
C\Bigl\{\limsup_{n\to\infty}\,
E\bigl(|\varphi_{t\wedge\tau_n}|^2\bigr)\Bigr\}\\
&\le
CE\biggl[\,\limsup_{n\to\infty}\,|\varphi_{t\wedge\tau_n}|^2\biggr]
=CE\bigl(|\varphi_t|^2\bigr), \ \ \ t\in [0,T].
\end{align*}
de donde \,$\sup\limits_{t\in [0,T]}E\bigl(|X_t|^2\bigr)\le
\sup\limits_{t\in [0,T]}C E\bigl(|\varphi_t|^2\bigr)<+\infty,$
para todo $T>0.$

{}{} De manera análoga, si $\varphi$ posee trayectorias continuas
c.s. y \,$E\bigl[\sup_{\, t\in [0,T]}|\varphi_t|^2\bigr]<+\infty,$ usando de nuevo un razonamiento
similar al del lema \ref{apriori} se obtiene
\[E\biggl[\sup\limits_{\, t\in [0,T]}|X_t|^2\biggr]<+\infty.\]
La prueba de existencia es exactamente la misma que en el teorema
\ref{existencia}\,: para cada $T>0$ se define la aplicación
$\Phi_T:\M_T^2(\R^d)\to\M_T^2(\R^d)$ que a cada $X\in\M_T^2(\R^d)$
le asigna el proceso
\begin{equation}
\Phi_T(X)_t=\varphi_t+\int_0^t b(s,X_s)\,ds+\int_0^t\sigma
(s,X_s)\,dW_s, \ \ \ \ t\in [0,T].
\end{equation}
Esta aplicación resulta ser una \textit{contracción estricta} de
$\M_T^2(\R^d)$ en sí mismo con la norma
\begin{equation}\label{norma3}
|\!|X|\!|_\lambda^2:=E\!\int_0^T e^{-\lambda s}|X_s|^2\,ds \,, \ \ \ \
\lambda>2K^2(1+T)
\end{equation}
que es equivalente a la norma usual de $\M_T^2(\R^d).$ Luego
$\M_T^2(\R^d)$ también es un espacio de Banach con la norma
(\ref{norma3}), y por lo tanto $\Phi_T$ posee un único punto fijo,
que notaremos por $X^T,$ y es solución de la EDE (\ref{ede1}) en
el intervalo $[0,T].$

{} Si $T\le T',$ el proceso $(X_t^{T'})_{t\in [0,T]}$ es una
\textit{versión} de $X^T,$ pues también pertenece a $\M_T^2(\R^d)$
y es punto fijo de $\Phi_T.$ Esto permite definir el proceso
continuo y $\F_t-$ adaptado $(X_t)_{t\geq 0}$ por
\[X_t:=X_t^T \ \text{si} \ t\in [0,T]\]
el cual claramente satisface
\begin{equation*}
X_t=\varphi_t+\int_0^t b(s,X_s)\,ds+\int_0^t\sigma (s,X_s)\,dW_s,
\ \ \ \ t\geq 0 \ \ \text{c.s.}
\end{equation*}
Si $X'$ es otra solución de (\ref{extension3}), para cada $T>0$ se
tiene
\[E\int_0^T |X'_s|^2\,ds=\int_0^T E\bigl(|X'_s|^2\bigr)\,ds
\le\int_0^T \sup_{t\in [0,T]}E\bigl(|X'_t|^2\bigr)\,ds=T
\sup_{t\in [0,T]}E\bigl(|X'_t|^2\bigr)<+\infty\] luego
$(X'_t)_{t\in [0,T]}$ también pertenece a $\M_T^2(\R^d)$ y es
punto fijo de $\Phi_T,$ por lo tanto es una \textit{versión} de
$(X_t)_{t\in [0,T]}$ para todo $T>0.$ Si $\varphi$ es continuo,
entonces $X$ y $X'$ tienen también trayectorias continuas y por lo
tanto son \textit{indistinguibles.}
\end{proof}

El siguiente resultado muestra la dependencia continua de las
soluciones de EDEs con coeficientes aleatorios con respecto a las
condiciones iniciales y a los coeficientes
\begin{teorema}\label{dep}
Sean $\{\sigma_k\}_{k\geq 0}$ y $\{W_k\}_{k\geq 0}$ sucesiones de
coeficientes aleatorios que satisfacen las hipótesis del teorema
anterior con la misma constante $K$ y sea $\{\varphi^k\}_{k\geq
0}$ una sucesión de procesos medibles y $\F_t-$adaptados. Suponga
que
\begin{align}
\sup_{k\geq 0}\sup_{\, 0\le s\le T}E\bigl(|\varphi_s^k|^2\bigr)&=C_T<+\infty\label{dep1}\\[.2cm]
\lim_{k\to\infty}\sup_{\, 0\le s\le
T}E\bigl(|\varphi_s^k-\varphi^0_s|^2\bigr)&=0, \ \ \forall T>0\label{dep2}\\
\lim_{k\to\infty}\mathbf{P}\Bigl(\sup_{\ |x|\le
N}|W_k(t,x)-W_0(t,x)|+|\!|\sigma_k(t,x)-&\sigma_0(t,x)|\!|>\varepsilon\Bigr)=0\label{convprob}
\end{align}
para todo $t\geq 0$ y para todo $\varepsilon,N>0.$ Sea
$(X_t^k)_{t\geq 0}$ la solución de la EDE con coeficientes
aleatorios
\begin{equation}
X_t^k=\varphi_t^k+\int_0^t W_k(s,X_s^k)\,ds+\int_0^t\sigma_k
(s,X_s^k)\,dW_s, \ \ \ \ t\geq 0.
\end{equation}
Entonces
\begin{equation}
\lim_{k\to\infty}\sup_{\, t\in [0,T]}E\bigl(|X_t^k-X^0_t|^2\bigr)=0, \ \ \ \forall T>0.
\end{equation}
\end{teorema}
\begin{proof}
Para cada $k\in\mathbb{N},$ sean
$\alpha_k(t)=\varphi_t^k-\varphi^0_t,$
\[\beta_k(t)=\displaystyle\int_0^t
\bigl[\sigma_k(s,X_s^0)-\sigma_0(s,X_s^0)\bigr]\,dW_s, \ \ \
\gamma_k(t)=\int_0^t\bigl[W_k(s,X_s^0)-W_0(s,X_s^0)\bigr]\,ds\] y
$\eta_k(t)=\alpha_k(t)+\beta_k(t)+ \gamma_k(t).$ Por la hipótesis
(\ref{dep2}) tenemos
\begin{equation}\label{dep4}
\lim_{k\to\infty}\sup_{\, t\in [0,T]}E\bigl(|\alpha_k(t)|^2\bigr)=0, \ \ \ \forall T>0
\end{equation}
luego para $s\geq 0$ y para $\varepsilon,N>0$ se tiene
\begin{align*}
&\mathbf{P}\bigl(|\!|\sigma_k(s,X_s^0)-\sigma_0(s,X_s^0)|\!|>\varepsilon\bigr)\\
&=\mathbf{P}\bigl(|\!|\sigma_k(s,X_s^0)-\sigma_0(s,X_s^0)|\!|>\varepsilon,\,
|X_s^0|\le N\bigr)
+\mathbf{P}\bigl(|\!|\sigma_k(s,X_s^0)-\sigma_0(s,X_s^0)|\!|>\varepsilon,\,
|X_s^0|>N\bigr)\\
&\le\mathbf{P}\Bigl(\sup_{\ |x|\le
N}|\!|\sigma_k(s,x)-\sigma_0(s,x)|\!|>\varepsilon\Bigr)+\mathbf{P}\bigl(|X_s^0|>N\bigr)
\end{align*}
de donde haciendo $k\to\infty$ se obtiene que
\[\limsup_{k\to\infty}\mathbf{P}\bigl(|\!|\sigma_k(s,X_s^0)-\sigma_0(s,X_s^0)|\!|
>\varepsilon\bigr)\le\mathbf{P}\bigl(|X_s^0|>N\bigr)\] y ahora haciendo
$N\to\infty,$ se sigue
\[\limsup_{k\to\infty}\mathbf{P}\bigl(|\!|\sigma_k(s,X_s^0)-\sigma_0(s,X_s^0)|\!|
>\varepsilon\bigr)=0\] y análogamente para $\{W_k\}_{k\geq 0}$ se obtiene
\[\limsup_{k\to\infty}\mathbf{P}\bigl(|W_k(s,X_s^0)-W_0(s,X_s^0)|>\varepsilon\bigr)=0\]
es decir
\begin{equation}\label{dep5}
\sigma_k(s,X_s^0)\xrightarrow[k\to\infty]{\mathbf{P}}\sigma_0(s,X_s^0),
\ \ \ \ \
W_k(s,X_s^0)\xrightarrow[k\to\infty]{\mathbf{P}}W_0(s,X_s^0)
\end{equation}
para todo $s\geq 0.$ Usando (\ref{dep4}), (\ref{dep5}), las
estimaciones
\begin{align*}
|\!|\sigma_k(s,X_s^0)|\!|&\le K^2\bigl(1+|X_s^0|\bigr)^2, \ \ \ \ \forall k\geq 1\\
E\int_0^T \bigl(1+|X_s^0|\bigr)^2\,ds&\le T\Bigl\{1+\sup_{t\in [0,T]} E\bigl(|\varphi_t^0|^2\bigr)\Bigr\}<+\infty
\end{align*}
isometría de Itô y el hecho de que convergencia en probabilidad y
dominada implica convergencia en $L^2,$ obtenemos
\begin{equation}\label{dep6}
\sup_{\, t\in [0,T]}E\bigl(|\beta_k(t)|^2\bigr)\le E\int_0^T
|\!|\sigma_k(s,X_s^0)-\sigma_0(s,X_s^0)|\!|^2\,ds\xrightarrow[k\to\infty]\,0\,,
\ \ \ \forall T>0.
\end{equation}
Igualmente por (\ref{dep5}) y por un criterio similar (usando
desigualdad de Hölder en vez de Isometría de Itô) se sigue
\begin{equation}\label{dep7}
\sup_{\, t\in [0,T]}E\bigl(|\gamma_k(t)|^2\bigr)=0\xrightarrow[k\to\infty]\,0\,, \ \
\ \forall T>0.
\end{equation}
Combinando (\ref{dep4}), (\ref{dep6}) y (\ref{dep7}) se obtiene
\begin{equation}\label{dep8}
\sup_{\, t\in [0,T]}E\bigl(|\eta_k(t)|^2\bigr)=0\xrightarrow[k\to\infty]\,0\,, \ \ \
\forall T>0,
\end{equation}
y dado que
\[X_t^k-X_t^0=\eta_k(t)+\int_0^t \bigl[W_k(s,X_s^k)-W_k(s,X_s^0)\bigr]\,ds
+\int_0^t\bigl[\sigma_k(s,X_s^k)-\sigma_k(s,X_s^0)\bigr]\,dW_s,\]
usando la condición de Lipschitz, la desigualdad de Hölder e
isometría de Itô se deduce que
\[E\bigl(|X_t^k-X_t^0|^2\bigr)\le 3E\bigl(|\eta_k(t)|^2\bigr)
+3(T+1)K^2\int_0^t E\bigl(|X_s^k-X_s^0|^2\bigr)\,ds,\] y por el
lema de Gronwall junto con (\ref{dep8}) tenemos
\[\sup_{\, t\in [0,T]}E\bigl(|X_t^k-X_t^0|^2\bigr)\le C\sup_{\, t\in [0,T]}E\bigl(|\eta_k(t)|^2\bigr)=0\xrightarrow[k\to\infty]\,0\,, \ \ \
\forall T>0.\]
\end{proof}
\begin{obse}
Si los procesos $\varphi^k$ son continuos y las condiciones
(\ref{dep1}) y (\ref{dep2}) se reemplazan por
\begin{align}
\sup_{k\geq 0}E\biggl[\sup_{\, t\in [0,T]}|\varphi_t^k|^2\biggr]&=C_T<+\infty,\label{depe1}\\[0.2cm]
\lim_{k\to\infty}E\biggl[\sup_{\, t\in [0,T]}|\varphi_t^k-\varphi^0_t|^2\biggr]&=0, \ \ \forall
T>0,\label{depe2}
\end{align}
entonces la conclusión será
\begin{equation}
\lim_{k\to\infty}E\biggl[\sup_{\, t\in [0,T]}|X_t^k-X^0_t|^2\biggr]=0, \ \ \ \forall T>0.
\end{equation}
La demostración se hace de manera similar, sólo hay que cambiar \,
$\sup_{t\in [0,T]}E\bigl(|\cdot|^2\bigr)$ por
$E\bigl[\sup_{\,t\in [0,T]}|\cdot|^2\bigr].$ Asimismo la
hipótesis de convergencia en probabilidad (\ref{convprob}) se
puede reforzar por convergencia puntual
\begin{equation}\label{convpuntual}
W_k(t,x,\omega)\xrightarrow[k\to\infty]\,b(t,x,\omega), \ \ \
\sigma_k(t,x,\omega)\xrightarrow[k\to\infty]\,\sigma(t,x,\omega)
\end{equation}
para todo $t\geq 0,$ $x\in\R^d$ y $\omega\in\Omega,$ y aplicar
directamente convergencia dominada para obtener (\ref{dep6}) y
(\ref{dep7}) en la demostración.
\end{obse}
{} Ahora usaremos el anterior resultado de dependencia continua
para obtener la \textit{diferenciabilidad} con respecto a $x$ de
la solución $(\X)_{s\geq t}$ de la EDE
\begin{equation}\label{edetx}
\X=x+\int_t^s b(r,X_r^{t,x})\,dr+\int_t^s\sigma
(r,X_r^{t,x})\,dW_r, \ \ \ \ s\geq t
\end{equation}
Este último resultado será pieza clave en la demostración de la
diferenciabilidad \textit{a priori} de la solución de la
ecuación de Kolmogorov con condición final (en inglés, \emph{Backward Kolmogorov equation}).

{}{}\textbf{Notación.} Si $b:[0,\infty)\times\R^d\to\R^d$ es
diferenciable, notaremos por $D_x b$ la matriz de tamaño $d\times
d$ con componentes $(D_x b)_{ij}=\frac{\partial W_i}{\partial
x_j}$ y la llamaremos la \textit{diferencial} de $b$ con respecto
a $x$ (es la generalización del gradiente al caso en que $b$ es un
campo vectorial). La $i-$ésima columna de $D_x b$ la notaremos por
\[\dfrac{\partial b}{\partial x_i}=\begin{pmatrix}
\frac{\partial W_1}{\partial x_i}\\
\vdots\\
\frac{\partial W_d}{\partial x_i}
\end{pmatrix}\]

\begin{defi}
Sea $\varnothing\neq G\subseteq\R^d$ un conjunto abierto y sean
$f,g:G\times\Omega\to\R$ funciones medibles. Diremos que $g$ tiene
\textit{derivada $f$ en $L^2(\Omega)$ con respecto a $x_i$} si
\[\frac{1}{h}\bigl[g(x+he_i)-g(x)\bigr]\xrightarrow[h\to 0]{L^2}f(x), \ \ \ \forall x\in G\]
donde $e_i=(\delta_{ij})_{1\le j\le d}, \ i=1,\ldots,d.$
Denotaremos $\frac{\partial g}{\partial x_i}(x)=f(x).$ De una
manera similar se definen las derivadas en $L^2(\Omega)$ de orden
mayor.
\end{defi}

Para el siguiente teorema asumiremos $m=1,$ luego $(W_t)_{t\geq
0}$ será un M.B. unidimensional (el caso $m>1$ se demuestra de
manera análoga).
\begin{teorema}\label{dif}
Sean $b,\sigma:[0,\infty)\times\R^d\to\R^d$ funciones medibles que
satisfacen las condiciones de Lipschitz global y de crecimiento
lineal. Suponga además que las derivadas parciales $\frac{\partial
b}{\partial x_i}, \,\frac{\partial \sigma}{\partial x_i}$ existen,
son continuas y acotadas, para $1\le i\le d.$

{}{} Entonces la derivada parcial $\displaystyle\frac{\partial
\X}{\partial x_i}$ existe en $L^2(\Omega),$ para todo $1\le i\le
d,$ y tiene como versión la solución $Y_t^i$ de la EDE con
coeficientes aleatorios
\begin{equation}\label{dif1}
Y_s^i=e_i+\int_t^s D_x b(r,X_r^{t,x})Y_r^i\,dr+\int_t^s D_x\sigma
(r,X_r^{t,x})Y_r^i\,dW_r, \ \ \ \ s\geq t
\end{equation}
donde $e_i=(\delta_{ij})_{1\le j\le d}.$
\end{teorema}
\begin{proof}
Usando la desigualdad de Cauchy-Schwarz y el hecho de que las
diferenciales $D_x b$ y $D_x\sigma$ son acotadas, se sigue que los
coeficientes aleatorios de (\ref{dif1}) satisfacen las hipótesis
del teorema \ref{extension}, por lo tanto (\ref{dif1}) posee una
solución $(Y_s^i)_{s\geq t}$ única en trayectoria. Sean $1\le j\le
d$ y $h,\lambda\in\R$ con $0\le\lambda\le 1$ y $h\ne 0.$ Por la
regla de la cadena se tiene que
\begin{align*}
W_j(s,X_s^{t,x+he_i})-W_j(s,X_s^{t,x})
&=\int_0^1\frac{d}{d\lambda}W_j(s,X_s^{t,x}+\lambda(X_s^{t,x+he_i}-X_s^{t,x}))\,d\lambda\\
&=\int_0^1\nabla_x W_j(s,X_s^{t,x}+\lambda(X_s^{t,x+he_i}
-X_s^{t,x}))\cdot(X_s^{t,x+he_i} -X_s^{t,x})\,d\lambda,
\end{align*}
donde $\nabla_x W_j=\bigl(\frac{\partial W_j}{\partial
x_1},\cdots,\frac{\partial W_j}{\partial x_d}\bigr).$

{}{} Para $\sigma$ se obtiene una relación similar, y que en
notación matricial equivalen a
\begin{align}
b(s,X_s^{t,x+he_i})-b(s,X_s^{t,x})&=\int_0^1 D_x
b(s,X_s^{t,x}+\lambda(X_s^{t,x+he_i}
-X_s^{t,x}))\cdot(X_s^{t,x+he_i} -X_s^{t,x})\,d\lambda,\label{dif2}\\
\sigma(s,X_s^{t,x+he_i})-\sigma(s,X_s^{t,x})&=\int_0^1 D_x
\sigma(s,X_s^{t,x}+\lambda(X_s^{t,x+he_i}
-X_s^{t,x}))\cdot(X_s^{t,x+he_i}
-X_s^{t,x})\,d\lambda.\label{dif3}
\end{align}
Definiendo, para $h\ne 0$
\begin{align*}
Z_s^h:=&\,\frac{1}{h}\bigl[X_s^{t,x+he_i}
-X_s^{t,x}\bigr],\\
b^h(s,z)&:=\int_0^1 D_x b(s,X_s^{t,x}+\lambda(X_s^{t,x+he_i}
-X_s^{t,x}))z\,d\lambda,\\
\sigma^h(s,z)&:=\int_0^1 D_x
\sigma(s,X_s^{t,x}+\lambda(X_s^{t,x+he_i} -X_s^{t,x}))z\,d\lambda,
\end{align*}
y usando (\ref{dif2}) y (\ref{dif3}) vemos que $(Z_s^h)_{s\geq t}$
satisface
\begin{align*}
Z_s^h&=e_i+\frac{1}{h}\int_t^s
\bigl[b(r,X_r^{t,x+he_i})-b(r,X_r^{t,x})\bigr]\,dr
+\frac{1}{h}\int_t^s[\sigma(r,X_r^{t,x+he_i})-\sigma(r,X_r^{t,x})\bigr]\,dW_r\\
&=e_i+\int_t^s \int_0^1 D_x b(r,X_r^{t,x}+\lambda(X_r^{t,x+he_i}
-X_r^{t,x}))\cdot\frac{1}{h}(X_r^{t,x+he_i}
-X_r^{t,x})\,d\lambda\,dr\\
&\hspace{2cm}+\int_t^s \int_0^1 D_x
\sigma(r,X_r^{t,x}+\lambda(X_r^{t,x+he_i}
-X_r^{t,x}))\cdot\frac{1}{h}(X_r^{t,x+he_i}-X_r^{t,x})\,d\lambda
\,dW_r
\end{align*}
es decir
\begin{equation}\label{dif4}
Z_s^h=e_i+\int_t^s b^h(r,Z_r^h)\,dr+\int_t^s
\sigma^h(s,Z_r^h)\,dW_r, \ \ \ \ s\geq t.
\end{equation}
Si además definimos $Z_s^0:=Y_s^i$, entonces $(Z_s^0)_{s\geq t}$
satisface (\ref{dif4}) con
\[b^0(s,z):=D_x b(s,\X)z, \ \ \ \sigma^0(s,z):=D_x \sigma(s,\X)z.\]
Por la proposición \ref{propmarkov1} se tiene que
\[\lim_{h\to 0}E\biggl[\sup_{\, t\le s\le T}|X_s^{t,x+he_i}-X_s^{t,x}|^2\biggr]=0,\]
luego para todo $0\le \lambda\le 1$
\[X_s^{t,x}+\lambda(X_s^{t,x+he_i}
-X_s^{t,x})\xrightarrow[h\to 0]{L^2}X_s^{t,x}, \ \ \ \forall s\geq
t\] y por continuidad de $D_x b$
\[D_x b(s,X_s^{t,x}+\lambda(X_s^{t,x+he_i}
-X_s^{t,x}))\xrightarrow[h\to 0]{L^2}D_x b(s,X_s^{t,x}), \ \ \
\forall s\geq t.\] Usando desigualdad de Hölder, el teorema de
Fubini, el hecho de que $D_x b$ es acotada y convergencia dominada
se obtiene
\begin{align*}
E\biggl[&\Bigl\|\int_0^1 D_x b(s,X_s^{t,x}+\lambda(X_s^{t,x+he_i}
-X_s^{t,x}))\,d\lambda-D_x b(s,X_s^{t,x})\Bigr\|^2\biggr]\\
&=E\biggl[\Bigl\|\int_0^1 \bigl[D_x
b(s,X_s^{t,x}+\lambda(X_s^{t,x+he_i}
-X_s^{t,x}))-D_x b(s,X_s^{t,x})\bigr]\,d\lambda\Bigr\|^2\biggr]\\
&\le E\int_0^1 |\!|D_x b(s,X_s^{t,x}+\lambda(X_s^{t,x+he_i}
-X_s^{t,x}))-D_x b(s,X_s^{t,x})|\!|^2\,d\lambda\\
&=\int_0^1 E\bigl(|\!|D_x b(s,X_s^{t,x}+\lambda(X_s^{t,x+he_i}
-X_s^{t,x}))-D_x
b(s,X_s^{t,x})|\!|^2\bigr)\,d\lambda\xrightarrow[h\to 0]\,0,
\end{align*}
y debido a que convergencia en $L^2$ implica convergencia en
probabilidad \[D_x b(s,X_s^{t,x}+\lambda(X_s^{t,x+he_i}
-X_s^{t,x}))\xrightarrow[h\to 0]{\mathbf{P}}D_x b(s,X_s^{t,x})\]
luego
\begin{align*}
\mathbf{P}\biggl(&\sup_{\,|z|\le
N}|b^h(s,z)-b^0(s,z)|>\varepsilon\biggr)\\
&=\mathbf{P}\Bigl(\sup_{\,|z|\le N}\Big|\int_0^1 D_x
b(s,X_s^{t,x}+\lambda(X_s^{t,x+he_i}-X_s^{t,x}))\,z\,d\lambda-D_x
b(s,X_s^{t,x})\,z\Bigr|>\varepsilon\Bigr)\\
&=\mathbf{P}\biggl(\Big\|\int_0^1 D_x
b(s,X_s^{t,x}+\lambda(X_s^{t,x+he_i}-X_s^{t,x}))\,d\lambda-D_x
b(s,X_s^{t,x})\Bigr\|\,N>\varepsilon\biggr)\\
&=\mathbf{P}\biggl(\Big\|\int_0^1 D_x
b(s,X_s^{t,x}+\lambda(X_s^{t,x+he_i}-X_s^{t,x}))\,d\lambda-D_x
b(s,X_s^{t,x})\Bigr\|>\varepsilon/N\biggr)\xrightarrow[h\to 0]\,0,
\end{align*}
y de forma completamente análoga se deduce
\[\mathbf{P}\biggl(\sup_{\ |z|\le
N}|\sigma^h(s,z)-\sigma^0(s,z)|>\varepsilon\biggr)\xrightarrow[h\to
0]\,0\] para todo $s\geq t$ y para todo $\varepsilon,N>0.$
Aplicando el teorema \ref{dep} obtenemos
\[E\biggl[\sup_{\, t\le s\le T}|Z_s^h-Z_s^0|^2\biggr]\xrightarrow[h\to 0]\,0,\]
y por lo tanto
\[\frac{1}{h}\bigl[X_s^{t,x+he_i}
-X_s^{t,x}\bigr]=Z_s^h\xrightarrow[h\to 0]{L^2}Z_s^0=Y_s^i, \ \ \
\ t\le s\le T\] lo cual prueba el teorema.
\end{proof}
Note que las derivadas $\displaystyle\frac{\partial \X}{\partial
x_i}$ satisfacen la EDE con coeficientes aleatorios que se
obtienen de (\ref{edetx}) derivando los coeficientes $b,\sigma$
con respecto a $x_i.$ De forma similar al anterior teorema se demuestra el
siguiente resultado
\begin{teorema}\label{dif2orden}
Sean $b,\sigma:[0,\infty)\times\R^d\to\R^d$ funciones medibles que
satisfacen las condiciones de Lipschitz global y de crecimiento
lineal. Suponga además que las derivadas $D_x^\alpha W_i,
\,D_x^\alpha \sigma_i, \,1\le i\le d$ existen, son continuas y
acotadas para todo $1\le|\alpha|\le 2.$

{}{} Entonces para todo $1\le i,j\le d,$ la derivada parcial de
segundo orden  $\displaystyle\frac{\partial^2 \X}{\partial
x_i\partial x_j}$ existe en $L^2(\Omega)$ y tiene como versión la
solución de la EDE cuyos coeficientes se obtienen derivando los
coeficientes de (\ref{edetx}) con respecto a $x_i$ y $x_j.$
\end{teorema}

\begin{teorema}\label{lemakolmo}
Suponga que se cumplen las hipótesis del teorema \ref{dif2orden}.
Sea $f:\R^d\to\R$ tal que $D_x^\alpha f$ existe, es continua y
satisface
\[|D_x^\alpha f(x)|\le C\bigl(1+|x|^\beta\bigr), \ \ \
0\le|\alpha|\le 2\] para algunas constantes $C>0, \,\beta\geq 1.$
Entonces, para cada $s\geq t$ la aplicación
\begin{equation}\label{defphi}
\varphi(x):=E[f(\X)]
\end{equation}
está en $\C^2(\R^d)$ y sus derivadas parciales se obtienen
derivando el lado derecho de (\ref{defphi}) bajo el valor
esperado.
\end{teorema}
\begin{proof}
Probaremos que
\begin{equation}\label{difphi}
\frac{\partial \varphi}{\partial x_i}=E\biggl[\nabla_x f(\X)
\frac{\partial \X}{\partial x_i}\biggr].
\end{equation}
Sean $1\le i\le d$ fijo y $h\ne 0.$ Usando el mismo razonamiento
que en el teorema \ref{dif}, por la regla de la cadena se tiene
que
\begin{align*}
\varphi(x+he_i)-\varphi(x)&=E[f(X_s^{t,x+he_i})]-E[f(\X)]\\
&=E\int_0^1\frac{d}{d\lambda}f(X_s^{t,x}+\lambda(X_s^{t,x+he_i}-X_s^{t,x}))\,d\lambda\\
&=\int_0^1 E\bigl[\nabla_x f(X_s^{t,x}+\lambda(X_s^{t,x+he_i}
-X_s^{t,x}))\cdot(X_s^{t,x+he_i} -X_s^{t,x})\bigr]\,d\lambda.
\end{align*}
Usando la misma notación $Z_s^h=\dfrac{1}{h}\bigl[X_s^{t,x+he_i}
-X_s^{t,x}\bigr]$ se obtiene
\[\frac{\varphi(x+he_i)-\varphi(x)}{h}
=\int_0^1 E\bigl[\nabla_x
f\bigl(X_s^{t,x}+\lambda(X_s^{t,x+he_i}-X_s^{t,x})\bigr)\cdot
Z_s^h\bigr]\,d\lambda.\] Por la proposición \ref{propmarkov1} y
por el teorema \ref{dif} se tiene que
\[X_s^{t,x}+\lambda(X_s^{t,x+he_i}-X_s^{t,x})\xrightarrow[h\to
0]{L^2}X_s^{t,x}, \ \ \ Z_s^h\xrightarrow[h\to
0]{L^2}\frac{\partial \X}{\partial x_i},\] y por continuidad de
$\nabla_x f$ y del producto interno, tenemos
\[\nabla_x f\bigl(X_s^{t,x}+\lambda(X_s^{t,x+he_i}
-X_s^{t,x})\bigr)\cdot Z_s^h\xrightarrow[h\to 0]{L^2}\nabla_x
f(X_s^{t,x})\frac{\partial \X}{\partial x_i}.\] Usando la
desigualdad de Jensen se obtiene en particular
\[E\bigl[\nabla_x f\bigr(X_s^{t,x}+\lambda(X_s^{t,x+he_i}
-X_s^{t,x})\bigr)\cdot Z_s^h\bigr]\xrightarrow[h\to
0]\,E\biggl[\nabla_x f(s,X_s^{t,x})\frac{\partial \X}{\partial
x_i}\biggr]\] y de nuevo por la proposición \ref{propmarkov1}, y
usando la desigualdad $ab\le a^2/2+b^2/2,$ para $0<|h|\le 1$ se
tiene que
\begin{align*}
E\bigl[\nabla_x f&\bigr(X_s^{t,x}+\lambda(X_s^{t,x+he_i}
-X_s^{t,x})\bigr)\cdot Z_s^h\bigr]\\
&\le\frac{C^2}{2}E\Bigl[\bigl(1+|X_s^{t,x}+\lambda(X_s^{t,x+he_i}
-X_s^{t,x})|^\beta\bigr)^2\Bigr]+\frac{1}{2}E\bigl[|Z_s^h|^2\bigr]\\
&\le C^2E\bigl[1+|X_s^{t,x}+\lambda|X_s^{t,x+he_i}
-X_s^{t,x}|^{2\beta}\bigr]+K/2\\
&\le
C^2+2^{2\beta-1}C^2E\bigl(|X_s^{t,x}|^{2\beta}\bigr)+\lambda^{2\beta}2^{2\beta-1}C^2
E\bigl(|X_s^{t,x+he_i}-X_s^{t,x}|^{2\beta}\bigr)+K/2\\
&\le
C^2+2^{2\beta-1}C^2E\bigl(|X_s^{t,x}|^{2\beta}\bigr)+\lambda^{2\beta}2^{2\beta-1}C^2K+K/2.
\end{align*}
Este último termino no depende de $h$ y es claramente integrable
entre $0$ y $1,$ luego podemos aplicar convergencia dominada y
obtener
\[\int_0^1 E\bigl[\nabla_x
f\bigl(X_s^{t,x}+\lambda(X_s^{t,x+he_i}-X_s^{t,x})\bigr)\cdot
Z_s^h\bigr]\,d\lambda\xrightarrow[h\to 0]\,\int_0^1
E\biggl[\nabla_x f(s,X_s^{t,x})\frac{\partial \X}{\partial
x_i}\biggr]\,d\lambda,\] es decir
\[\lim_{h\to 0}\frac{\varphi(x+he_i)-\varphi(x)}{h}=
E\biggl[\nabla_x f(s,X_s^{t,x})\frac{\partial \X}{\partial
x_i}\biggr],\] lo cual prueba (\ref{difphi}). Aplicando el mismo
razonamiento a $\dfrac{\partial\varphi}{\partial x_i}$ en vez de
$\varphi$ se demuestra
\[\frac{\partial^2 \varphi}{\partial
x_i\partial x_j}=E\biggl[\frac{\partial(\nabla_x f)}{\partial x_j}
\bigl(\X\bigr)\cdot\frac{\partial \X}{\partial x_i} +\nabla_x
f(\X)\cdot\frac{\partial}{\partial x_j}\biggl(\frac{\partial
\X}{\partial x_i}\biggr)\biggr].\]
\end{proof}

Ahora podemos utilizar los resultados anteriores para obtener una
demostración{} probabilística de la existencia de soluciones de
algunas EDPs parabólicas que involucran al operador diferencial de
segundo orden $\A=\frac{\partial}{\partial t}+L_t$ donde
\begin{equation}\label{operador2}
(L_tu)(t,x) = \frac{1}{2}\sum_{i,j=1}^d
a_{ij}(t,x)\frac{\partial^2u}{\partial x_i\partial x_j}(t,x)+
\sum_{i=1}^d W_i(t,x)\frac{\partial u}{\partial x_i}(t,x)
\end{equation}
que como ya vimos es el generador infinitesimal de las
probabilidades de transición del proceso de difusión con
coeficiente de \textit{drift} $b(t,x)$ y matriz de
\textit{difusión} $a(t,x)=\sigma(t,x)\sigma(t,x)^*.$
\begin{teorema}[Ecuación de Kolmogorov con condición final]
Suponga que se cumplen las hipótesis del teorema \ref{lemakolmo}.
Entonces la función
\begin{equation}
u(t,x):=E\bigl[f(X_T^{t,x})\bigr], \ \ \ \ (t,x)\in[0,T]\times\R^d
\end{equation}
pertenece a $\C^{1,2}\bigl([0,T)\times\R^d\bigr)$ y satisface el problema de valor final
\begin{equation}\label{eckolmogorov}
\begin{split}
\frac{\partial u}{\partial t}(t,x)+L_{t}u(t,x)&=0, \ \ \ \ \ \
\text{en}
\ [0,T)\times\R^d\\
u(T,x)&=f(x), \ \ \ x\in\R^d
\end{split}
\end{equation}
\end{teorema}
\begin{proof}
Sea $g\in\C^2(\R^d)$ que toma valores en los reales y satisface
\[|D_x^\alpha g(x)|\le C_1\bigl(1+|x|^{\beta_1}\bigr), \ \ \
0\le|\alpha|\le 2\] para algunas constantes $C_1,\beta_1>0.$ Para
$0<h\le t$ y $0\le s\le 1,$ defina
\[Y_h(s):=(L_{t-hs}g)\bigl(t-hs,X_{t-hs}^{t-h,x}\bigr).\]
Por el lema \ref{apriori}
\[E\Bigl[\bigl|X_{t-hs}^{t-h,x}-X_t^{t-h,x}\bigr|^2\Bigr]\xrightarrow[h\to 0^+]\,0,\]
y por la proposición \ref{propmarkov1}
\[E\Bigl[\bigl|X_t^{t-h,x}-X_t^{t,x}\bigr|^2\Bigr]\xrightarrow[h\to 0^+]\,0.\]
Dado que $\bigl|X_{t-hs}^{t-h,x}-X_t^{t,x}\bigr|^2\le
2\bigl|X_{t-hs}^{t-h,x}-X_t^{t-h,x}\bigr|^2 +
2\bigl|X_t^{t-h,x}-X_t^{t,x}\bigr|^2,$ entonces
\[X_{t-hs}^{t-h,x}\xrightarrow[h\to 0]{L^2}X_t^{t,x}=x, \ \ \ 0\le s\le
1,\]y usando la continuidad de $a(t,x), \,b(t,x),
\,\dfrac{\partial g}{\partial x_i}$ y $\dfrac{\partial^2
g}{\partial x_i\partial x_j}$ obtenemos
\[Y_h(s)\xrightarrow[h\to 0]{L^2}L_tg(t,x).\]
Por la desigualdad de Jensen se tiene en particular que
\[E[Y_h(s)]\xrightarrow[h\to 0]\,L_tg(t,x), \ \ \forall
s\in[0,1].\] Al igual que en (\ref{generador3}) en la demostración
de la proposición \ref{generador2}, usando las condiciones de
crecimiento sobre $\sigma(s,x), \,b(s,x), \, \frac{\partial
g}{\partial x_i}$ y $\frac{\partial^2 g}{\partial x_i\partial
x_j},$ se obtiene la estimación
\[|L_s g(s,x)|\le
2^{\beta_1+1}C_1(md^2K^2+2K\sqrt{d})(1+|x|^{\beta_1+2})\] luego
\[E[Y_h(s)]\le
C_2\Bigl\{1+E\bigl(|X_{t-hs}^{t-h,x}|^{\beta_1+2}\bigr)\Bigr\} \le
C_2\bigl(2+|x|^{\beta_1+2}\bigr), \ \ \forall s\in[0,1],\] donde
$C_2=2^{\beta+1}C_1(md^2K^2+2K\sqrt{d}).$ Podemos entonces aplicar
convergencia dominada a la familia de funciones
$E\bigl[Y_h(\cdot)\bigr], \ 0<h\le t,$ y obtener
\[\int_0^1 E\bigl[Y_h(s)\bigr]\,ds\xrightarrow[h\to 0]\,\int_0^1 L_tg(t,x)\,ds=L_tg(t,x)\]
Haciendo el cambio de variable $r=t-hs$
\[\int_0^1 E\bigl[Y_h(s)\bigr]\,ds
=\int_0^1
E\bigl[(L_{t-hs}g)\bigl(t-hs,X_{t-hs}^{t-h,x}\bigr)\bigr]\,ds
=\frac{1}{h}E\int_{t-h}^t(L_rg)\bigl(r,X_r^{t-h,x}\bigr)\,dr.\]
Aplicando fórmula de Itô al proceso
$\bigl(X_s^{t-h,x}\bigr)_{s\geq t-h}$ con la función $g$ e
integrando entre $t-h$ y $t$ obtenemos
\[g(X_t^{t-h,x})-g(x)=\int_{t-h}^t (L_sg)\bigl(s,X_s^{t-h,x}\bigr)\,ds
+\int_{t-h}^t\nabla_x
g\bigl(X_s^{t-h,x}\bigr)\cdot\sigma\bigl(s,X_s^{t-h,x}\bigr)\,dW_s.\]
Usando las condiciones de crecimiento sobre $\nabla_x g$ \,y
$\sigma,$ y el lema \ref{apriori} se puede ver que
\[E\int_{t-h}^t\bigl\|\nabla_x g\bigl(X_s^{t-h,x}\bigr)\cdot
\sigma\bigl(s,X_s^{t-h,x}\bigr)\bigr\|^2\,ds<+\infty,\] luego el
valor esperado de la integral estocástica es cero y
\begin{equation}\label{kolmo1}
E\bigl[g(X_t^{t-h,x})\bigr]-g(x)=E\int_{t-h}^t
(L_sg)\bigl(s,X_s^{t-h,x}\bigr)\,ds,
\end{equation}
entonces
\begin{equation}\label{kolmo2}
\frac{E\bigl[g(X_t^{t-h,x})\bigr]-g(x)}{h}=
\frac{1}{h}E\int_{t-h}^t(L_rg)\bigl(r,X_r^{t-h,x}\bigr)\,dr\xrightarrow[h\to
0^+]\,L_tg(t,x),
\end{equation}
y usando la estimación para $L_s g(s,x)$ y (\ref{kolmo1}), se
deduce la desigualdad
\begin{equation}\label{kolmo3}
\bigl|E\bigl[g(X_t^{t-h,x})\bigr]-g(x)\bigr|\le \int_{t-h}^t
C_2\bigl(2+|x|^{\beta_1+2}\bigr)\,ds=hC_2\bigl(2+|x|^{\beta_1+2}\bigr).
\end{equation}
Por el lema \ref{lemamarkov} y la igualdad
$X_T^{t-h,x}=X_T^{t,X_t^{t-h,x}}$ (consecuencia de la unicidad en
trayectoria) se sigue la relación
\[u(t-h,x)=E\bigl[u\bigl(t,X_t^{t-h,x}\bigr)\bigr]\]
Tomando $g(x)=u(t,x)$ en (\ref{kolmo2}) y (\ref{kolmo3}) se
obtiene
\begin{equation}\label{kolmo4}
\lim_{h\to 0^-}\frac{u(t+h,x)-u(t,x)}{h}=-\lim_{h\to
0^+}\frac{u(t-h,x)-u(t,x)}{h} =-L_t u(t,x)
\end{equation}
y $|u(t-h,x)-u(t,x)|\le C_3(x)h,$ donde $C_3(x)$ es una constante
que no depende de $h,$ lo que significa que para $x$ fijo,
$u(\cdot,x)$ es \textit{absolutamente continua} (ver \cite{ash1},
sección 2.3) y por lo tanto $\frac{\partial u}{\partial s}(s,x)$
existe para casi todo $s$ y satisface
\[u(t,x)=u(0,x)+\int_0^t\frac{\partial u}{\partial s}(s,x)\,ds\]
lo que implica, junto con (\ref{kolmo4}), que
\[u(t,x)=u(0,x)-\int_0^t L_s u(s,x)\,ds.\]
Como $s\mapsto L_s u(s,x)$ es continua, por el teorema fundamental
del cálculo, $\frac{\partial u}{\partial t}(t,x)$ existe en todos
lados y satisface
\[\frac{\partial u}{\partial t}(t,x)+L_{t}u(t,x)=0, \ \ \ \forall
(t,x)\in[0,T)\times\R^d.\]
\end{proof}

\subsection[El problema de Cauchy. Representación de
Feynman-Kac]{El problema de Cauchy. Representación de\\
Feynman-K\u{a}c}
\begin{teorema}[K\u{a}c, Rosenblatt (1951)]\label{feynmankac}
Sea $T>0$ fijo, y sean $f:\R^d\longrightarrow\R, {}
h:[0,T]\times\R^d\longrightarrow\R$ y $c:[0,T]\times\R^d\to
[0,+\infty)$ funciones continuas que satisfacen
\begin{align}
&\mathbf{(i)} \ \ |f(x)|\le L\bigl(1+|x|^{2\lambda}\bigr), \ \ \ \
\  \text{o} \ \ \ \mathbf{(i')} \ \ |f(x)|\geq 0, \ \ \ \ \forall
x\in\R^d\\
&\mathbf{(ii)} \ |h(t,x)|\le L\bigl(1+|x|^{2\lambda}\bigr), \ \ \
\text{o} \ \ \ \mathbf{(ii')} \ |h(t,x)|\geq 0, \ \ \ \forall
(t,x)\in [0,T]\times\R^d
\end{align}
para algunas constantes $L>0$ y $\lambda\geq 1.$ Si
$v:[0,T]\times\mathbb{R}^{d}\rightarrow\mathbb{R}$ es continua de
clase $\C^{1,2}\bigl([0,T)\times\R^d\bigr),$ solución del problema
de Cauchy
\begin{equation}\label{cauchy}
\begin{split}
\frac{\partial v}{\partial t}+L_{t}v&=h+cv, \ \ \ \ \ \ \text{en}
\ [0,T)\times\R^d\\
v(T,x)&=f(x), \ \ \ \ \ \ \ \ \ x\in\R^d
\end{split}
\end{equation}
y satisface la condición de crecimiento polinomial
\begin{equation}\label{feynman-poli}
\max_{t\in [0,T]}|v(t,x)|\le M\bigl(1+|x|^{2\mu}\bigr), \ \ \
x\in\R^d
\end{equation}
para algún $M>0, \ \mu\geq 1,$ entonces $v$ admite la
representación estocástica conocida como fórmula de
Feynman-K\u{a}c
\[v(t,x)=E\biggl[f(X_T^{t,x})\exp\biggl\{-\int_t^T c(r,X_r^{t,x})\,dr\biggr\}
-\int_t^T h(s,X_s^{t,x})\exp\biggl\{-\int_t^s
c(r,X_r^{t,x})\,dr\biggr\}\,ds\,\biggr]
\]
\end{teorema}
\begin{proof}
Para cada $(t,x)\in [0,T]\times\R^d$ notemos
$Z_s^{t,x}=e^{-\int_{t}^{s}c(r,X_{r}^{t,x})\,dr}, \ s\geq t.$
Debido a que las trayectorias de $(\X)_{s\geq t}$ son continuas,
las trayectorias de $(\Z)_{s\geq t}$ son diferenciables y
satisfacen
\begin{equation}\label{feynman1}
d\Z=-c\bigl(s,\X\bigr)\Z\,ds.
\end{equation}
Aplicando fórmula de Itô al proceso $(\X)_{s\geq t}$ con la
función $v$
\begin{equation}\label{feynman2}
d\bigl[v\bigl(s,\X\bigr)\bigr]=\bigl(\tfrac{\partial v}{\partial
t}+L_s v\bigr)\bigl(s,\X\bigr)\,ds+(\nabla_x
v\cdot\sigma)\bigl(s,\X\bigr)\,dW_s,
\end{equation}
y usando la regla del producto con (\ref{feynman1}) y
(\ref{feynman2}) obtenemos
\begin{align*}
d\bigl\{v\bigl(s,\X\bigr)\Z\bigr\}&=\bigl\{-v(s,\X)\cdot
c(s,\X)\Z+\bigl(\tfrac{\partial v}{\partial
t}+L_s v\bigr)(s,\X)\Z\bigr\}\,ds\\[0.2cm]
&\ \ \ +(\nabla_x v\cdot\sigma)\bigl(s,\X\bigr)\Z\,dW_s.
\end{align*}
Definimos la sucesión de tiempos de parada $\tau_n:=\inf\{s\geq t:
|\X|\geq n\}, \ n\geq 1.$ Integrando la anterior expresión entre
$t$ y $\tau_n\wedge T$
\begin{align*}
v\bigl(&\tau_n\wedge T,X_{\tau_n\wedge
T}^{t,x}\bigr)Z_{\tau_n\wedge
T}^{t,x}-v(t,x)\\[0.2cm]
&=\int_t^{\tau_n\wedge T}\bigl(-cv+\tfrac{\partial v}{\partial
t}+L_s v\bigr)(s,\X)\Z\,ds+\int_t^{\tau_n\wedge
T}(\nabla_x v\cdot\sigma)\bigl(s,\X\bigr)\Z\,dW_s\\
&=\int_t^{\tau_n\wedge T}h(s,\X)\Z\,ds+\int_t^{\tau_n\wedge
T}(\nabla_x v\cdot\sigma)\bigl(s,\X\bigr)\Z\,dW_s.
\end{align*}
Dado que si $t\le s<\tau_n$ entonces $|\X|\le n,$ y $Z_s^{t,x}\le
1,$ por la condición de crecimiento lineal sobre $\sigma$ se tiene
\[E\int_t^{\tau_n\wedge T}\bigl\|(\nabla_x
v\cdot\sigma)\bigl(s,X_s^{t,x}\bigr)\Z\bigr\|^2\,ds\le
K_T^2(1+n)^2 (T-t)M_n,\] donde $M_n=\max\bigl\{|\nabla_x
v(s,y)|^2:t\le s\le T, \ |y|\le n\bigr\}.$ Entonces el valor
esperado de la integral estocástica es cero y
\begin{align*}
v(t,x)&=E\bigl[v\bigl(\tau_n\wedge T,X_{\tau_n\wedge
T}^{t,x}\bigr)Z_{\tau_n\wedge T}^{t,x}\bigr]-E\int_t^{\tau_n\wedge
T}\!h(s,\X)\Z\,ds\\
&=E\bigl[v\bigl(\tau_n,X_{\tau_n}^{t,x}\bigr)Z_{\tau_n}^{t,x}\mathbf{1}_{\{\tau_n<T\}}\bigr]
+E\bigl[f\bigl(X_T^{t,x}\bigr)Z_T^{t,x}\mathbf{1}_{\{\tau_n\geq
T\}}\bigr]-E\int_t^{\tau_n\wedge T}\!h(s,\X)\Z\,ds,
\end{align*}
pues $v\bigl(T,X_T^{t,x}\bigr)=f(X_T^{t,x}).$ Dado que
\begin{align*}
\Bigl|E\bigl[v\bigl(\tau_n,X_{\tau_n}^{t,x}\bigr)Z_{\tau_n}^{t,x}%
\mathbf{1}_{\{\tau_n<T\}}\bigr]\Bigr|&\le
E\Bigl[\bigl|v\bigl(\tau_n,X_{\tau_n}^{t,x}\bigr)Z_{\tau_n}^{t,x}%
\mathbf{1}_{\{\tau_n<T\}}\bigr|\Bigr]\\&\le
E\Bigl[\bigl|v\bigl(\tau_n,X_{\tau_n}^{t,x}\bigr)\bigr|\mathbf{1}_{\{\tau_n<T\}}\Bigr]
\le M(1+n^{2\mu})\cdot\mathbf{P}(\tau_n<T),
\end{align*}
y por la desigualdad de Markov y el lema \ref{apriori}
\[\mathbf{P}(\tau_n<T)\le \mathbf{P}\Bigl(\max_{\ t\le s\le
T}|\X|\geq n\Bigr)\le \frac{1}{n^{2p}}E\biggl[\max_{t\le s\le
T}|\X|^{2p}\biggr]\le
\frac{Ce^{CT}}{n^{2p}}\bigl(1+|x|^{2p}\bigr)\] donde
$C=C(m,d,K,T),$ eligiendo $p>\mu$ obtenemos
\[\Bigl|E\bigl[v\bigl(\tau_n,X_{\tau_n}^{t,x}\bigr)Z_{\tau_n}^{t,x}%
\mathbf{1}_{\{\tau_n<T\}}\bigr]\Bigr|
\le\frac{MCe^{CT}}{n^{2p}}(1+n^{2\mu})\bigl(1+|x|^{2p}\bigr)
\xrightarrow[n\to\infty]\,0.\] Si $f$ satisface \textbf{(i)}
entonces
$\bigl|f\bigl(X_T^{t,x}\bigr)Z_T^{t,x}\mathbf{1}_{\{\tau_n\geq
T\}}\bigr|\le L\bigl(1+|X_T^{t,x}|^{2\mu}\bigr).$ Dado que
\[f\bigl(X_T^{t,x}\bigr)Z_T^{t,x}\mathbf{1}_{\{\tau_n\geq
T\}}\xrightarrow[n\to\infty]{c.s.}f\bigl(X_T^{t,x}\bigr)Z_T^{t,x}
\ \ \ \ \ \text{y} \ \ \ \ \
E\bigl(|X_T^{t,x}|^{2\mu}\bigr)<+\infty,\] por convergencia
dominada se sigue
\begin{equation}\label{feynman3}
E\bigl[f\bigl(X_T^{t,x}\bigr)Z_T^{t,x}\mathbf{1}_{\{\tau_n\geq
T\}}\bigr]\xrightarrow[n\to\infty]\,E\bigl[f\bigl(X_T^{t,x}\bigr)Z_T^{t,x}\bigr].
\end{equation}
Se obtiene lo mismo si $f$ satisface \textbf{(i')} aplicando
convergencia monótona a la sucesión creciente
$f\bigl(X_T^{t,x}\bigr)Z_T^{t,x}\mathbf{1}_{\{\tau_n\geq T\}}, \
n\geq 1.$ De manera análoga, si $h$ satisface \textbf{(ii)}
entonces
\[\Bigl|E\bigl[h(s,\X)\Z\mathbf{1}_{[0,\tau_n)}(s)\bigr]\Bigr|
\le L\bigl\{1+E\big(|\X|^{2\lambda}\bigr)\bigr\}\le
L\bigl\{1+Ce^{CT}\big(1+|x|^{2\lambda}\bigr)\bigr\}.\] Por la
desigualdad de Hölder, se tiene
\begin{align*}
\Bigl|E\bigl[h(s,\X)\Z\bigr]
-E\bigl[h(s,\X)\Z\mathbf{1}_{[0,\tau_n)}(s)\bigr]\Bigr|^2
&=\Bigl|E\bigl[h(s,\X)\Z\mathbf{1}_{[\tau_n,\infty)}(s)\bigr]\Bigr|^2\\
&\le E\bigl[\bigl(h(s,\X)\Z\bigr)^2\bigr]\cdot
E\bigl[\mathbf{1}_{[\tau_n,\infty)}(s)\bigr],
\end{align*}
y dado que $E\bigl[\mathbf{1}_{[\tau_n,\infty)}(s)\bigr]
=E\bigl[\mathbf{1}_{\{\tau_n\le s\}}\bigr]=\mathbf{P}(\tau_n\le
s)\xrightarrow[n\to\infty]\,0,$
\[E\bigl[h(s,\X)\Z\mathbf{1}_{[0,\tau_n)}(s)\bigr]\xrightarrow[n\to\infty]\,
E\bigl[h(s,\X)\Z\bigr], \ \ \ \ \forall s\in [t,T],\] y por
convergencia dominada
\[\int_t^T
E\bigl[h(s,\X)\Z\mathbf{1}_{[0,\tau_n)}(s)\bigr]\,ds\xrightarrow[n\to\infty]\,
\int_t^T E\bigl[h(s,\X)\Z\bigr]\,ds,\] es decir
\[E\int_t^{\tau_n\wedge T} h(s,\X)\Z\,ds\xrightarrow[n\to\infty]\,
E\int_t^T h(s,\X)\Z\,ds.\] Obtenemos el mismo resultado si $h$
satisface \textbf{(ii')} y aplicamos convergencia monótona a la
sucesión creciente $h(s,\X)\Z\mathbf{1}_{[0,\tau_n)}(s), \ n\geq
1.$ En cualquier caso, tomando límites se tiene
\[v(t,x)=E\bigl[f\bigl(X_T^{t,x}\bigr)Z_T^{t,x}\bigr]-E\int_t^T\!h(s,\X)\Z\,ds\]
lo cual demuestra el teorema.
\end{proof}
\begin{obse}
A diferencia de la ecuación de Kolmogorov, si
no se supone la existencia \textit{a priori} de una solución
$v\in\C^{1,2}\bigl([0,T)\times\R^d\bigr)$ del problema de Cauchy
(\ref{cauchy}), la función dada por la fórmula de Feynman-Kac no
es necesariamente de clase $\C^{1,2}.$

{} Aún así se puede ver que bajo ciertas condiciones el problema
de Cauchy posee solución. Por ejemplo, el siguiente conjunto de
hipótesis son suficientes para garantizar la existencia de una
solución de (\ref{cauchy}) que satisfaga la condición de
crecimiento polinomial (\ref{feynman-poli})\,:
\begin{enumerate}
 \item El operador diferencial parcial de segundo orden $L_t$ es \textit{uniformemente
 elíptico,} es decir, existe una constante
 positiva $\delta>0$ tal que
 \[\sum_{i,j=1}^d a_{ij}(t,x)y_i y_j\geq\delta|y|^2, \ \ \ \forall y\in\R^d, \
 \forall (t,x)\in [0,\infty)\times\R^d,\]
 \item Las funciones $a_{ij}(t,x), \ W_i(t,x)$ son acotadas
 y \textit{Lipschitz} sobre compactos de {}$[0,T]\times\R^d.$
 \item Las funciones $a_{ij}(t,x)$ y $h(t,x)$ son \textit{Hölder} continuas en $x$
 uniformemente con respecto a $t$ en $[0,T]\times\R^d.$
 \item La función $c(t,x)$ es acotada en $[0,T]\times\R^d$ y \textit{Hölder} continua en $x$
 uniformemente con respecto a $t$ sobre compactos de $[0,T]\times\R^d.$
 \item $f$ y $h$ satisfacen las condiciones de \textit{crecimiento
 polinomial} \textbf{(i)} y \textbf{(ii)} respectivamente.
\end{enumerate}
Ver \cite{friedman1,friedman2}.
\end{obse}

\subsection{El problema de Dirichlet} {} Sea $D$ un subconjunto abierto
de $\R^d,$ y asuma que tanto $b$ como $\sigma$ no dependen de $t.$
En ese caso notaremos por $(X_t^x)_{t\geq 0}$ la solución de la
EDE
\begin{equation}\label{dirichlet1}
X_t^x=x+\int_0^t b(X_s^x)\,ds+\int_0^t\sigma(X_s^x)\,dW_s, \ \ \ \
t\geq 0
\end{equation}
con condición inicial $X_0^x=x,$ y su generador infinitesimal será
el operador diferencial
\begin{equation}\label{dirichlet2}
(Lu)(x)=\frac{1}{2}\sum_{i,j=1}^d a_{ij}(x)\frac{\partial^2
u}{\partial x_i\partial x_j}(x)+ \sum_{i=1}^d W_i(x)\frac{\partial
u}{\partial x_i}(x), \ \ \ \ x\in\R^d
\end{equation}
para $u\in\C^2(\R^d).$ Diremos que $L$ es \textit{elíptico} en $D$
si
\[\sum_{i,j=1}^d a_{ij}(x)y_i y_j>0, \ \ \ \forall
y\in\R^d\setminus\{0\}, \
 \forall x\in D.
\]
Sea $L$ elíptico en un dominio abierto y acotado $D$ y sean
$c:\overline{D}\to[0,\infty), \ h:\overline{D}\to\R$ y $f:\partial
D\to\R$ funciones continuas. El \textit{problema de Dirichlet}
consiste en encontrar una función $u:\overline{D}\to\R$ continua
de clase $\C^2(D)$ que satisfaga la \textit{ecuación elíptica} con
\textit{condición de frontera}
\begin{equation}\label{dirichlet}
\begin{split}
Lu&=h+cu, \ \ \ \text{en} \ D\\
u(x)&=f(x), \ \ \ \ \ x\in\partial D.
\end{split}
\end{equation}
\begin{prop}
Sea $u\in\C^2(D)$ solución del problema de Dirichlet
(\ref{dirichlet}) en un dominio abierto y acotado $D.$ Para cada
$x\in\overline{D}$ sea $\tau_x:=\inf\{t\geq 0:X_t^x\notin D\}.$ Si
\begin{equation}\label{condtau}
E[\tau_x]<+\infty, \ \ \forall x\in\overline{D},
\end{equation}
entonces $u$ admite la representación estocástica
\[u(x)=E\biggl[f(X_{\tau_x}^x)\exp\biggl\{-\int_0^{\tau_x}c(X_r^x)\,dr\biggr\}
-\int_0^{\tau_x}h(X_s^x)\exp\biggl\{-\int_0^s
c(X_r^x)\,dr\biggr\}\,ds\,\biggr].\]
\end{prop}
\begin{proof}
Para cada $x\in D$ notaremos $Z_s^x=e^{-\int_{0}^{s}c(X_r^x)\,dr},
\ s\geq 0.$ Igual que en la demostración del teorema
\ref{feynmankac} se ve que el proceso $u(X_s^x)Z_s^x, \ s\geq 0,$
satisface
\begin{equation}\label{dirichlet3}
d\bigl\{u\bigl(X_s^x\bigr)Z_s^x\bigr\}=\bigl\{-u(X_s^x)\cdot
c(X_s^x)Z_s^x+(Lu)(X_s^x)Z_s^x\bigr\}\,ds+(\nabla_x
u\cdot\sigma)\bigl(X_s^x\bigr)Z_s^x\,dW_s.
\end{equation}
Sea $\{D_n\}_{n=1}^\infty$ una sucesión creciente de conjuntos
abiertos tales que $\overline{D}_n\subseteq D, \
\bigcup_{n=1}^\infty D_n=D$ y
\[\lim_{n\to\infty}\tau_n=\tau_x, \ \
\text{c.s.} \ \ \ \ \ \text{donde} \ \ \tau_n:=\inf\{t\geq 0:
X_t^x\notin D_n\}, \ \ n\geq 1.\] Integrando (\ref{dirichlet3})
entre $0$ y $t\wedge\tau_n,$
\begin{align*}
u\bigl(X_{t\wedge\tau_n}^x\bigr)Z_{t\wedge\tau_n}^x-u(x)&=
\int_0^{t\wedge\tau_n}(-cu+Lu)\bigl(X_s^x\bigr)\,ds+\int_0^{t\wedge\tau_n}(\nabla_x
u\cdot\sigma)\bigl(X_s^x\bigr)Z_s^x\,dW_s\\
&=\int_0^{t\wedge\tau_n}h\bigl(X_s^x\bigr)Z_s^x\,ds+\int_0^{t\wedge\tau_n}(\nabla_x
u\cdot\sigma)\bigl(X_s^x\bigr)Z_s^x\,dW_s, \ \ \ \ t\geq 0.
\end{align*}
Dado que si $0\le s<\tau_n$ entonces $|X_s^x|\in\overline{D}_n,$
por la condición de crecimiento lineal sobre $\sigma$ se tiene
\begin{align*}
E\int_0^{t\wedge\tau_n}\bigl\|(\nabla_x
u\cdot\sigma)\bigl(X_s^x\bigr)Z_s^x\bigr\|^2\,ds
&\le\max_{y\in\overline{D}_n}|\nabla_x u(y)|^2
\cdot 2tK_t^2\bigl\{1+E\bigl(|X_s^x|^2\bigr)\bigr\}\\
&\le \max_{y\in\overline{D}_n}|\nabla_x u(y)|^2 \cdot
2tK_t^2\bigl\{1+Ce^{Ct}(1+|x|^2)\bigr\}.
\end{align*}
Entonces el valor esperado de la integral estocástica es cero y
$u(x)=E[M_t^n]$ donde
\[M_t^n:=u\bigl(X_{t\wedge\tau_n}^x\bigr)Z_{t\wedge\tau_n}^x
-\int_0^{t\wedge\tau_n}\!h\bigl(X_s^x\bigr)Z_s^x\,ds, \ \ \ t\geq
0, \ \ \ n\geq1.
\] Debido a que $\tau_n\xrightarrow[n\to\infty]{c.s.}\tau_x,$ para
cada $t\geq 0$ se tiene $M_t^n\xrightarrow[n\to\infty]{c.s.}M_t,$
donde \[M_t:=u\bigl(X_{t\wedge\tau_x}^x\bigr)Z_{t\wedge\tau_x}^x
-\int_0^{t\wedge\tau_x}\!h\bigl(X_s^x\bigr)Z_s^x\,ds.\] Dado que
$|M_t^n|\le\max_{y\in\overline{D}}|u(y)|+t\cdot\max_{y\in\overline{D}}|h(y)|,$
para todo $n\geq 1,$ por convergencia dominada
$E[M_t^n]\xrightarrow[n\to\infty]\,E\bigl[M_t], \ \forall t\geq
0.$ Por último, usando la estimación
\[|M_t|\le\max_{y\in\overline{D}}|u(y)|+\tau_x\cdot\max_{y\in\overline{D}}|h(y)|,\]
la hipótesis $E[\tau_x]<+\infty$ y de nuevo convergencia dominada
se obtiene
\[E\bigl[M_t]\xrightarrow[t\to\infty]\,
E\bigl[u\bigl(X_{\tau_x}^x\bigr)Z_{\tau_x}^x\bigr]
-E\int_0^{\tau_x}\!h\bigl(X_s^x\bigr)Z_s^x\,ds=
E\bigl[f\bigl(X_{\tau_x}^x\bigr)Z_{\tau_x}^x\bigr]
-E\int_0^{\tau_x}\!h\bigl(X_s^x\bigr)Z_s^x\,ds.\] Como $u(x)$ no
depende ni de $t$ ni de $n,$ tomando el límite cuando $t\to\infty$
y $n\to\infty$ se concluye finalmente que
\[u(x)=E\bigl[f\bigl(X_{\tau_x}^x\bigr)Z_{\tau_x}^x\bigr]
-E\int_0^{\tau_x}\!h\bigl(X_s^x\bigr)Z_s^x\,ds\] lo cual prueba la
proposición.
\end{proof}

{}{} Es natural preguntarse, ¿cuándo se tiene la condición
(\ref{condtau})? La siguiente proposición da una condición
suficiente para (\ref{condtau}):
\begin{prop}
Suponga que para el dominio abierto y acotado $D$ se cumple para
algún $1\le i\le d,$
\begin{equation}\label{condtau2}
\min_{x\in\overline{D}}a_{ii}(x)>0.
\end{equation}
Entonces $E[\tau_x]<+\infty, \ \forall x\in D.$
\end{prop}
\begin{proof}
Sean $A:=\min_{x\in\overline{D}}a_{ii}(x), \
B:=\max_{x\in\bar{D}}|b(x)|, q:=\min_{x\in\bar{D}}x_i$ y
$\nu>2B/A,$ y considere la función \[h(x):=-\mu e^{\nu x_i}, \ \ \
x\in D, \ \mu>0.\] Esta función es de clase $\C^\infty(D)$ y
satisface
\[-(Lh)(x)=\mu e^{\nu x_i}\bigl\{\tfrac{1}{2}\nu^2a_{ii}(x)+\nu W_i(x)\bigr\}
\geq\tfrac{1}{2}A\mu\nu e^{\nu q}\left(\nu-\frac{2B}{A}\right), \
\ \ \ x\in D.\] Escogemos $\mu>0$ suficientemente grande tal que
$Lh\le -1$ en $D,$ y aplicando fórmula de Itô con la función $h$
obtenemos
\begin{align*}
h(X_{t\wedge\tau_x}^x)-h(x)&=\int_0^{t\wedge\tau_x}(Lh)(X_s^x)\,ds
+\int_0^{t\wedge\tau_x}(\nabla_x h\cdot\sigma)(X_s^x)\,dW_s\\
&\le -(t\wedge\tau_x)+\int_0^{t\wedge\tau_x}(\nabla_x
h\cdot\sigma)(X_s^x)\,dW_s
\end{align*}
La función $h$ y sus derivadas son acotadas en $\overline{D},$
luego
\[E\int_0^{t\wedge\tau_x}\bigl\|(\nabla_x h\cdot\sigma)(X_s^x)\bigr\|^2\,ds<+\infty,\]
y
\[E(t\wedge\tau_x)\le h(x)-E\bigl[h(X_{t\wedge\tau_x}^x)\bigr]
\le 2\max_{y\in\overline{D}}|h(y)|.\] Haciendo $t\to\infty$ se
obtiene el resultado.
\end{proof}
\begin{obse}
La condición (\ref{condtau2}) es más fuerte que
\textit{elipticidad} pero más débil que \textit{elipticidad
uniforme} sobre D. Se puede ver que bajo las hipótesis
\begin{enumerate}
  \item $L$ es uniformemente elíptica,
  \item los coeficientes $a_{ij}(x), \ W_i(x), \ c(x)$ y $g(x)$
  son Hölder continuos y
  \item todo $a\in\partial D$ tiene la \textit{propiedad de la
  esfera exterior}, es decir, existe una bola cerrada
  $B[a,\varepsilon]$ tal que $B[a,\varepsilon]\cap D=\varnothing,
  \ B[a,\varepsilon]\cap\partial D=\{a\}.$
\end{enumerate}
(además de la continuidad de $f$ sobre $\partial D$) existe una
función $u\in\C(\bar{D})\cap\C^2(D)$ (de hecho con derivadas
parciales en $D$ se segundo orden Hölder continuas) que es
solución el problema de Dirichlet (\ref{dirichlet}). Ver
\cite{friedman1,friedman2}.
\end{obse}

\end{document}